\documentclass{amsart}[12pt]
\usepackage{amsmath}
\usepackage{amsfonts}
\usepackage{amssymb}
\usepackage{color}
\usepackage{graphicx}
\usepackage{blkarray}
\usepackage{subfigure}
\usepackage[margin=1.2in]{geometry}

\usepackage[numbers]{natbib} \setcitestyle{open={},close={}}

\usepackage{tikz}
\usetikzlibrary{matrix,arrows,decorations.pathmorphing}

\usepackage{bm}

\allowdisplaybreaks[4]
\usepackage[all]{xy}
\newtheorem{theorem}{Theorem}[section]
\newtheorem{proposition}[theorem]{Proposition}
\newtheorem{lemma}[theorem]{Lemma}
\newtheorem{remark}[theorem]{Remark}

\newtheorem{definition}[theorem]{Definition}
\newtheorem{corollary}[theorem]{Corollary}

\newcommand{\be}{\begin{equation}}
\newcommand{\ee}{\end{equation}}
\newcommand{\bea}{\begin{eqnarray}}
\newcommand{\eea}{\end{eqnarray}}
\newcommand{\ben}{\begin{eqnarray*}}
	\newcommand{\een}{\end{eqnarray*}}

\begin{document}

\title{On the 1-leg Donaldson-Thomas $\mathbb{Z}_2\times\mathbb{Z}_2$-vertex}

\author[Yijie Lin]{Yijie Lin}
\address{School of Mathematics and Statistics\\Key Laboratory of Analytical Mathematics and Applications (Ministry of Education)\\Fujian Normal University\\350117 Fuzhou, P.R. China}
\email{yjlin12@163.com}

\maketitle

\begin{abstract}
We introduce a notion of restricted pyramid configurations for computing the 1-leg  Donaldson-Thomas $\mathbb{Z}_2\times\mathbb{Z}_2$-vertex. We study a special type of restricted pyramid configurations with the prescribed 1-leg partitions, and find one unique class of them satisfying the symmetric interlacing property. 
This leads us  to obtain an explicit formula for a class of 1-leg Donaldson-Thomas $\mathbb{Z}_2\times\mathbb{Z}_2$-vertex through establishing its connection with 1-leg Donaldson-Thomas $\mathbb{Z}_4$-vertex using the vertex operator methods of Okounkov-Reshetikhin-Vafa and Bryan-Young.
\end{abstract}




\tableofcontents

\section{Introduction}

The topological vertex arising from the duality relating  open string theory with Chern-Simons theory  has  attracted a lot of attention for computing invariants in enumerative geometry  [\cite{AKMV}]. The  mathematical theories for variant topological vertices  have been established in several branches of counting theories such as  (orbifold) Gromov-Witten theory and Donaldson-Thomas theory [\cite{LLLZ,LLZ2,MNOP,MOOP,ORV,BCY,Ross}]. This paper is concerned with the orbifold DT topological vertex developed by Bryan, Cadman and Young [\cite{BCY}]. They define the orbifold DT $G$-vertex $V^{G}_{\lambda\mu\nu}$, which is a generating function of 3D partitions asymptotic to a triple of partitions $(\lambda,\mu,\nu)$  colored by irreducible representations of a finite Abelian group $G$, see [\cite{BCY}, Equation (3)]. We will restrict ourselves to   $G\subseteq\mathrm{SO}(3)$. It is well known from the classification of polyhedral groups that a finite Abelian subgroup $G$ of $\mathrm{SO}(3)$ is either $\mathbb{Z}_n$ ($n\geq1$) or $\mathbb{Z}_2\times\mathbb{Z}_2$ up to isomorphism. The closed formula for  $V_{\emptyset\emptyset\emptyset}^{\mathbb{Z}_n}$  is obtained by using the vertex operator method, while the explicit formula for  $V^{\mathbb{Z}_2\times\mathbb{Z}_2}_{\emptyset\emptyset\emptyset}$ is derived by additionally applying some  properties of the 3D combinatorial objects called pyramid partitions  [\cite{BY}].  Moreover, an explicit formula for $V^{\mathbb{Z}_n}_{\lambda\mu\nu}$ with any triple of partitions  $(\lambda,\mu,\nu)$ has been presented in terms of Schur functions, see [\cite{BCY}, Theorem 12]. As pointed out in [\cite{BCY}, Section 5.1], it is natural to expect  an explicit formula for the full DT $\mathbb{Z}_2\times\mathbb{Z}_2$-vertex. But actually it is still  open  even for an explicit formula of the 1-leg DT $\mathbb{Z}_2\times\mathbb{Z}_2$-vertex, i.e., 
$V^{\mathbb{Z}_2\times\mathbb{Z}_2}_{\lambda\emptyset\emptyset}$, $V^{\mathbb{Z}_2\times\mathbb{Z}_2}_{\emptyset\mu\emptyset}$ and  
$V^{\mathbb{Z}_2\times\mathbb{Z}_2}_{\emptyset\emptyset\nu}$ with nonempty partitions  $\lambda,\mu,\nu$. 

We will consider the  1-leg case of the DT $\mathbb{Z}_2\times\mathbb{Z}_2$-vertex in this paper. Now let us briefly recall the ocurrence of pyramid partitions in enumerative invariants and  the proof of the explicit formula for $V^{\mathbb{Z}_2\times\mathbb{Z}_2}_{\emptyset\emptyset\emptyset}$ by Bryan-Young. The  non-commutative DT invariants of the conifold  first appear in Szendroi's definition of invariants counting framed cyclic $A$-modules [\cite{Szendroi}]. By torus localization, he conjectures some formulas for  generating functions of noncommutative DT invariants invloving the notion of pyramid partitions, which are proved by Young using the dimer shuffling algorithm [\cite{You}] and interpreted by Nagao-Nakajima as a result of wall-crossing phenomena [\cite{NN}]. One may refer to the non-commutative DT theory in [\cite{Szendroi,BCY,NN,N1,N2,N3,MR,KS,Zhou}] for more details. The notion of pyramid partitions is refined and described as the finite set of bricks by some algebraic definition [\cite{BY}, Section 5]. Let $\mathfrak{P}$ be the set of all pyramid partitions as mentioned above and $Z_\mathfrak{P}$ be the generating function of pyramid partitions, which is denoted by the notation $Z_{\mathrm{pyramid}}$ in [\cite{BY}, Definition 6.1].  Then the explicit formula for $V^{\mathbb{Z}_2\times\mathbb{Z}_2}_{\emptyset\emptyset\emptyset}$ is derived by Bryan-Young through the following steps:

$(\mathrm{I})$   $Z_\mathfrak{P}$ is computed  in two different ways.  One is counting pyramid partitions  along  their antidiagonal slices where we  denote $Z_\mathfrak{P}$ also by $Z_\mathfrak{P+}$  in this way. The other is to count pyramid partitions  along  their diagonal slices where we  write $Z_\mathfrak{P}$ also by $Z_\mathfrak{P-}$. Here  $Z_\mathfrak{P}=Z_\mathfrak{P+}=Z_\mathfrak{P-}$.

$(\mathrm{II})$ Compute $Z_\mathfrak{P-}$ directly by the vertex operator method of Okounkov-Reshetikhin-Vafa [\cite{ORV}].

$(\mathrm{III})$ Establish  a connection between $V^{\mathbb{Z}_2\times\mathbb{Z}_2}_{\emptyset\emptyset\emptyset}$ and $Z_\mathfrak{P+}$ using the vertex operator method as above together with introducing some new vertex operators such as $E_{\pm}(x)$ in [\cite{BY}, Definition 7.4].

Actually every pyramid partition has the symmetric interlacing property of type $\emptyset$ (see Remark \ref{interlacing-emptyset}), i.e., the same interlacing properties \eqref{diagonal-interlacing} and \eqref{antidiagonal-interlacing} along its  diagonal and antidiagonal slices respectively, which ensure that $Z_\mathfrak{P+}=Z_\mathfrak{P-}$ in Step $(\mathrm{I})$ by Remark \ref{characteristic-PP}. However, when it comes to computing the 1-leg DT $\mathbb{Z}_2\times\mathbb{Z}_2$-vertex, the notion of  pyramid partitions  is not suitable for describing  the interlacing property induced by the leg partitions. Therefore, we need a new notion of 3D combinatorial objects to describe the desired interlacing property.

We  introduce a notion of restricted pyramid configurations, which are constructed from  pyramid partitions by removing some finite bricks. We  construct the special class of restricted pyramid configurations of antidiagonal (resp. diagonal) type $(\nu,l)$ for any leg partition $\nu\neq\emptyset$ and $l\in\mathbb{Z}_{\geq0}$, denoted by $\mathfrak{RP}_{+}(\nu,l)$ (resp. $\mathfrak{RP}_{-}(\nu,l)$), see Section 3.1 and 3.3 for the definition. Let $\pi\in\mathfrak{P}$. Every restricted pyramid configuration $\widetilde{\bm{\wp}}_{\pi}(\nu,l)$   (resp. $\bm{\wp}_{\pi}(\nu,l)$) in $\mathfrak{RP}_{+}(\nu,l)$ (resp. $\mathfrak{RP}_{-}(\nu,l)$) can be shown to has the desired interlacing property along antidiagonal (resp. diagonal)  slices induced by  the 1-leg partition $\nu$, see Definition \ref{interlacing-second-type} and Lemma \ref{generalized-RPC-interlacing1}. 
Let $Z_{\mathfrak{RP}_{+}(\nu,l)}$ (resp. $Z_{\mathfrak{RP}_{-}(\nu,l)}$) be the generating function of  restricted pyramid configurations of antidiagonal (resp. diagonal) type $(\nu,l)$ where  both of them are shown to be  independent of $l$, see Remark \ref{Independent-of-l}. Then we are only concerned with   $l=0$, and take the simplified notations $\widetilde{\bm{\wp}}_{\pi}(\nu)=\widetilde{\bm{\wp}}_{\pi}(\nu,0)$,
 $\bm{\wp}_{\pi}(\nu)=\bm{\wp}_{\pi}(\nu,0)$, $\mathfrak{RP}_{+}(\nu)=\mathfrak{RP}_{+}(\nu,0)$, and $\mathfrak{RP}_{-}(\nu)=\mathfrak{RP}_{-}(\nu,0)$, see also Remark \ref{generalization of RPC notions}. 
 
With the above construction, one would naturally study the relation between  $Z_{\mathfrak{RP}_{+}(\nu)}$
 and $Z_{\mathfrak{RP}_{-}(\nu)}$ where $\mathfrak{RP}_{+}(\nu)=\{\widetilde{\bm{\wp}}_{\pi}(\nu)\;| \;\pi\in\mathfrak{P}\}$
 and $\mathfrak{RP}_{-}(\nu)=\{\bm{\wp}_{\pi}(\nu)\;| \;\pi\in\mathfrak{P}\}$. We first show and find the following equivalent condition (see Proposition \ref{unique symmetric interlacing})
  \ben
  \widetilde{\bm{\wp}}_{\pi}(\nu)=\bm{\wp}_{\pi}(\nu)\; \mbox{ for any   $\pi\in\mathfrak{P}$  if and only if $\nu=(m,m-1,\cdots,2,1)$ for some $m\in\mathbb{Z}_{\geq1}$}.
  \een
  This means that  
 $\mathfrak{RP}_{+}(\nu)$  $($or  $\mathfrak{RP}_{-}(\nu)$$)$ is the unique class of restricted pyramid configurations of antidiagonal $($or diagonal$)$   type $(\nu,0)$ with the symmetric interlacing property of type $\nu$ (see Definition \ref{dfn-symmetric-interlacing} and Lemma \ref{RPC-interlacing1}) only when $\nu=(m,m-1,\cdots,2,1)$ where  $m\in\mathbb{Z}_{\geq1}$ is the length of $\nu$. This uniqueness can be generalized to the case for any $l\in\mathbb{Z}_{\geq0}$, see Proposition \ref{generalized-unique symmetric interlacing} and Remark \ref{generalize-uniqueness}. With this special choice of 1-leg partitions $\nu$, we have 
$\mathfrak{RP}_{+}(\nu)=\mathfrak{RP}_{-}(\nu)$ and hence $Z_{\mathfrak{RP}_{+}(\nu)}=Z_{\mathfrak{RP}_{-}(\nu)}$, see Definition \ref{GF-RPC}.
One can refer to the discussion for the other choices of $\nu$  in  Section 5.

As in step $(\mathrm{III})$, one can relate 1-leg DT $\mathbb{Z}_2\times\mathbb{Z}_2$-vertex $V^{\mathbb{Z}_2\times\mathbb{Z}_2}_{\emptyset\emptyset\nu}$ with $Z_{\mathfrak{RP}_{+}(\nu)}$ through comparing their vertex operator product expressions, see Remark \ref{general-Z2Z2-RPC+}.
On the other hand, there is an observation that the generating function $Z_\mathfrak{P-}$ in Step $(\mathrm{II})$ can be alternatively computed by comparing to 1-leg DT $\mathbb{Z}_4$-vertex $V^{\mathbb{Z}_4}_{\emptyset\emptyset\emptyset}$ (with their varibles identified as in Section 4.1) due to their interlacing properties and color assignments, see Remark \ref{another-proof-GFPP}. Using the vertex operator methods of  Okounkov-Reshetikhin-Vafa and Bryan-Young, one can derive the explicit formula for  $V^{\mathbb{Z}_2\times\mathbb{Z}_2}_{\emptyset\emptyset\nu}$ where $\nu=(m,m-1,\cdots,2,1)$ with $m\in\mathbb{Z}_{\geq1}$   through the following steps:

$(\mathfrak{I}_1)$ Establish a connection between $V^{\mathbb{Z}_2\times\mathbb{Z}_2}_{\emptyset\emptyset\nu}$ and $Z_{\mathfrak{RP}_{+}(\nu)}$ by the vertex operator methods.

$(\mathfrak{I}_2)$ Establish a connection between $V^{\mathbb{Z}_4} _{\emptyset\emptyset\nu}$ and $Z_{\mathfrak{RP}_{-}(\nu)}$ by the vertex operator methods.

$(\mathfrak{I}_3)$ Establish a connection between $V^{\mathbb{Z}_2\times\mathbb{Z}_2}_{\emptyset\emptyset\nu}$ and $V^{\mathbb{Z}_4} _{\emptyset\emptyset\nu}$ using the equality $Z_{\mathfrak{RP}_{+}(\nu)}=Z_{\mathfrak{RP}_{-}(\nu)}$,  and then apply the explicit formula of Bryan-Cadman-Young in [\cite{BCY}, Theorem 12] for $V^{\mathbb{Z}_4} _{\emptyset\emptyset\nu}$.

To present  the explicit formula for 1-leg DT $\mathbb{Z}_2\times\mathbb{Z}_2$-vertex $V^{\mathbb{Z}_2\times\mathbb{Z}_2}_{\emptyset\emptyset\nu}$ when $\nu=(m,m-1,\cdots,2,1)$ with $m\in\mathbb{Z}_{\geq1}$, we introduce some notations.
For  $l\in\mathbb{Z}$ and  commutative variables $x,q$, let
\ben
&&\mathbf{M}(x,q)=\prod_{n=1}^\infty\frac{1}{(1-xq^n)^n},\\
&&\widetilde{\mathbf{M}}(x,q)=\mathbf{M}(x,q)\cdot\mathbf{M}(x^{-1},q),\\
&&\widehat{\mathbf{M}}(x,q)=\left(\widetilde{\mathbf{M}}(x,q)\cdot\widetilde{\mathbf{M}}(-x,q)\right)^{-1},\\ &&\widetilde{\mathbf{M}}_{0}(x,q;l)=\frac{\mathbf{M}(xq^l,q)}{\mathbf{M}(x,q)\cdot(qx^{-1};q)_{\infty}^l},\\
&&\widetilde{\mathbf{M}}_{1}(x,q;l)=\frac{\mathbf{M}(x^{-1}q^l,q)}{\mathbf{M}(x^{-1},q)\cdot(x;q)_{\infty}^l},\\
&&\widehat{\mathbf{M}}_0(x,q;l)=\widetilde{\mathbf{M}}_{0}(x,q;l)\cdot\widetilde{\mathbf{M}}_{0}(-x,q;l),\\
&&\widehat{\mathbf{M}}_1(x,q;l)=\widetilde{\mathbf{M}}_{1}(x,q;l)\cdot\widetilde{\mathbf{M}}_{1}(-x,q;l),
\een
where the first two functions are called  MacMahon and MacMahon tilde functions  in [\cite{BY}, Definition 1.1], and for any variable $a$
we take the following standard $q$-series 
\ben
(a;q)_{\infty}=\prod_{k=0}^\infty(1-aq^{k}).
\een
Although two notations $\widetilde{\mathbf{M}}_{0}(x,q;l)$ and $\widetilde{\mathbf{M}}_{1}(x,q;l)$ satisfy
\ben
\widetilde{\mathbf{M}}_{1}(qx^{-1},q;l+1)=\widetilde{\mathbf{M}}_{0}(x,q;l)\cdot\frac{(x;q)_\infty}{(qx^{-1};q)_\infty},
\een
we will use them simultaneously for unifying the presentation of formulas below. For any $p\in\mathbb{R}$,  let $\lfloor p\rfloor$ be the greatest integer less than or equal to $p$ and  $\{p\}=x-\lfloor x\rfloor$ be the fractional part of  $p$. 

We first derive a connection between $V^{\mathbb{Z}_2\times\mathbb{Z}_2}_{\emptyset\emptyset\nu}$ and $Z_{\mathfrak{RP}_{+}(\nu)}$ for Step $(\mathfrak{I}_1)$, see Lemma \ref{Z2Z2-RPC+}, and then obtain a connection between $V^{\mathbb{Z}_4} _{\emptyset\emptyset\nu}$ and $Z_{\mathfrak{RP}_{-}(\nu)}$ for Step $(\mathfrak{I}_2)$, see Lemma \ref{Z4-RPC-}.
	  Then we  have the following connection between 1-leg DT $\mathbb{Z}_{2}\times\mathbb{Z}_{2}$-vertex and 1-leg DT $\mathbb{Z}_{4}$-vertex for Step $(\mathfrak{I}_3)$.
\begin{theorem}(see Theorem \ref{Z2Z2-Z4})
	Assume  $\nu=(m,m-1,\cdots,2,1)$ with $m\in\mathbb{Z}_{\geq1}$.  Let $q=q_aq_bq_cq_0$. Then we have \\
	(i) if $m\equiv0\; \mbox{(mod 4)}$ or $m\equiv3\; \mbox{(mod 4)}$,
	\ben
	&&V_{\nu\emptyset\emptyset}^{\mathbb{Z}_{2}\times\mathbb{Z}_{2}}(q_0,q_c,q_a,q_b)=
	V^{\mathbb{Z}_4}_{\emptyset\emptyset\nu}(q_0,q_a,q_b,q_c)\cdot\Phi(q_c,q_a,q_b,q;m),\\
	&&V_{\emptyset\nu\emptyset}^{\mathbb{Z}_{2}\times\mathbb{Z}_{2}}(q_0,q_b,q_c,q_a)= V^{\mathbb{Z}_4}_{\emptyset\emptyset\nu}(q_0,q_c,q_a,q_b)\cdot\Phi(q_b,q_c,q_a,q;m),\\
	&&V_{\emptyset\emptyset\nu}^{\mathbb{Z}_{2}\times\mathbb{Z}_{2}}(q_0,q_a,q_b,q_c)
	= V^{\mathbb{Z}_4}_{\emptyset\emptyset\nu}(q_0,q_b,q_c,q_a)
	\cdot\Phi(q_a,q_b,q_c,q;m),
	\een	
	(ii) if $m\equiv1\; \mbox{(mod 4)}$ or $m\equiv2\; \mbox{(mod 4)}$,
	\ben
	&&V_{\nu\emptyset\emptyset}^{\mathbb{Z}_{2}\times\mathbb{Z}_{2}}(q_0,q_c,q_a,q_b)=V^{\mathbb{Z}_4}_{\emptyset\emptyset\nu}(q_b,q_a,q_0,q_c)\cdot\Phi(q_c,q_a,q_b,q;m),\\
	&&V_{\emptyset\nu\emptyset}^{\mathbb{Z}_{2}\times\mathbb{Z}_{2}}(q_0,q_b,q_c,q_a)=V^{\mathbb{Z}_4}_{\emptyset\emptyset\nu}(q_a,q_c,q_0,q_b)\cdot\Phi(q_b,q_c,q_a,q;m),\\
	&&V_{\emptyset\emptyset\nu}^{\mathbb{Z}_{2}\times\mathbb{Z}_{2}}(q_0,q_a,q_b,q_c)=V^{\mathbb{Z}_4}_{\emptyset\emptyset\nu}(q_c,q_b,q_0,q_a)\cdot\Phi(q_a,q_b,q_c,q;m),
	\een	
		where 
		\ben
		&&\Phi(q_a,q_b,q_c,q;m)\\
		&=&\frac{\widehat{\mathbf{M}}(q_a,q)\cdot\widehat{\mathbf{M}}(q_b,q)\cdot\widehat{\mathbf{M}}(q_c,q)\cdot\widehat{\mathbf{M}}(q_aq_bq_c,q)\cdot \widetilde{\mathbf{M}}(q_aq_b,q)\cdot\widetilde{\mathbf{M}}_{2\{\frac{m}{2}\}}(q_aq_b,q;2\lfloor \frac{m+1}{2}\rfloor)}{\widehat{\mathbf{M}}_{2\{\frac{m}{2}\}}(q_a,q;\lfloor \frac{m+1}{2}\rfloor)\cdot\widehat{\mathbf{M}}_{2\{\frac{m}{2}\}}(q_b,q;\lfloor \frac{m+1}{2}\rfloor)\cdot\widehat{\mathbf{M}}_{1-2\{\frac{m}{2}\}}(q_c,q;\lfloor \frac{m+1}{2}\rfloor)\cdot\widehat{\mathbf{M}}_{2\{\frac{m}{2}\}}(q_aq_bq_c,q;\lfloor \frac{m+1}{2}\rfloor)}.
		\een
\end{theorem}

Combined with the explicit formula of Bryan-Cadman-Young for the DT $\mathbb{Z}_4$-vertex, we have the following explicit formula for 1-leg DT  $\mathbb{Z}_{2}\times\mathbb{Z}_{2}$-vertex.
\begin{theorem}(see Theorem \ref{Z2Z2-1-leg})
	Assume  $\nu=(m,m-1,\cdots,2,1)$ with $m\in\mathbb{Z}_{\geq1}$. Let $q=q_aq_bq_cq_0$.  Then we have
	\ben
	&&V_{\nu\emptyset\emptyset}^{\mathbb{Z}_{2}\times\mathbb{Z}_{2}}(q_0,q_c,q_a,q_b)=V_{\emptyset\emptyset\emptyset}^{\mathbb{Z}_{2}\times\mathbb{Z}_{2}}(q_0,q_c,q_a,q_b)\cdot\Upsilon(q_c,q_a,q_b,q;m),\\
	&&V_{\emptyset\nu\emptyset}^{\mathbb{Z}_{2}\times\mathbb{Z}_{2}}(q_0,q_b,q_c,q_a)=V_{\emptyset\emptyset\emptyset}^{\mathbb{Z}_{2}\times\mathbb{Z}_{2}}(q_0,q_b,q_c,q_a)\cdot\Upsilon(q_b,q_c,q_a,q;m),\\ 
	&&V_{\emptyset\emptyset\nu}^{\mathbb{Z}_{2}\times\mathbb{Z}_{2}}(q_0,q_a,q_b,q_c)=V_{\emptyset\emptyset\emptyset}^{\mathbb{Z}_{2}\times\mathbb{Z}_{2}}(q_0,q_a,q_b,q_c)\cdot\Upsilon(q_a,q_b,q_c,q;m),
	\een	
	where 
\ben
&&\Upsilon(q_a,q_b,q_c,q;m)\\
&=& \frac{\widetilde{\mathbf{M}}_{2\{\frac{m}{2}\}}\left(q_aq_b,q;2\lfloor\frac{m+1}{2}\rfloor\right)}{\widetilde{\mathbf{M}}_{2\{\frac{m}{2}\}}(-q_a,q;\lfloor\frac{m+1}{2}\rfloor)\cdot\widetilde{\mathbf{M}}_{2\lbrace\frac{m}{2}\rbrace}(-q_b,q;\lfloor\frac{m+1}{2}\rfloor)\cdot\widetilde{\mathbf{M}}_{1-2\{\frac{m}{2}\}}(-q_c,q;\lfloor\frac{m+1}{2}\rfloor)\cdot\widetilde{\mathbf{M}}_{2\{\frac{m}{2}\}}(-q_aq_bq_c,q;\lfloor\frac{m+1}{2}\rfloor)}
\een	
 and the explicit formula of Bryan-Young for  $V_{\emptyset\emptyset\emptyset}^{\mathbb{Z}_{2}\times\mathbb{Z}_{2}}(q_0,q_a,q_b,q_c)$ is presented in Theorem \ref{Z2Z2-nolegs}.
\end{theorem}

It is worth mentioning that we also  have the explicit formula for  $Z_{\mathfrak{RP}_{+}(\nu)}=Z_{\mathfrak{RP}_{-}(\nu)}$ relating to $Z_{\mathfrak{P}}$ when $\nu=(m,m-1,\cdots,2,1)$ with $m\in\mathbb{Z}_{\geq1}$, see Corollary \ref{GF-RPC-withleg}. So far we have applied a class of restricted pyramid configurations with symmetric interlacing property to derive the explicit formula for a class of 1-leg  DT  $\mathbb{Z}_{2}\times\mathbb{Z}_{2}$-vertex. Then it remains to consider  1-leg  DT  $\mathbb{Z}_{2}\times\mathbb{Z}_{2}$-vertex through the study of  restricted pyramid configurations with non-symmetric interlacing property. We will investigate this issue in the future, see Section 5 for the brief  discussion.

This paper is organized as follows. In Section 2, we recollect vertex operators and pyramid partitions together with their relevant properties, which are important for computing the DT $\mathbb{Z}_2\times\mathbb{Z}_2$-vertex  and $\mathbb{Z}_n$-vertex. We also recall the known explicit formulas for these two types of orbifold DT vertices. In Section 3, we introduce a notion of restricted pyramid configurations, and study a special type of restricted pyramid configurations which is suitable for computing the 1-leg DT $\mathbb{Z}_2\times\mathbb{Z}_2$-vertex. And we find a unique class of restricted pyramid configurations with the symmetric interlacing property along their diagonal and antidiagonal slices induced by some particular 1-leg partitions mentioned above. In Section 4, 
with this special 1-leg partitions, we establish a connection between 1-leg DT $\mathbb{Z}_2\times\mathbb{Z}_2$-vertex  and 1-leg DT $\mathbb{Z}_4$-vertex to obtain an explicit formula for this class of 1-leg DT $\mathbb{Z}_2\times\mathbb{Z}_2$-vertex. In Section 5, we discuss some further direction for  computing the 1-leg DT $\mathbb{Z}_2\times\mathbb{Z}_2$-vertex with other leg partitions.


\section{Preliminaries}
In this section, we will first recall vertex operators and their properties, which are important in formulating  the orbifold topological vertex. Then we  recollect  pyramid partitions and their relevant results  for the further definition  of restricted pyramid configurations. And we recall the known explicit formulas for the orbifold
DT $\mathbb{Z}_2\times\mathbb{Z}_2$-vertex  and $\mathbb{Z}_n$-vertex used in Section 4.

\subsection{Vertex operators}
We briefly recall the relevant formulas for vertex operators in [\cite{BY}] in this subsection as follows. See also [\cite{Okounkov}, Appendix A] for more details.

Let $\lambda\subset(\mathbb{Z}_{\geq0})^2$ be a partition (or a 2D Young diagram) with its conjugate denoted by $\lambda^\prime$. Let $\lambda=(\lambda_0,\lambda_1,\cdots)$ and $\mu=(\mu_0,\mu_1,\cdots)$ be two partitions. We call $\lambda$ interlaces with $\mu$, denoted by  $\lambda\succ\mu$ or $\mu\prec\lambda$, if the row lengths $\lambda_i$, $\mu_i$ satisfy
\ben
\lambda_0\geq\mu_0\geq\lambda_1\geq\mu_1\geq\cdots.
\een
Let  $\lambda\succ^\prime\mu$ or $\mu\prec^\prime\lambda$ denote the interlacing relation  $\lambda^\prime\succ\mu^\prime$  if the column lengths $\lambda_i^\prime$, $\mu_i^\prime$ satisfy
\ben
\lambda_0^\prime\geq\mu_0^\prime\geq\lambda_1^\prime\geq\mu_1^\prime\geq\cdots.
\een

Any 3D Young diagram (or plane partition) has the following important property.
\begin{lemma}([\cite{BY}, Definition 6])\label{interlacing}
	The following are equivalent:\\
	$(i)$ $\lambda\succ\mu$;\\
	$(ii)$ $\lambda_i^\prime-\mu_i^\prime=0\mbox{ or }1$, for any $i\geq1$\\
	$(iii)$ $\lambda$ and $\mu$ are two adjacent diagonal slices of some 3D Young diagram.
\end{lemma}

Let $\mathcal{Y}$ be the set of all partitions. First, we have the  $\Gamma_{\pm}$-operators acting on elements of $\mathcal{Y}$.

\begin{definition}([\cite{BY}, Section 3])
Let $x$ be any variable. Then 
\ben
&&\Gamma_{+}(x)\lambda=\sum_{\mu\prec\lambda}x^{|\lambda|-|\mu|}\mu;\;\;\;\;\; \Gamma_{-}(x)\lambda=\sum_{\lambda\prec\mu}x^{|\mu|-|\lambda|}\mu;\\
&&\Gamma_{+}^\prime(x)\lambda=\sum_{\mu\prec^\prime\lambda}x^{|\lambda|-|\mu|}\mu;\;\;\;\;\;   \Gamma_{-}^\prime(x)\lambda=\sum_{\lambda\prec^\prime\mu}x^{|\mu|-|\lambda|}\mu.
\een
\end{definition}
They satisfy  the following commutation relations.
\begin{lemma}([\cite{BY}, Lemma 3.3])
Let $x$ and $y$ be two commuting variables, then we have 
\ben
&&\left[\Gamma_{+}(x),\Gamma^\prime_{-}(y)\right]=1+xy,\;\;\;\;\; \left[\Gamma^\prime_{+}(x),\Gamma_{-}(y)\right]=1+xy;\\
&&\left[\Gamma_{+}(x),\Gamma_{-}(y)\right]=\frac{1}{1-xy},\;\;\;\;\; \left[\Gamma^\prime_{+}(x),\Gamma^\prime_{-}(y)\right]=\frac{1}{1-xy}.
\een
\end{lemma}
Next, there are two type of  weight operators acting on elments of $\mathcal{Y}$.
\begin{definition}([\cite{BY}, Definitions 3.4 and 7.1])
	For $g\in\mathbb{Z}_n$ and  $h, h_1,h_2\in\mathbb{Z}_2\times\mathbb{Z}_2$, define
	\ben
	&&Q_{g}\lambda=\tilde{q}_{g}^{|\lambda|}\lambda,\;\;\;\;\;\; Q_{h}\lambda=q_{h}^{|\lambda|}\lambda,\\
	&&Q_{h_1h_2}\lambda=q_{h_1}^{|\{(i,j)\in\lambda\,|\,i\equiv j(\mathrm{mod}\; 2)\}|}\cdot q_{h_2}^{|\{(i,j)\in\lambda\,|\,i\not\equiv j(\mathrm{mod}\; 2)\}|}\cdot\lambda,
	\een
	where $q_h, q_{h_1},q_{h_2}$  and $\tilde{q}_{g}$ are variables corresponding to colors $h, h_1,h_2\in\mathbb{Z}_2\times\mathbb{Z}_2$ and  $g\in\mathbb{Z}_n$.
\end{definition}

Let $\alpha_{-n}$ ($n>0$) be the operator acting on a 2D Young diagram $\eta$ by adding any possible border strip of length $n$ onto  $\eta$ with  sign $(-1)^{\widehat{h}+1}$. Here, $\widehat{h}$ is the height of the border strip. Similarly the operator $\alpha_{n}$ ($n>0$) is defined  by deleting border strips with signs, which is adjoint to $\alpha_{-n}$. See [\cite{BY}, Section 2] for more details. Now we have the following two new vertex operators.
\begin{definition}([\cite{BY}, Definition 7.4])
Let $x$ be any variable.
Define
\ben
E_{\pm}(x)=\exp\left(\sum_{k\geq1}\frac{x^{2k}}{k}\alpha_{\pm2k}\right).
\een
\end{definition}
It follows easily from the definition  of $\alpha_{\pm n}$ that
\bea
&&\label{E-left}\bigg\langle\emptyset\,\bigg|E_{-}(x)=\bigg\langle\emptyset\,\bigg|,\\
&&\label{E+right}E_{+}(x)\bigg|\,\emptyset\bigg\rangle=\bigg|\,\emptyset\bigg\rangle.
\eea
And we  have the following important properties for $E_{\pm}$-operators.
\begin{lemma}([\cite{BY}, Lemma 7.5]) \label{vertex-exchange1}
Let $x$ and $y$ be two commuting variables, then we have 
\ben
&&\Gamma_{\pm}(x)=\Gamma_{\pm}^\prime(x)E_{\pm}(x),\\
&&\left[E_{\pm}(x),\Gamma_{\pm}(y)\right]=0,\\
&&\left[E_{\pm}(x),\Gamma_{\pm}^\prime(y)\right]=0,\\
&&E_{+}(x)\Gamma_{-}(y)=\frac{1}{1-(xy)^2}\Gamma_{-}(y)E_{+}(x),\\
&&\Gamma_{+}(x)E_{-}(y)=\frac{1}{1-(xy)^2}E_{-}(y)\Gamma_{+}(x),\\
&&\Gamma^\prime_{+}(x)E_{-}(y)=(1-(xy)^2)E_{-}(y)\Gamma^\prime_{+}(x).
\een
\end{lemma}

\begin{lemma}([\cite{BY}, Lemma 7.6]) \label{vertex-exchange2}
Let $g,h\in\mathbb{Z}_2\times\mathbb{Z}_2$ and $x$ be any variable. Then we have 
\ben
&&E_{-}(x\sqrt{q_gq_h})Q_{gh}=Q_{gh}E_{-}(x);\\
&&Q_{gh}E_{+}(x\sqrt{q_gq_h})=E_{+}(x)Q_{gh}.
\een
\end{lemma}
By the similar argument in the proof of [\cite{BY}, Lemma 7.6], we have
\begin{lemma}\label{vertex-exchange3}
	Let $g\in\mathbb{Z}_n$ and  $x$ be any variable. Then we have 
	\ben
	&&E_{-}(x\tilde{q}_g)Q_{g}=Q_{g}E_{-}(x);\\
	&&Q_{g}E_{+}(x\tilde{q}_g)=E_{+}(x)Q_{g}.
	\een
\end{lemma}

\subsection{Pyramid partitions}

We will follow $[\cite{BY}]$ to give the definition of pyramid partitions and recollect some known results.
Let $\mathbb{Z}_2\times\mathbb{Z}_2=\{0,a,b,c\}$. Let $P$ be a quiver with  vertices  $\{0,a,b,c\}$ and edges labelled by $\{v_1,w_1,v_2,w_2\}$ with directions shown as follows:

\begin{displaymath}
\xymatrix@=2cm{
	*+[F]{0} \ar@<.5ex>[r]^{v_1} \ar@<.5ex>[d]^{v_2} & *+[F]{a} \ar@<.5ex>[l]^{w_1} \ar@<.5ex>[d]^{w_2} \\
	*+[F]{b} \ar@<.5ex>[r]^{w_1} \ar@<.5ex>[u]^{w_2} & *+[F]{c} \ar@<.5ex>[l]^{v_1} \ar@<.5ex>[u]^{v_2}	
}
\end{displaymath}

\begin{definition}([\cite{BY}, Section 5])\label{pyramidpartition}
$(i)$ A word in $P$ is defined to be the concatenation of the  edge labels of some directed path in $P$, with the starting vertex of this path as the base. \\
$(ii)$ If $B_1, B_2$ are two words in $P$, define the multiplication  $B_1B_2$ by concatenating $B_1$ and $B_2$ if the result is  a well-defined word in $P$, otherwise $B_1B_2$ is defined to be zero. Then the set $\mathbb{C}P$ spanned by all words in $P$ forms a path algebra.\\
$(iii)$ Define the noncommutative quotient ring $A=\mathbb{C}P/\mathcal{I}$ where
\ben
\mathcal{I}=\langle \upsilon_1\omega_i\upsilon_2-\upsilon_2\omega_i\upsilon_1, \omega_1\upsilon_j\omega_2-\omega_2\upsilon_j\omega_1\rangle,\;\;\;\;\; i,j\in\{1,2\}.
\een
A brick is an element $[B]$ of $\mathbb{C}P/\mathcal{I}$, where $B$ is a word in $\mathbb{C}P$ based at the vertex $0$. Denote by $\mathfrak{B}$ the set of all bricks.\\
(iv) A pyramid partition $\pi$ is a finite subset of $\mathfrak{B}$ such that if $[B]\in\pi$, then every prefix of $B$ also represents a brick in $\pi$. Denote by $\mathfrak{P}$ the set of all pyramid partitions.  
\end{definition}

In [\cite{BY}, Section 5], the edge labels are interpreted as  vectors in $\mathbb{Z}^3$ by setting 
\ben
\upsilon_1=(-1,0,1),\;\;\upsilon_2=(1,0,1),\;\;\omega_1=(0,-1,1),\;\;\omega_2=(0,1,1).
\een
Then the position of a brick $[B]$ in $\mathbb{Z}^3$ is defined to be the sum of vectors corresponding to the edge labels in $[B]$. Here, the position of the brick corresponding to the empty word $[\,]$ is defined as  $(0,0,0)$. One can refer to [\cite{BY}, Figure 4] for some visualized bricks and pyramid partitions.

The color of any brick is given by
\begin{definition}([\cite{BY}, Definition 5.5])\label{color-brick-RPC}
The coloring on $\mathfrak{B}$  is define as the map 
\ben
K_{\mathfrak{P}}:\mathfrak{B}\to \mathbb{Z}_2\times\mathbb{Z}_2
\een
by setting $K_{\mathfrak{P}}([B])$ to be the final vertex of any path whose word is $B$. We call $K_{\mathfrak{P}}([B])$ the color of the brick $[B]$. 
\end{definition}
The bricks have the following relations.
\begin{lemma}\label{brick-relation}
Let $B_1$ and $B_2$ be two arbitrary words in $P$, then we have 
\ben
&&[B_1(\upsilon_1\omega_i\upsilon_2)B_2]=[B_1(\upsilon_2\omega_i\upsilon_1)B_2], \\
&& [B_1(\omega_1\upsilon_j\omega_2)B_2]=[B_1(\omega_2\upsilon_j\omega_1)B_2],\\
&& [B_1(\upsilon_i\omega_j)^s(\upsilon_k\omega_l)^tB_2]=[B_1(\upsilon_k\omega_l)^t(\upsilon_i\omega_j)^sB_2],\\
&&[B_1(\upsilon_i\omega_j)^s(\upsilon_j\omega_i)^sB_2]=[B_1(\upsilon_i\omega_i)^s(\upsilon_j\omega_j)^sB_2], \\
&&[B_1\upsilon_n(\omega_i\upsilon_j)^s(\omega_k\upsilon_l)^t)B_2]=[B_1(\upsilon_j\omega_i)^s(\upsilon_l\omega_k)^t\upsilon_nB_2],
\een
where $1\leq i,j,k,l,n\leq2$ and  $s,t\geq0$.
\end{lemma}
\begin{proof}
The first two relations follow from Definition \ref{pyramidpartition} $(iii)$, and the other equations follow from repeated applications of the first two relations.
\end{proof}

For a pyramid partition $\pi$, let $\pi_k$ be  the $k$-th diagonal slice of $\pi$, i.e. the set of all bricks in $\pi$ whose position $(x,y,z)$ satisfying $x-y=k$. 

\begin{lemma}([\cite{BY}, Lemmas 5.8 and 5.9])\label{interlacing1}
For $k\geq0$,  we have 
\ben
&&\pi_{-2k}=\{[(\upsilon_{1}\omega_2)^kW]\}\cap \pi,\\
&&\pi_{2k}=\{[(\upsilon_{2}\omega_1)^kW]\}\cap \pi,
\een 
where $W$ runs over all words in $\upsilon_{1}\omega_{1}$ and $\upsilon_{2}\omega_{2}$, and 
\ben
&&\pi_{-2k-1}=\{[(\upsilon_{1}\omega_2)^k\upsilon_1W^\prime]\}\cap \pi,\\
&&\pi_{2k+1}=\{[(\upsilon_{2}\omega_1)^k\upsilon_2W^\prime]\}\cap \pi,
\een 
where $W^\prime$ runs over all words in $\omega_1\upsilon_1$ and $\omega_2\upsilon_2$. Moreover, the bricks of $\pi_k$ ($k\in\mathbb{Z}$) form a 2D Young diagram and its slices are single-colored as follows
\ben
K_{\mathfrak{P}}(\pi_k)=\left\{
\begin{aligned}
	&0 ,\;\;\;\;\;\; \mbox{if $k=0$ $($\mbox{mod} 4$)$}, \\
	&b,  \;\;\;\;\;\;\mbox{if $k=1$ $($\mbox{mod} 4$)$},\\
	&c , \;\;\;\;\;\; \mbox{if $k=2$ $($\mbox{mod} 4$)$}, \\
	&a,   \;\;\;\;\;\;\mbox{if $k=3$ $($\mbox{mod} 4$)$}.
\end{aligned}
\right.
\een

Moreover, for $k\geq0$ we have the following interlacing properties 
\ben
\pi_{2k}\succ\pi_{2k+1},\;\;\pi^\prime_{-2k}\succ\pi^\prime_{-2k-1},\;\;\pi_{2k+1}^\prime\succ\pi_{2k+2}^\prime,\;\;\pi_{-2k-1}\succ\pi_{-2k-2}
\een
where $\pi_{\pm k}$ is identified with  its corresponding 2D Young diagram (or partition) and $\pi_{\pm k}^\prime$ denotes the transposition of 2D Young diagram $\pi_{\pm k}$.
\end{lemma}

\begin{remark}\label{brick-box1}
The  interlacing properties for a pyramid partition $\pi$ in Lemma \ref{interlacing1} are equivalent to 
\bea\label{diagonal-interlacing}
\cdots\prec\pi_{-5}\prec^\prime\pi_{-4}\prec\pi_{-3}\prec^\prime\pi_{-2}\prec\pi_{-1}\prec^\prime\pi_0\succ\pi_1\succ^\prime\pi_2\succ\pi_3\succ^\prime\pi_4\succ\pi_5\succ^\prime\cdots.
\eea
In the proof of [\cite{BY}, Lemma 5.8], we have the following correspondence between bricks of $\pi_{\pm k}$ and boxes of 2D Young diagram $\pi_{\pm k}$ for $k\in\mathbb{Z}_{\geq0}$:\\
$(i)$ the brick $[(\upsilon_{1}\omega_2)^k(\upsilon_1\omega_1)^j(\upsilon_2\omega_2)^i]$ in $\pi_{-2k}$ $\longleftrightarrow$ $(i,j)$-box in 2D Young diagram $\pi_{-2k}$;\\
$(ii)$ the brick $[(\upsilon_{2}\omega_1)^k(\upsilon_1\omega_1)^j(\upsilon_2\omega_2)^i]$ in $\pi_{2k}$ $\longleftrightarrow$ $(i,j)$-box in 2D Young diagram $\pi_{2k}$;\\
$(iii)$ the brick $[(\upsilon_{1}\omega_2)^k\upsilon_1(\omega_1\upsilon_1)^j(\omega_2\upsilon_2)^i]$ in $\pi_{-2k-1}$ $\longleftrightarrow$ $(i,j)$-box in 2D Young diagram $\pi_{-2k-1}$;\\
$(iv)$ the brick $[(\upsilon_{2}\omega_1)^k\upsilon_2(\omega_1\upsilon_1)^j(\omega_2\upsilon_2)^i]$ in $\pi_{2k+1}$ $\longleftrightarrow$ $(i,j)$-box in 2D Young diagram $\pi_{2k+1}$,\\
where $\upsilon_1\omega_1$ and $\upsilon_2\omega_2$ (resp. $\omega_1\upsilon_1$ and $\omega_2\upsilon_2$) commute in $\mathbb{C}P/\mathcal{I}$.
\end{remark}

 For a pyramid partition $\pi$, let $\widetilde{\pi}_k$ be  the $k$-th antidiagonal slice of $\pi$, i.e. the set of all bricks in $\pi$ whose position $(x,y,z)$ satisfying $x+y=k$. 
 Similarly, it is implicitly shown in [\cite{BY}, Lemma 7.3] that
 \begin{lemma}\label{interlacing2}
 For $k\geq0$, 
\ben
&&\widetilde{\pi}_{-2k}=\{[(\upsilon_{1}\omega_1)^k\widetilde{W}]\}\cap \pi,\\
&&\widetilde{\pi}_{2k}=\{[(\upsilon_{2}\omega_2)^k\widetilde{W}]\}\cap \pi,
\een 
where $\widetilde{W}$ runs over all words in $\upsilon_{1}\omega_{2}$ and $\upsilon_{2}\omega_{1}$, and 
\ben
&&\widetilde{\pi}_{-2k-1}=\{[(\upsilon_{1}\omega_1)^k\upsilon_1\widetilde{W}^\prime]\}\cap \pi,\\
&&\widetilde{\pi}_{2k+1}=\{[(\upsilon_{2}\omega_2)^k\upsilon_2\widetilde{W}^\prime]\}\cap \pi,
\een 
where $\widetilde{W}^\prime$ runs over all words in $\omega_2\upsilon_1$ and $\omega_1\upsilon_2$. Moreover, the bricks of $\widetilde{\pi}_k$ ($k\in\mathbb{Z}$) form a 2D Young diagram and we have the following interlacing properties for $k\geq0$, 
\ben
\widetilde{\pi}_{2k}\succ\widetilde{\pi}_{2k+1},\;\;\widetilde{\pi}^\prime_{-2k}\succ\widetilde{\pi}^\prime_{-2k-1},\;\;\widetilde{\pi}_{2k+1}^\prime\succ\widetilde{\pi}_{2k+2}^\prime,\;\;\widetilde{\pi}_{-2k-1}\succ\widetilde{\pi}_{-2k-2}.
\een
where $\widetilde{\pi}_{\pm k}$ is identified with  its corresponding 2D Young diagram (or partition) and $\widetilde{\pi}_{\pm k}^\prime$ denotes the transposition of 2D Young diagram $\widetilde{\pi}_{\pm k}$.
\end{lemma}

\begin{remark}\label{brick-box2}
The  interlacing properties for a pyramid partition $\pi$ in Lemma \ref{interlacing2} are equivalent to 
\bea\label{antidiagonal-interlacing}
\cdots\prec\widetilde{\pi}_{-5}\prec^\prime\widetilde{\pi}_{-4}\prec\widetilde{\pi}_{-3}\prec^\prime\widetilde{\pi}_{-2}\prec\widetilde{\pi}_{-1}\prec^\prime\widetilde{\pi}_0\succ\widetilde{\pi}_1\succ^\prime\widetilde{\pi}_2\succ\widetilde{\pi}_3\succ^\prime\widetilde{\pi}_4\succ\widetilde{\pi}_5\succ^\prime\cdots.
\eea
As in Remark \ref{brick-box1}, we have the following correspondence between bricks of $\widetilde{\pi}_{\pm k}$ and boxes of 2D Young diagram $\widetilde{\pi}_{\pm k}$ for $k\in\mathbb{Z}_{\geq0}$:\\
$(i)$ the brick $[(\upsilon_{1}\omega_1)^k(\upsilon_1\omega_2)^j(\upsilon_2\omega_1)^i]$ in $\widetilde{\pi}_{-2k}$ $\longleftrightarrow$ $(i,j)$-box in 2D Young diagram $\widetilde{\pi}_{-2k}$;\\
$(ii)$ the brick $[(\upsilon_{2}\omega_2)^k(\upsilon_1\omega_2)^j(\upsilon_2\omega_1)^i]$ in $\widetilde{\pi}_{2k}$ $\longleftrightarrow$ $(i,j)$-box in 2D Young diagram $\widetilde{\pi}_{2k}$;\\
$(iii)$ the brick $[(\upsilon_{1}\omega_1)^k\upsilon_1(\omega_2\upsilon_1)^j(\omega_1\upsilon_2)^i]$ in $\widetilde{\pi}_{-2k-1}$ $\longleftrightarrow$ $(i,j)$-box in 2D Young diagram $\widetilde{\pi}_{-2k-1}$;\\
$(iv)$ the brick $[(\upsilon_{2}\omega_2)^k\upsilon_2(\omega_2\upsilon_1)^j(\omega_1\upsilon_2)^i]$ in $\widetilde{\pi}_{2k+1}$ $\longleftrightarrow$ $(i,j)$-box in 2D Young diagram $\widetilde{\pi}_{2k+1}$,\\
where $\upsilon_1\omega_2$ and $\upsilon_2\omega_1$ (resp. $\omega_2\upsilon_1$ and $\omega_1\upsilon_2$) commute in $\mathbb{C}P/\mathcal{I}$.
\end{remark}
It is shown in [\cite{BY}, Section 7]  that any $k$-th antidiagonal slice $\widetilde{\pi}_k$ is checkerborder colored. One can refer to [\cite{BY}, Figure 6] for the coloring scheme.

\begin{remark}\label{characteristic-PP}
Let $\{\eta_k\,|\,k\in\mathbb{Z}\}$ be a sequence of partitions satisfying the interlacing properties \eqref{diagonal-interlacing} $($resp.  \eqref{antidiagonal-interlacing}$)$, and $\pi=\bigcup\limits_{k\in\mathbb{Z}}\eta_k$ be identified with  the corresponding bricks in $\mathfrak{P}$ by the  brick-box correspondence in  Remark \ref{brick-box1} $($resp.  Remark \ref{brick-box2}$)$. Then $\pi$ is a pyramid partition since for each brick $[B]$ in any $\eta_k$, every prefix of $B$ is also a brick in $\pi$ due to the interlacing properties \eqref{diagonal-interlacing} $($resp.  \eqref{antidiagonal-interlacing}$)$ as proved in  [\cite{BY}, Lemma 5.9]. 
\end{remark}

\subsection{The orbifold topological  vertex}
In this subsection, we will briefly recollect the definition of the    DT $\mathbb{Z}_2\times\mathbb{Z}_2$-vertex and its  known explicit formulas, and then recall the explict formula for the DT $\mathbb{Z}_n$-vertex ($n\geq2$), see [\cite{BY}, Section 1] and [\cite{BCY}, Section 3] for more details.

\begin{definition}([\cite{BY}, Section 1])\label{coloring}
The $\mathbb{Z}_n$-coloring is defined by a homomorphism of additive monoids $K_{\mathbb{Z}_n}: (\mathbb{Z}_{\geq0})^3\to\mathbb{Z}_n$ such that
\ben
&&K_{\mathbb{Z}_n}(1,0,0)=1,\\
&&K_{\mathbb{Z}_n}(0,1,0)=-1,\\
&&K_{\mathbb{Z}_n}(0,0,1)=0.
\een
The $\mathbb{Z}_2\times\mathbb{Z}_2$-coloring is defined by a homomorphism of additive monoids $K_{\mathbb{Z}_2\times\mathbb{Z}_2}: (\mathbb{Z}_{\geq0})^3\to\mathbb{Z}_2\times\mathbb{Z}_2$ such that
\ben
&&K_{\mathbb{Z}_2\times\mathbb{Z}_2}(1,0,0)=a,\\
&&K_{\mathbb{Z}_2\times\mathbb{Z}_2}(0,1,0)=b,\\
&&K_{\mathbb{Z}_2\times\mathbb{Z}_2}(0,0,1)=c,
\een
where $\mathbb{Z}_2\times\mathbb{Z}_2=\{0,a,b,c\}$.
\end{definition}	
Let $\mathcal{P}(\lambda,\mu,\nu)$ be the set of all 3D partitions $\pi$ asymptotic to $(\lambda,\mu,\nu)$ defined in [\cite{BCY}, Definition 4]. The three legs of $\pi\in\mathcal{P}(\lambda,\mu,\nu)$ are defined with respect to   $\lambda,\mu,\nu$ respectively  by
\ben
&&\{(i,j,k)\,|\,(j,k)\in\lambda\}\subset\pi,\\
&&\{(i,j,k)\,|\,(k,i)\in\mu\}\subset\pi,\\
&&\{(i,j,k)\,|\,(i,j)\in\nu\}\subset\pi,
\een
and we call $\lambda,\mu,\nu$ three leg partitions.
For any $(i,j,k)\in\pi$, set    
\ben
\xi_{\pi}(i,j,k)=1- \mbox{the number of legs containing $(i,j,k)$}.
\een
With the definition of the $\mathbb{Z}_2\times\mathbb{Z}_2$-coloring (see also the definition of $K_{G}$ in [\cite{BY}, Page 117]), we have  
\begin{definition}([\cite{BCY}, Equation (3)])\label{Z2Z2-vertex}
	The $\mathrm{DT}$ $\mathbb{Z}_{2}\times\mathbb{Z}_{2}$-vertex is defined by
	\ben
	V_{\lambda\mu\nu}^{\mathbb{Z}_{2}\times\mathbb{Z}_{2}}(q_0,q_a,q_b,q_c)=\sum_{\pi\in\mathcal{P}(\lambda,\mu,\nu)}q_{0}^{\Vert\pi\Vert_{0}}q_{a}^{\Vert\pi\Vert_{a}}q_{b}^{\Vert\pi\Vert_{b}} q_{c}^{\Vert\pi\Vert_{c}}
	\een
	where for $l=0,a,b,c$,
	\ben
	\Vert\pi\Vert_{l}=\sum_{\substack{(i,j,k)\in\pi\\ ia+jb+kc=l} }\xi_{\pi}(i,j,k)
	\een
	is called the renormalized number of boxes $(i,j,k)\in\pi$ with color $l\in\{0,a,b,c\}=\mathbb{Z}_2\times\mathbb{Z}_2$.
\end{definition}
\begin{remark}\label{Z2Z2-symmetry}
	It can be shown that the $\mathrm{DT}$ $\mathbb{Z}_{2}\times\mathbb{Z}_{2}$-vertex has the following symmetries:
	\ben
	V_{\lambda\mu\nu}^{\mathbb{Z}_{2}\times\mathbb{Z}_{2}}(q_0,q_a,q_b,q_c)=V_{\nu\lambda\mu}^{\mathbb{Z}_{2}\times\mathbb{Z}_{2}}(q_0,q_c,q_a,q_b)=V_{\mu\nu\lambda}^{\mathbb{Z}_{2}\times\mathbb{Z}_{2}}(q_0,q_b,q_c,q_a).
	\een
\end{remark}
Then we have the following explicit formula for the $\mathrm{DT}$ $\mathbb{Z}_{2}\times\mathbb{Z}_{2}$-vertex without legs.
\begin{theorem}([\cite{BY}, Theorem 1.5]) \label{Z2Z2-nolegs}
Let $q=q_0q_aq_bq_c$. Then  
\ben
V_{\emptyset\emptyset\emptyset}^{\mathbb{Z}_{2}\times\mathbb{Z}_{2}}(q_0,q_a,q_b,q_c)=\mathbf{M}(1,q)^4\cdot\frac{\widetilde{\mathbf{M}}(q_aq_b,q)\cdot\widetilde{\mathbf{M}}(q_aq_c,q)\cdot\widetilde{\mathbf{M}}(q_bq_c,q)}{\widetilde{\mathbf{M}}(-q_a,q)\cdot\widetilde{\mathbf{M}}(-q_b,q)\cdot\widetilde{\mathbf{M}}(-q_c,q)\cdot\widetilde{\mathbf{M}}(-q_aq_bq_c,q)}.
\een	
\end{theorem}

Next, we employ the notations in [\cite{BCY}, Section 3.4] as follows.
Let $\mathbf{\tilde{q}}_0=1$.  For $t\in\mathbb{Z}$, $\mathbf{\tilde{q}}_t$  is defined recursively by $\mathbf{\tilde{q}}_t=\tilde{q}_t\cdot\mathbf{\tilde{q}}_{t-1}$. Then 
\ben
\{\cdots,\mathbf{\tilde{q}}_{-2},\mathbf{\tilde{q}}_{-1},\mathbf{\tilde{q}}_0,\mathbf{\tilde{q}}_1,\mathbf{\tilde{q}}_2,\cdots\}=\{\cdots,\tilde{q}_0^{-1}\tilde{q}_{-1}^{-1},\tilde{q}_0^{-1},1,\tilde{q}_1,\tilde{q}_1\tilde{q}_2,\cdots\}.
\een
For any partition $\eta=(\eta_0,\eta_1,\eta_2,\cdots)$, let 
\ben
\mathbf{\tilde{q}}_{\bullet-\eta}=\{\mathbf{\tilde{q}}_{-\eta_0},\mathbf{\tilde{q}}_{1-\eta_1},\mathbf{\tilde{q}}_{2-\eta_2},\mathbf{\tilde{q}}_{3-\eta_3},\cdots\}.
\een
Let $\tilde{q}=\tilde{q}_0 \tilde{q}_1\cdots\tilde{q}_{n-1}$.  For any $k\in\mathbb{Z}$ and  partitions $\lambda$ and $\nu$, let 
\ben
&& \tilde{q}^{-A_{\lambda}}:=\prod_{k=0}^{n-1} \tilde{q}_k^{-A_{\lambda}(k,n)},\\
&&A_{\lambda}(k,n):=\sum_{(i,j)\in\lambda}\left\lfloor\frac{i+k}{n}\right\rfloor,\\
&&|\nu|_{\overline{k}}=\left|\{(i,j)\in\nu\,|\, i-j\equiv k\; (\mathrm{mod}\; n)\}\right|,\\
&&h^k_\nu(i,j)=\mbox{ the number of boxes of color $\overline{k}\in\mathbb{Z}_n$ in the $(i,j)$-hook of $\nu$}.
\een

Now we have the following explicit formula of Bryan-Cadman-Young for the DT $\mathbb{Z}_n$-vertex.
\begin{theorem}([\cite{BCY}, Theorem 12])\label{DT-Z_n}
The formula for the  $\mathrm{DT}$ $\mathbb{Z}_n$-vertex $V_{\lambda\mu\nu}^{\mathbb{Z}_n}(\tilde{q}_0,\tilde{q}_1,\cdots,\tilde{q}_{n-1})$ is
\ben
V_{\lambda\mu\nu}^{\mathbb{Z}_n}=V_{\emptyset\emptyset\emptyset}^{\mathbb{Z}_n}\cdot \tilde{q}^{-A_{\lambda}}\cdot \overline{\tilde{q}^{-A_{\mu^\prime}}}\cdot H_{\nu}\cdot O_{\nu}\cdot \sum_{\eta}\tilde{q}_{0}^{-|\eta|}\cdot \overline{s_{\lambda^\prime/\eta}(\mathbf{\tilde{q}}_{\bullet-\nu^\prime})}\cdot s_{\mu/\eta}(\mathbf{\tilde{q}}_{\bullet-\nu})
\een
where  $s_{\xi/\eta}$ is the skew Schur function associated to partitions $\eta\subset\xi$ ($s_{\xi/\eta}=0$  if $\eta\not\subset\xi$),  the overline denotes the exchange of variables $\tilde{q}_{k}\leftrightarrow \tilde{q}_{-k}$ with subscripts in $\mathbb{Z}_{n}$, and
\ben
&&V_{\emptyset\emptyset\emptyset}^{\mathbb{Z}_n}=\mathbf{M}(1,\tilde{q})^n\prod_{0<a\leq b<n}\widetilde{\mathbf{M}}(\tilde{q}_a\cdots \tilde{q}_b,\tilde{q}),\\
&&H_{\nu}=\prod_{(i,j)\in\nu}\frac{1}{1-\prod\limits_{k=0}^{n-1}\tilde{q}_k^{h^k_\nu(i,j)}},\\
&&O_\nu=\prod_{k=0}^{n-1}V_{\emptyset\emptyset\emptyset}^{\mathbb{Z}_n}(\tilde{q}_k,\tilde{q}_{k+1},\cdots, \tilde{q}_{n+k-1})^{-2|\nu|_{\overline{k}}+|\nu|_{\overline{k+1}}+|\nu|_{\overline{k-1}}}.
\een
\end{theorem}

\section{Restricted pyramid configurations}
In this section, we will introduce a notion of restricted pyramid configurations constructed from pyramid partitions. Let $\nu\neq\emptyset$ be a partition and $l\in\mathbb{Z}_{\geq0}$. We are interested in a special type of them called a class of restricted pyramid configurations $\widetilde{\bm{\wp}}_{\pi}(\nu,l)$ of antidiagonal type $(\nu,l)$, which are the natural 3D combinatorial objects for computing the 1-leg DT $\mathbb{Z}_2\times\mathbb{Z}_2$-vertex with the leg partition $\nu$. We also define a class of restricted pyramid configurations $\bm{\wp}_{\pi}(\nu,l)$ of diagonal type $(\nu,l)$, which is closely related to the 1-leg DT $\mathbb{Z}_4$-vertex $V_{\emptyset\emptyset\nu}^{\mathbb{Z}_4}$. See Definition \ref{unified-dfn-RPC} and Remark \ref{generalization of RPC notions} for notations.
In Section 3.2, we first show that   $\widetilde{\bm{\wp}}_{\pi}(\nu,0)=\bm{\wp}_{\pi}(\nu,0)$ for any  pyramid partition $\pi$  if and only if $\nu=(m,m-1,\cdots,2,1)$ where  $m\in\mathbb{Z}_{\geq1}$ is the length of $\nu$. In other word, among restricted pyramid configurations of antidiagonal (or diagonal) type $(\nu,0)$  for all partitions $\nu\neq\emptyset$, a class of  them  with $\nu=(m,m-1,\cdots,2,1)$ is the unique class with  the same interlacing property along their diagonal and antidiagonal slices. In Section 3.3, we generalize this uniqueness  for restricted pyramid configurations  of antidiagonal  (or diagonal)  type $(\nu,l)$ for any fixed  $l\in\mathbb{Z}_{\geq0}$.

\subsection{Restricted pyramid configurations}
We propose a notion of restricted pyramid configurations and 
define some classes of them according to the prescribed interlacing properties. And we show a sequence of partitions satisfying some interlacing property can be realized as a sequence of slices of  a restricted pyramid configuration.
 
\begin{definition}
For any pyramid partition $\pi\in\mathfrak{P}$, we call any finite subset $\bm{\wp}\subseteq\pi$ a restricted pyramid configuration.
\end{definition}

Actually a restricted pyramid configuration is constructed from some pyramid partition $\pi$ by removing  some finite  bricks in $\pi$. Notice  a pyramid partition is also a restricted pyramid configuration, but a restricted pyramid configuration is not necessary a pyramid partition.

Given any restricted pyramid configuration $\bm{\wp}\subseteq\pi$ for  some  $\pi\in\mathfrak{P}$, let 
\ben
\bm{\wp}_k:=\bm{\wp}\cap\pi_k
\een
 be  the $k$-th diagonal slice of $\bm{\wp}$, i.e., the set of all bricks in $\bm{\wp}$ whose position $(x,y,z)$ satisfying $x-y=k$ and  let 
 \ben
 \widetilde{\bm{\wp}}_k:=\bm{\wp}\cap\widetilde{\pi}_k
 \een
  be  the $k$-th antidiagonal slice of $\pi$, i.e., the set of all bricks in $\bm{\wp}$ whose position $(x,y,z)$ satisfying $x+y=k$.

To describe the interlacing property induced by the 1-leg partition $\nu$, we take the definition of  the edge sequence of a partition $\nu$ in [\cite{BCY}, Section 7.1]. The edge sequence of $\nu$ is defined for $t\in\mathbb{Z}$ by 
\ben
\nu(t)=\left\{
\begin{aligned}
	&+1 ,\;\;\;\;\;\; \mbox{if $t\in\mathit{S}(\nu)$}, \\
	& -1 ,\;\;\;\;\;\; \mbox{if $t\notin\mathit{S}(\nu)$}.
\end{aligned}
\right.
\een
where 
\ben
\mathit{S}(\nu)=\{\nu_j-j-1\,|\,j\geq0\}\subset\mathbb{Z}.
\een

\begin{remark}\label{charge0}
It is obvious that for a given partition $\nu$, we have  $\nu(t)=-1$  for $t\gg0$ and $\nu(t)=1$ for $t\ll0$. 
For a given partition $\nu$, we set 
\ben
\mathit{S}(\nu)_{+}=\mathit{S}(\nu)\setminus\mathbb{Z}_{<0},\;\;\;\;\mathit{S}(\nu)_{-}=\mathbb{Z}_{<0}\setminus \mathit{S}(\nu).
\een
Then we have 
\ben
|\mathit{S}(\nu)_{+}|=|\{t\,|\,\nu(t)=1,\,t\geq0\}|,\;\;\;\; |\mathit{S}(\nu)_{-}|=|\{t\,|\,\nu(t)=-1,\,t<0\}|.
\een
It is well known that $|\mathit{S}(\nu)_{+}|=|\mathit{S}(\nu)_{-}|$ for any partition $\nu$, which means the edge sequence $\nu(t)$ has charge zero, see [\cite{BCY}, Section 7.3] or [\cite{Okounkov}, Appendix A].
\end{remark}

\begin{definition}([\cite{BCY}, Definition 31])
	Let $\xi_1$ and $\xi_2$ be two partitions and $\tau=\pm1$, define
	\ben
	\xi_1\underset{\tau}{\prec\succ}\xi_2=\left\{
	\begin{aligned}
		&\xi_1\prec\xi_2 ,\;\;\;\;\;\; \mbox{if $\tau=+1$}, \\
		&\xi_1\succ\xi_2 ,\;\;\;\;\;\; \mbox{if $\tau=-1$}.
	\end{aligned}
	\right.
	\een
	and 
	\ben
	\xi_1\underset{\tau}{\prec\succ^\prime}\xi_2=\left\{
	\begin{aligned}
		&\xi_1\prec^\prime\xi_2 ,\;\;\;\;\;\; \mbox{if $\tau=+1$}, \\
		&\xi_1\succ^\prime\xi_2 ,\;\;\;\;\;\; \mbox{if $\tau=-1$}.
	\end{aligned}
	\right.
	\een
\end{definition}
 
\begin{definition}\label{dfn-interlacing}
	Let $\nu$ be a partition. We say that a sequence of partitions $\{\eta_k\}_{k\in\mathbb{Z}}$ satisfies the interlacing property of the first type $\nu$ if the following interlacing relations hold:
	\ben
	\cdots\underset{\nu^\prime(-5)}{\prec\succ}\eta_4\underset{\nu^\prime(-4)}{\prec\succ}\eta_3\underset{\nu^\prime(-3)}{\prec\succ}\eta_2\underset{\nu^\prime(-2)}{\prec\succ}\eta_1\underset{\nu^\prime(-1)}{\prec\succ}\eta_0\underset{\nu^\prime(0)}{\prec\succ}\eta_{-1}\underset{\nu^\prime(1)}{\prec\succ}\eta_{-2}\underset{\nu^\prime(2)}{\prec\succ}\eta_{-3}\underset{\nu^\prime(3)}{\prec\succ}\eta_{-4}\underset{\nu^\prime(4)}{\prec\succ}\cdots
	\een
	or equivalently 
	\ben
	\eta_{k}\underset{\nu^\prime(-k)}{\prec\succ}\eta_{k-1},\;\;\;\;\;\; \forall\, k\in\mathbb{Z}.
	\een

	We say that a sequence of partitions $\{\eta_k\}_{k\in\mathbb{Z}}$ satisfies the interlacing property of the second type $\nu$, if the following interlacing relations hold:
	\ben
	\cdots\underset{\nu^\prime(-5)}{\prec\succ}\eta_4\underset{\nu^\prime(-4)}{\prec\succ^\prime}\eta_3\underset{\nu^\prime(-3)}{\prec\succ}\eta_2\underset{\nu^\prime(-2)}{\prec\succ^\prime}\eta_1\underset{\nu^\prime(-1)}{\prec\succ}\eta_0\underset{\nu^\prime(0)}{\prec\succ^\prime}\eta_{-1}\underset{\nu^\prime(1)}{\prec\succ}\eta_{-2}\underset{\nu^\prime(2)}{\prec\succ^\prime}\eta_{-3}\underset{\nu^\prime(3)}{\prec\succ}\eta_{-4}\underset{\nu^\prime(4)}{\prec\succ^\prime}\cdots
	\een
	or equivalently 
	\ben
	\eta_{2k}\underset{\nu^\prime(-2k)}{\prec\succ^\prime}\eta_{2k-1}\;\;\mbox{and}\;\;\;\eta_{2k+1}\underset{\nu^\prime(-2k-1)}{\prec\succ}\eta_{2k},\;\;\;\;\;\; \forall\, k\in\mathbb{Z}.
	\een
	
	Here we use the edge sequence of $\nu^\prime$.
\end{definition}

The  interlacing property of the second type  $\nu$ describes  two kinds of restricted pyramid configurations with respect to  $\nu$ as follows.
\begin{definition}\label{interlacing-second-type}
	Let $\nu$ be a partition.  For any pyramid partition $\pi\in\mathfrak{P}$, we say that a restricted pyramid configuration $\bm{\wp}\subseteq\pi$  is of antidiagonal $($resp. diagonal$)$ type $\nu$ if the induced sequence $\{\widetilde{\bm{\wp}}_k\}_{k\in\mathbb{Z}}$  $($resp. $\{\bm{\wp}_k\}_{k\in\mathbb{Z}}$$)$ 
	satisfies the interlacing property of the second type $\nu$.
\end{definition}
A restricted pyramid configuration with the same interlacing property along its  diagonal and antidiagonal slices with respect to $\nu$ is defined as follows.
\begin{definition}\label{dfn-symmetric-interlacing}
	We say that a restricted pyramid configuration $\bm{\wp}$ has the symmetric interlacing property of type $\nu$ if $\bm{\wp}$ is  both  of antidiagonal  type $\nu$ and of  diagonal type $\nu$.
\end{definition}
\begin{remark}\label{interlacing-emptyset}
Since $\emptyset(t)=-1$ for $t\geq0$ and $\emptyset(t)=1$ for $t<0$, it follows from Remarks \ref{brick-box1} and  \ref{brick-box2} that any pyramid partition $\pi\in\mathfrak{P}$ is both of diagonal type  $\emptyset$ and of antidiagonal type $\emptyset$, that is, any pyramid partition $\pi\in\mathfrak{P}$ has the symmetric interlacing property of type $\emptyset$. 
\end{remark}
Let $\nu\neq\emptyset$ be a partition.
Next, we will  focus on constructing  a class of restricted pyramid configurations of antidiagonal type $\nu$ and a class of restricted pyramid configurations of diagonal type $\nu$ from pyramid partitions.
We first introduce some notations.
For $i,j\in\mathbb{Z}_{\geq0}$ and $k\in\mathbb{Z}$,
\ben
&&\widetilde{\mathbf{S}}_{2k}(i,j)=\left\{
\begin{aligned}
	& \left[(\upsilon_{1}\omega_1)^{-k}(\upsilon_1\omega_2)^j(\upsilon_2\omega_1)^i\right]\, \;\;\;\;\;\;\; \quad\quad \mbox{if  $k\leq 0$},\\
	& \left[(\upsilon_{2}\omega_2)^k(\upsilon_1\omega_2)^j(\upsilon_2\omega_1)^i\right], \;\;\;\;\;\;\;\;\quad\quad \mbox{if  $k>0$}.
\end{aligned}
\right.\\
&&\widetilde{\mathbf{S}}_{2k-1}(i,j)=\left\{
\begin{aligned}
	& \left[(\upsilon_{1}\omega_1)^{-k}\upsilon_1(\omega_2\upsilon_1)^j(\omega_1\upsilon_2)^i\right],  \;\;\;\;\;\; \mbox{if  $k\leq0$},\\
	& \left[(\upsilon_{2}\omega_2)^{k-1}\upsilon_2(\omega_2\upsilon_1)^j(\omega_1\upsilon_2)^i\right],\;\;\;\;\; \mbox{if  $k>0$}.
\end{aligned}
\right.\\
&&\mathbf{S}_{2k}(i,j)=\left\{
\begin{aligned}
	& \left[(\upsilon_{1}\omega_2)^{-k}(\upsilon_1\omega_1)^j(\upsilon_2\omega_2)^i\right],  \;\;\;\;\;\;\quad\quad\mbox{if  $k\leq0$},\\
	& \left[(\upsilon_{2}\omega_1)^k(\upsilon_1\omega_1)^j(\upsilon_2\omega_2)^i\right],\;\;\;\;\;\;\;\; \quad\quad\mbox{if  $k>0$}.
\end{aligned}
\right.\\
&&\mathbf{S}_{2k-1}(i,j)=\left\{
\begin{aligned}
	& \left[(\upsilon_{1}\omega_2)^{-k}\upsilon_1(\omega_1\upsilon_1)^j(\omega_2\upsilon_2)^i\right],  \;\;\;\;\;\;\mbox{if  $k\leq0$},\\
	& \left[(\upsilon_{2}\omega_1)^{k-1}\upsilon_2(\omega_1\upsilon_1)^j(\omega_2\upsilon_2)^i\right],\;\;\;\;\; \mbox{if  $k>0$}.
\end{aligned}
\right.
\een

The above bricks have the following relations.
\begin{lemma}\label{antidiagonal-diagonal}
	If $k\leq0$, we have
	\ben
	\widetilde{\mathbf{S}}_{2k}(i,j)=\left\{
	\begin{aligned}
		& \mathbf{S}_{-2(j-i)}(i,i-k),  \;\;\;\;\;\;\; \mbox{if  $j\geq i\geq0$},\\
		& \mathbf{S}_{2(i-j)}(j,j-k),\;\;\;\;\;\;\;\;\; \mbox{if  $i\geq j\geq0$}.
	\end{aligned}
	\right.
	\een
	and 
	\ben
	\widetilde{\mathbf{S}}_{2k-1}(i,j)=\left\{
	\begin{aligned}
		& \mathbf{S}_{-2(j-i)-1}(i,i-k),  \;\;\;\;\;\;\;\;\;\;\;\;\;\, \mbox{if  $j\geq i\geq0$},\\
		& \mathbf{S}_{2(i-j)-1}(j,j-k+1),\;\;\;\;\;\;\;\;\; \mbox{if  $i> j\geq0$}.
	\end{aligned}
	\right.
	\een
	If $k>0$, we have
	\ben
	\widetilde{\mathbf{S}}_{2k}(i,j)=\left\{
	\begin{aligned}
		& \mathbf{S}_{-2(j-i)}(i+k,i),  \;\;\;\;\;\;\; \mbox{if  $j\geq i\geq0$},\\
		& \mathbf{S}_{2(i-j)}(j+k,j),\;\;\;\;\;\;\;\;\; \mbox{if  $i\geq j\geq0$}.
	\end{aligned}
	\right.
	\een
	and 
	\ben
	\widetilde{\mathbf{S}}_{2k-1}(i,j)=\left\{
	\begin{aligned}
		& \mathbf{S}_{-2(j-i-1)-1}(i+k,i),   \;\;\;\;\;\;\;\;\;\;\;\;\;\, \mbox{if  $j> i\geq0$},\\
		& \mathbf{S}_{2(i-j+1)-1}(j+k-1,j),\;\;\;\;\;\;\;\;\; \mbox{if  $i\geq j\geq0$}.
	\end{aligned}
	\right.
	\een
\end{lemma}

\begin{proof}
	If $k\leq0$ and $j\geq i\geq0$, then 
	by Lemma \ref{brick-relation} we have 
	\ben
	\widetilde{\mathbf{S}}_{2k}(i,j)&=&\left[(\upsilon_{1}\omega_1)^{-k}(\upsilon_1\omega_2)^j(\upsilon_2\omega_1)^i\right]\\
	&=&\left[(\upsilon_1\omega_2)^j(\upsilon_2\omega_1)^i(\upsilon_{1}\omega_1)^{-k}\right]\\
	&=&\left[(\upsilon_1\omega_2)^{j-i}(\upsilon_1\omega_2)^{i}(\upsilon_2\omega_1)^i(\upsilon_{1}\omega_1)^{-k}\right]\\
	&=&\left[(\upsilon_1\omega_2)^{j-i}(\upsilon_{1}\omega_1)^{i-k}(\upsilon_2\omega_2)^i\right]\\
	&=&\mathbf{S}_{-2(j-i)}(i,i-k)
	\een
	If $k\leq0$ and $i> j\geq0$, by Lemma  \ref{brick-relation} again we have 
	\ben
	\widetilde{\mathbf{S}}_{2k-1}(i,j)&=&\left[(\upsilon_{1}\omega_1)^{-k}\upsilon_1(\omega_2\upsilon_1)^j(\omega_1\upsilon_2)^i\right]\\
	&=&\left[(\upsilon_{1}\omega_1)^{-k}(\upsilon_1\omega_2)^j(\upsilon_2\omega_1)^i\upsilon_1\right]\\
	&=&\left[(\upsilon_2\omega_1)^i(\upsilon_1\omega_2)^j(\upsilon_{1}\omega_1)^{-k}\upsilon_1\right]\\
	&=&\left[(\upsilon_2\omega_1)^{i-j}(\upsilon_2\omega_1)^{j}(\upsilon_1\omega_2)^j(\upsilon_{1}\omega_1)^{-k}\upsilon_1\right]\\
	&=&\left[(\upsilon_2\omega_1)^{i-j}(\upsilon_{1}\omega_1)^{j-k}(\upsilon_2\omega_2)^j\upsilon_1\right]\\
	&=&\left[(\upsilon_2\omega_1)^{i-j-1}(\upsilon_2\omega_1)(\upsilon_{1}\omega_1)^{j-k}(\upsilon_2\omega_2)^j\upsilon_1\right]\\
	&=&\left[(\upsilon_2\omega_1)^{i-j-1}(\upsilon_2\omega_1)\upsilon_1(\omega_1\upsilon_{1})^{j-k}(\omega_2\upsilon_2)^j\right]\\
	&=&\left[(\upsilon_2\omega_1)^{i-j-1}\upsilon_2(\omega_1\upsilon_{1})^{j-k+1}(\omega_2\upsilon_2)^j\right]\\
	&=&\mathbf{S}_{2(i-j)-1}(j,j-k+1).
	\een
	
	The other cases are similar.	
\end{proof}
To define a class of restricted pyramid configurations with the expected interlacing property, we also need to employ the following notations.
For a fixed partition $\nu$, let 
\ben
&&\varepsilon_1(\nu;\imath)=\sum_{t=0}^\imath\frac{\nu^\prime(2t)+1}{2},\\
&&\varepsilon_2(\nu;\imath)=\sum_{t=0}^\imath\frac{\nu^\prime(2t+1)+1}{2},\\
&&\varepsilon_3(\nu;\jmath)=\sum_{t=1}^\jmath\frac{1-\nu^\prime(-2t)}{2},\\
&&\varepsilon_4(\nu;\jmath)=\sum_{t=1}^\jmath\frac{1-\nu^\prime(-2t+1)}{2},
\een
where  if $\imath<0$, define $\varepsilon_i(\nu;\imath)=0$  for $i=1,2$ and if $\jmath<1$, define $\varepsilon_j(\nu;\jmath)=0$  for $j=3,4$. It is obvious that  for $i=1,2$ and $j=3,4$,  
\ben
&&0\leq\varepsilon_i(\nu;\imath+1)-\varepsilon_i(\nu;\imath)\leq1, \;\;\forall\;\imath\in\mathbb{Z};\\
&&0\leq\varepsilon_j(\nu;\jmath+1)-\varepsilon_j(\nu;\jmath)\leq1, \;\;\forall\;\jmath\in\mathbb{Z}.
\een
For a given partition $\nu$, we have  $\nu^\prime(t)=-1$  for $t\gg0$ and $\nu^\prime(t)=1$ for $t\ll0$, and hence there exists a positive integer $\varrho(\nu)$ such that for $i=1,2$ and $j=3,4$,
\ben
&&0\leq\varepsilon_i(\nu;\imath)\leq\varepsilon_i(\nu;\varrho(\nu)), \;\;\forall\;\imath\in\mathbb{Z};\\
&&0\leq\varepsilon_j(\nu;\jmath)\leq\varepsilon_j(\nu;\varrho(\nu)), \;\;\forall\;\jmath\in\mathbb{Z}.
\een
Define  $\varepsilon_i(\nu;+\infty)=\varepsilon_i(\nu;\varrho(\nu))$ and $\varepsilon_j(\nu;+\infty)=\varepsilon_j(\nu;\varrho(\nu))$  for $i=1,2$ and $j=3,4$. Let
\ben
&&\rho_1(\nu)=\max\left\{\varepsilon_2(\nu;+\infty),\; \varepsilon_4(\nu;+\infty)\right\};\\
&&\rho_2(\nu)=\max\left\{\varepsilon_1(\nu;+\infty),\; \varepsilon_3(\nu;+\infty)\right\}.
\een
Then $\rho_1(\nu)\geq\varepsilon_{l_1}(\nu;k)\geq0$ and  $\rho_2(\nu)\geq\varepsilon_{l_2}(\nu;k)\geq0$ for any $l_1\in\{2,4\}$, $l_2\in\{1,3\}$ and $k\in\mathbb{Z}$.
\begin{remark}\label{empty-case}
It is easy to show that  $\nu=\emptyset$ if and only if $\varepsilon_1(\nu;+\infty)=\varepsilon_2(\nu;+\infty)=\varepsilon_3(\nu;+\infty)=\varepsilon_4(\nu;+\infty)=0$, if and only if   $\rho_1(\nu)=\rho_2(\nu)=0$.
\end{remark}

 For a given partition $\nu$, we set
\ben
&&\widetilde{\mathcal{S}}_{2k}(\nu)=\left\{
\begin{aligned}
	&\left\{ \widetilde{\mathbf{S}}_{2k}(i,j)\,\bigg|\,i\geq\rho_{1}(\nu)-\varepsilon_2(\nu;-k-1),\; j\geq \rho_{2}(\nu)-\varepsilon_1(\nu;-k-1) \right\},  \;\; \mbox{if  $k\leq 0$},\\
	&\left\{ \widetilde{\mathbf{S}}_{2k}(i,j)\,\bigg|\,i\geq\rho_{1}(\nu)-\varepsilon_4(\nu;k),\; j\geq\rho_{2}(\nu)- \varepsilon_3(\nu;k)\right\},  \;\; \;\; \;\; \;\; \;\; \;\; \;\; \;\; \;\; \;\; \mbox{if  $k> 0$}.
\end{aligned}
\right.
\\
&&\widetilde{\mathcal{S}}_{2k-1}(\nu)=\left\{
\begin{aligned}
	&\left\{ \widetilde{\mathbf{S}}_{2k-1}(i,j)\,\bigg|\,i\geq\rho_{1}(\nu)- \varepsilon_2(\nu;-k-1),\; j\geq \rho_{2}(\nu)-\varepsilon_1(\nu;-k) \right\},  \;\; \mbox{if  $k\leq 0$},\\
	&\left\{ \widetilde{\mathbf{S}}_{2k-1}(i,j)\,\bigg|\,i\geq\rho_{1}(\nu)-\varepsilon_4(\nu;k),\; j\geq\rho_{2}(\nu)- \varepsilon_3(\nu;k-1) \right\},  \;\;\;\; \;\; \;\, \mbox{if  $k> 0$}.
\end{aligned}
\right.
\een
and 
\ben
&&\mathcal{S}_{2k}(\nu)=\left\{
\begin{aligned}
	&\left\{ \mathbf{S}_{2k}(i,j)\,\bigg|\,i\geq\rho_{1}(\nu)-\varepsilon_2(\nu;-k-1),\; j\geq \rho_{2}(\nu)-\varepsilon_1(\nu;-k-1) \right\},  \;\; \mbox{if  $k\leq 0$},\\
	&\left\{ \mathbf{S}_{2k}(i,j)\,\bigg|\,i\geq\rho_{1}(\nu)- \varepsilon_4(\nu;k),\; j\geq\rho_{2}(\nu)- \varepsilon_3(\nu;k)\right\},  \;\; \;\; \;\; \;\; \;\; \;\; \;\; \;\; \;\; \;\; \mbox{if  $k> 0$}.
\end{aligned}
\right.
\\
&&\mathcal{S}_{2k-1}(\nu)=\left\{
\begin{aligned}
	&\left\{ \mathbf{S}_{2k-1}(i,j)\,\bigg|\,i\geq\rho_{1}(\nu)- \varepsilon_2(\nu;-k-1),\; j\geq \rho_{2}(\nu)-\varepsilon_1(\nu;-k)  \right\},  \;\; \mbox{if  $k\leq 0$},\\
	&\left\{ \mathbf{S}_{2k-1}(i,j)\,\bigg|\,i\geq\rho_{1}(\nu)-\varepsilon_4(\nu;k),\; j\geq\rho_{2}(\nu)- \varepsilon_3(\nu;k-1) \right\},  \;\;\;\; \;\; \;\, \mbox{if  $k> 0$}.
\end{aligned}
\right.
\een
For any $k\in\mathbb{Z}$, the bricks of  $\widetilde{\mathcal{S}}_{k}(\nu)$ (resp. $\mathcal{S}_{k}(\nu)$) are all locating in $k$-th antidiagonal (resp. diagonal) slice positions by Lemma \ref{interlacing2} (resp. Lemma \ref{interlacing1}).
\begin{remark}\label{empty-case1}
If $\nu=\emptyset$, by Remark \ref{empty-case}, we have for any $k\in\mathbb{Z}$,
\ben
\widetilde{\mathcal{S}}_{k}(\emptyset)=\left\{ \widetilde{\mathbf{S}}_{k}(i,j)\,\bigg|\,i\geq0,\; j\geq 0 \right\}\quad\mbox{and}\quad
\mathcal{S}_{k}(\emptyset)=\left\{ \mathbf{S}_{k}(i,j)\,\bigg|\,i\geq0,\; j\geq 0 \right\}.
\een
\end{remark}
For any subset $A\subseteq(\mathbb{Z}_{\geq0})^2$ and $\mathcal{A}=\{a(i,j)\,|\,(i,j)\in A\}$, define
\ben
\mathcal{A}^c=\left\{a(i,j)\,|\,(i,j)\in (\mathbb{Z}_{\geq0})^2\setminus A\right\}.
\een
And  we take convention that $\mathcal{B}=\{j\in\mathbb{Z}\,|\,b_1\leq j\leq b_2\}=\emptyset$ whenever integers $b_1,b_2$ satisfy $b_1>b_2$. 

In the following, we define two restricted pyramid configurations by removing bricks in $\bigcup\limits_{k\in\mathbb{Z}}\left(\widetilde{\mathcal{S}}_{k}(\nu)\right)^c$ $\left(\mbox{resp. $\bigcup\limits_{k\in\mathbb{Z}}\left(\mathcal{S}_{2k}(\nu)\right)^c$}\right)$ from any pyramid partition $\pi$.
\begin{definition}\label{dfn-antidiag-diag-RPC}
Let $\nu$ be a partition. For any pyramid partition $\pi\in\mathfrak{P}$, define  
\ben
&&\widetilde{\bm{\wp}}_{\pi}(\nu)=\bigcup_{k\in\mathbb{Z}}\widetilde{\mathcal{S}}_{k}(\nu)\cap\widetilde{\pi}_{k};\\
&&\bm{\wp}_{\pi}(\nu)=\bigcup_{k\in\mathbb{Z}}\mathcal{S}_{k}(\nu)\cap\pi_{k}.
\een
\end{definition}

For any $k\in\mathbb{Z}$, set 
\ben
&&\widetilde{\bm{\wp}}_{\pi}(\nu)_k:=\widetilde{\bm{\wp}}_{\pi}(\nu)\cap\widetilde{\pi}_{k}=\widetilde{\mathcal{S}}_{k}(\nu)\cap\widetilde{\pi}_{k};\\
&& \bm{\wp}_{\pi}(\nu)_k:=\bm{\wp}_{\pi}(\nu)\cap\pi_{k}=\mathcal{S}_{k}(\nu)\cap\pi_{k}.
\een
 It is obvious that the  bricks in   $\widetilde{\bm{\wp}}_{\pi}(\nu)_k$ (resp. $\bm{\wp}_{\pi}(\nu)_k$) also form a 2D Young diagram by the definition of $\widetilde{\mathcal{S}}_{k}(\nu)$ (resp. $\mathcal{S}_{k}(\nu)$) together with Remarks \ref{brick-box1} and \ref{brick-box2}. And we also identify $\widetilde{\bm{\wp}}_{\pi}(\nu)_k$ (resp. $\bm{\wp}_{\pi}(\nu)_k$) with its corresponding 2D Young diagram (or partition).  Then we have two sequences of partitions $\{\widetilde{\bm{\wp}}_{\pi}(\nu)_k\}_{k\in\mathbb{Z}}$ and $\{\bm{\wp}_{\pi}(\nu)_k\}_{k\in\mathbb{Z}}$. They satisfy the following interlacing properties.

\begin{lemma}\label{RPC-interlacing1}
Let $\nu$ be a partition. For any pyramid partition $\pi\in\mathfrak{P}$, the restricted pyramid configuration $\widetilde{\bm{\wp}}_{\pi}(\nu)$ $($resp. $\bm{\wp}_{\pi}(\nu)$$)$ is of antidiagonal $($resp. diagonal$)$ type $\nu$.
\end{lemma}

\begin{proof}
We will prove that $\widetilde{\bm{\wp}}_{\pi}(\nu)$ is of antidiagonal type $\nu$, the case for $\bm{\wp}_{\pi}(\nu)$ of diagonal type $\nu$ is similar.
One needs to show the following interlacing relations:
\ben
\widetilde{\bm{\wp}}_{\pi}(\nu)_{2k}\underset{\nu^\prime(-2k)}{\prec\succ^\prime}\widetilde{\bm{\wp}}_{\pi}(\nu)_{2k-1}\;\;\mbox{and}\;\;\;\widetilde{\bm{\wp}}_{\pi}(\nu)_{2k+1}\underset{\nu^\prime(-2k-1)}{\prec\succ}\widetilde{\bm{\wp}}_{\pi}(\nu)_{2k},\;\;\;\;\;\; \forall\, k\in\mathbb{Z}.
\een
For simplicity of presentation, we take the following notations for any $\imath,\jmath\in\mathbb{Z}$
\bea\label{notations-interlacing1}
&&\widehat{\varepsilon}_1(\nu;\imath)=\rho_{2}(\nu)-\varepsilon_1(\nu;\imath);\quad\quad\quad \widehat{\varepsilon}_2(\nu;\imath)=\rho_{1}(\nu)-\varepsilon_2(\nu;\imath);\\
\label{notations-interlacing2}
&&\widehat{\varepsilon}_3(\nu;\jmath)=\rho_{2}(\nu)-\varepsilon_3(\nu;\jmath);\quad\quad\quad
\widehat{\varepsilon}_4(\nu;\jmath)=\rho_{1}(\nu)-\varepsilon_4(\nu;\jmath).
\eea
Let $\widetilde{R}_k^j$ and $\widetilde{R}_{k,i}$ be the set of bricks in the $j$-th column and $i$-th row of $\widetilde{\pi}_{k}$ respectively. Set $|\widetilde{R}_k^j|=\widetilde{\ell}_{k}^j$ and $|\widetilde{R}_{k,i}|=\widetilde{\ell}_{k,i}$. That is, we have for any $i,j\geq0$ and $k\in\mathbb{Z}$,
\ben
&&\widetilde{R}_k^j=\left\{\widetilde{\mathbf{S}}_{k}(i,j)\;\bigg|\;0\leq i< \widetilde{\ell}_{k}^j\right\},\\
&&\widetilde{R}_{k,i}=\left\{\widetilde{\mathbf{S}}_{k}(i,j)\;\bigg|\;0\leq j<\widetilde{\ell}_{k,i}\right\}.
\een
If $k\leq0$, by Lemma \ref{interlacing2} we have
\ben
\widetilde{\bm{\wp}}_{\pi}(\nu)_{2k}&=&\left\{ \widetilde{\mathbf{S}}_{2k}(i,j)\,\bigg|\,i\geq \widehat{\varepsilon}_2(\nu;-k-1),\; j\geq  \widehat{\varepsilon}_1(\nu;-k-1) \right\}\bigcap\widetilde{\pi}_{2k}\\
&=&\left\{ \widetilde{\mathbf{S}}_{2k}(i,j)\,\bigg|\,i\geq \widehat{\varepsilon}_2(\nu;-k-1),\; \widehat{\varepsilon}_1(\nu;-k-1)\leq j< \widetilde{\ell}_{2k,i}\right\},\\
\widetilde{\bm{\wp}}_{\pi}(\nu)_{2k-1}&=&\left\{ \widetilde{\mathbf{S}}_{2k-1}(i,j)\,\bigg|\,i\geq \widehat{\varepsilon}_2(\nu;-k-1),\; j\geq \widehat{\varepsilon}_1(\nu;-k) \right\}\bigcap\widetilde{\pi}_{2k-1}\\
&=&\left\{ \widetilde{\mathbf{S}}_{2k-1}(i,j)\,\bigg|\,i\geq \widehat{\varepsilon}_2(\nu;-k-1),\;  \widehat{\varepsilon}_1(\nu;-k)\leq j< \widetilde{\ell}_{2k-1,i} \right\}.
\een
By Lemma \ref{interlacing2}, we have $\widetilde{\pi}_{2k}\succ^\prime\widetilde{\pi}_{2k-1}$, which means that $\widetilde{\ell}_{2k,i}-\widetilde{\ell}_{2k-1,i}\in\{0,1\}$ for any $i\geq0$.

For any $i\geq \widehat{\varepsilon}_2(\nu;-k-1)$, if $\nu^\prime(-2k)=1$, then we have 
\ben
&&\left[\widetilde{\ell}_{2k-1,i}-\widehat{\varepsilon}_1(\nu;-k)\right]-\left[\widetilde{\ell}_{2k,i}-\widehat{\varepsilon}_1(\nu;-k-1)\right]\\
&=&\frac{\nu^\prime(-2k)+1}{2}-\left(\widetilde{\ell}_{2k,i}-\widetilde{\ell}_{2k-1,i}\right)\\
&=&\left[1-\left(\widetilde{\ell}_{2k,i}-\widetilde{\ell}_{2k-1,i}\right)\right]\in\{0,1\}.
\een
This implies that $\widetilde{\bm{\wp}}_{\pi}(\nu)_{2k}\prec^\prime\widetilde{\bm{\wp}}_{\pi}(\nu)_{2k-1}$ by Lemma \ref{interlacing}. If $\nu^\prime(-2k)=-1$, then we have 
\ben
\left[\widetilde{\ell}_{2k,i}-\widehat{\varepsilon}_1(\nu;-k-1)\right]-\left[\widetilde{\ell}_{2k-1,i}-\widehat{\varepsilon}_1(\nu;-k)\right]=\left(\widetilde{\ell}_{2k,i}-\widetilde{\ell}_{2k-1,i}\right)\in\{0,1\},
\een
which implies that $\widetilde{\bm{\wp}}_{\pi}(\nu)_{2k}\succ^\prime\widetilde{\bm{\wp}}_{\pi}(\nu)_{2k-1}$.
Then we  prove $\widetilde{\bm{\wp}}_{\pi}(\nu)_{2k}\underset{\nu^\prime(-2k)}{\prec\succ^\prime}\widetilde{\bm{\wp}}_{\pi}(\nu)_{2k-1}$ for $k\leq0$.

Similarly, if $k<0$, we also have 
\ben
\widetilde{\bm{\wp}}_{\pi}(\nu)_{2k}&=&\left\{ \widetilde{\mathbf{S}}_{2k}(i,j)\,\bigg|\, \widehat{\varepsilon}_2(\nu;-k-1)\leq i<\widetilde{\ell}_{2k}^j, \; j\geq  \widehat{\varepsilon}_1(\nu;-k-1) \right\},\\
\widetilde{\bm{\wp}}_{\pi}(\nu)_{2k+1}
&=&\left\{ \widetilde{\mathbf{S}}_{2k+1}(i,j)\,\bigg|\, \widehat{\varepsilon}_2(\nu;-k-2)\leq i<\widetilde{\ell}_{2k+1}^j,\; j\geq \widehat{\varepsilon}_1(\nu;-k-1)\right\}.
\een
By Lemma \ref{interlacing2}, we have $\widetilde{\pi}_{2k+1}\succ\widetilde{\pi}_{2k}$, which means that $\widetilde{\ell}_{2k+1}^j-\widetilde{\ell}_{2k}^j\in\{0,1\}$ for any $j\geq0$.

For any $j\geq \widehat{\varepsilon}_1(\nu;-k-1)$, if $\nu^\prime(-2k-1)=1$, then 
\ben
&&\left[\widetilde{\ell}_{2k}^j- \widehat{\varepsilon}_2(\nu;-k-1)\right]-\left[\widetilde{\ell}_{2k+1}^j- \widehat{\varepsilon}_2(\nu;-k-2)\right]
=\left[1-\left(\widetilde{\ell}_{2k+1}^j-\widetilde{\ell}_{2k}^j\right)\right]\in\{0,1\}.
\een
Hence $\widetilde{\bm{\wp}}_{\pi}(\nu)_{2k+1}\prec\widetilde{\bm{\wp}}_{\pi}(\nu)_{2k}$. If $\nu^\prime(-2k-1)=-1$, then 
\ben
\left[\widetilde{\ell}_{2k+1}^j- \widehat{\varepsilon}_2(\nu;-k-2)\right]-\left[\widetilde{\ell}_{2k}^j- \widehat{\varepsilon}_2(\nu;-k-1)\right]=\widetilde{\ell}_{2k+1}^j-\widetilde{\ell}_{2k}^j\in\{0,1\}.
\een
By Lemma \ref{interlacing}, we have $\widetilde{\bm{\wp}}_{\pi}(\nu)_{2k+1}\succ\widetilde{\bm{\wp}}_{\pi}(\nu)_{2k}$. Then  $\widetilde{\bm{\wp}}_{\pi}(\nu)_{2k+1}\underset{\nu^\prime(-2k-1)}{\prec\succ}\widetilde{\bm{\wp}}_{\pi}(\nu)_{2k}$ holds for $k<0$.

The other cases for $k\geq0$ are similar.
\end{proof}
\begin{remark}
For the later generalization, we also call $\widetilde{\bm{\wp}}_{\pi}(\nu)$ $($resp. $\bm{\wp}_{\pi}(\nu)$$)$ a restricted pyramid configuration of antidiagonal $($resp. diagonal$)$ type $(\nu,0)$, see Definition \ref{unified-dfn-RPC} and Remark \ref{generalization of RPC notions}.	
\end{remark}

Now, we define the following two classes of restricted pyramid configurations as mentioned above.

\begin{definition}\label{Two classes of RPC}
	Let $\nu$ be a partition. Define two sets of restricted pyramid configurations
	\ben
	&&\mathfrak{RP}_{+}(\nu)=\left\{\widetilde{\bm{\wp}}_{\pi}(\nu)\;\bigg| \;\pi\in\mathfrak{P}\right\},\\
	&&\mathfrak{RP}_{-}(\nu)=\left\{\bm{\wp}_{\pi}(\nu)\;\bigg| \;\pi\in\mathfrak{P}\right\}.
	\een
And we call $\mathfrak{RP}_{+}(\nu)$ $($resp. $\mathfrak{RP}_{-}(\nu)$$)$ a class of  restricted pyramid configurations of antidiagonal $($resp. diagonal$)$ type $\nu$.
\end{definition}
\begin{remark}\label{v-0-type-classes}
	We also call $\mathfrak{RP}_{+}(\nu)$ $($resp. $\mathfrak{RP}_{-}(\nu)$$)$ a class of restricted pyramid configurations of antidiagonal $($resp. diagonal$)$ type $(\nu,0)$, see Definition \ref{unified-dfn-RPC} and Remark \ref{generalization of RPC notions}.	
\end{remark}
\begin{remark}\label{empty1}
When $\nu=\emptyset$, we have $\widetilde{\bm{\wp}}_{\pi}(\nu)=\bm{\wp}_{\pi}(\nu)=\pi$ by Remark \ref{empty-case1} and hence $\mathfrak{RP}_{+}(\emptyset)=\mathfrak{RP}_{-}(\emptyset)=\mathfrak{P}$. Then $\mathfrak{P}$ is a class of restricted pyramid configurations of antidiagonal $($resp. diagonal$)$ type $\emptyset$.  Notice that if $\pi_1, \pi_2\in\mathfrak{P}$ and $\pi_1\neq\pi_2$, it is still possible to have $\widetilde{\bm{\wp}}_{\pi_1}(\nu)=\widetilde{\bm{\wp}}_{\pi_2}(\nu)$ or $\bm{\wp}_{\pi_1}(\nu)=\bm{\wp}_{\pi_2}(\nu)$ when $\nu\neq\emptyset$.
\end{remark}
To establish the relation between 1-leg DT $\mathbb{Z}_2\times\mathbb{Z}_2$-vertex (resp. 1-leg DT $\mathbb{Z}_4$-vertex) and restricted pyramid configurations $\mathfrak{RP}_{+}(\nu)$ (resp. $\mathfrak{RP}_{-}(\nu)$) later, we need the following 

\begin{lemma}\label{RPC-interlacing}
	Let $\nu$ be a partition. Let $\{\eta_k\}_{k\in\mathbb{Z}}$ be any sequence  of partitions  satisfying  the interlacing property of the second type $\nu$ and $\eta_k=\emptyset$ for all $|k|\gg0$, then there exists a pyramid partition $\pi\in\mathfrak{P}$ such that  $\eta_k=\widetilde{\bm{\wp}}_{\pi}(\nu)_k$ $($as the same 2D Young diagrams$)$ for any $k\in\mathbb{Z}$. Similarly, there exists a pyramid pyramid partition $\vartheta\in\mathfrak{P}$ such that $\eta_k=\bm{\wp}_{\vartheta}(\nu)_k$ for any $k\in\mathbb{Z}$. 	
\end{lemma}

\begin{proof}
We will prove the first statement, the  second one is similar. The statement holds when $\nu=\emptyset$ by Remark \ref{characteristic-PP}, it remains to consider the case when $\nu\neq\emptyset$. The strategy is to construct a pyramid partition $\pi$ by adding certain bricks to the partition $\eta_k$ for any $k\in\mathbb{Z}$, where  each $\eta_k$ is viewed as the $k$-th diagonal slice of the expected restricted pyramid configuration $\widetilde{\bm{\wp}}_{\pi}(\nu)$ below.

Since $\eta_k=\emptyset$ for all $|k|\gg0$,  $\nu^\prime(k)=-1$ for $k\gg0$, and $\nu^\prime(k)=1$ for $k\ll0$, there is an positive even integer $\hbar$ depending on $\nu$ such that $\eta_k=\emptyset$ for all $|k|\geq\hbar$, $\nu^\prime(k)=-1$ for $k\geq\hbar$, and $\nu^\prime(k)=1$ for $k\leq-\hbar$. Let 
\ben
&&\Theta_{\{\eta_\star\}}:=\max\left\{(\eta_k)^0\,|\,k\in\mathbb{Z}\right\},\\
&&\Xi_{\{\eta_\star\}}:=\max\left\{(\eta_k)_0\,|\,k\in\mathbb{Z}\right\},
\een
where $(\eta_k)_i$ is the $i$-th row length of $\eta_k$ and $(\eta_k)^j$ is the $j$-th column length of $\eta_k$ for any $k\in\mathbb{Z}$ and $i, j\in\mathbb{Z}_{\geq0}$. Explicitly, the desired pyramid partition $\pi$ is  constructed by defining its $k$-th antidiagonal slices for all $k\in\mathbb{Z}$ and then verifying the interlacing property \eqref{antidiagonal-interlacing} as follows.  \\
(i) If $-\frac{\hbar}{2}\leq k\leq0$, set
\ben
\widetilde{\underline{\pi}}_{2k}(\nu):&=&\left\{ \widetilde{\mathbf{S}}_{2k}(i,j)\,\bigg|\,0\leq i<\widehat{\varepsilon}_2(\nu;-k-1),\; 0\leq j< \widehat{\varepsilon}_1(\nu;-k-1)+\Xi_{\{\eta_\star\}} \right\}\\
&\bigcup&\left\{\widetilde{\mathbf{S}}_{2k}(i,j)\,\bigg|\,0\leq i-\widehat{\varepsilon}_2(\nu;-k-1)<\Theta_{\{\eta_\star\}},\; 0\leq j< \widehat{\varepsilon}_1(\nu;-k-1)+(\eta_{2k})_{i-\widehat{\varepsilon}_2(\nu;-k-1)}\right\},\\
\widetilde{\underline{\pi}}_{2k-1}(\nu):&=&\left\{ \widetilde{\mathbf{S}}_{2k-1}(i,j)\,\bigg|\,0\leq i<\widehat{\varepsilon}_2(\nu;-k-1),\; 0\leq j< \widehat{\varepsilon}_1(\nu;-k)+\Xi_{\{\eta_\star\}} \right\}\\
&\bigcup&\left\{\widetilde{\mathbf{S}}_{2k-1}(i,j)\,\bigg|\,0\leq i-\widehat{\varepsilon}_2(\nu;-k-1)<\Theta_{\{\eta_\star\}},\; 0\leq j< \widehat{\varepsilon}_1(\nu;-k)+(\eta_{2k-1})_{i-\widehat{\varepsilon}_2(\nu;-k-1)}\right\}.
\een
(ii) If  $-\frac{\hbar}{2}-\Xi_{\{\eta_\star\}}\leq k<-\frac{\hbar}{2}$, set
\ben
\widetilde{\underline{\pi}}_{2k}(\nu):&=&\left\{ \widetilde{\mathbf{S}}_{2k}(i,j)\,\bigg|\,0\leq i<\widehat{\varepsilon}_2\left(\nu;\frac{\hbar}{2}-1\right),\; 0\leq j< \widehat{\varepsilon}_1\left(\nu;\frac{\hbar}{2}-1\right)+\Xi_{\{\eta_\star\}}+k+\frac{\hbar}{2}+1\right\}\\
&\bigcup&\left\{\widetilde{\mathbf{S}}_{2k}(i,j)\,\bigg|\,0\leq i-\widehat{\varepsilon}_2\left(\nu;\frac{\hbar}{2}-1\right)<\Theta_{\{\eta_\star\}},\; 0\leq j< \widehat{\varepsilon}_1\left(\nu;\frac{\hbar}{2}-1\right)\right\},\\
\widetilde{\underline{\pi}}_{2k-1}(\nu):&=&\left\{ \widetilde{\mathbf{S}}_{2k-1}(i,j)\,\bigg|\,0\leq i<\widehat{\varepsilon}_2\left(\nu;\frac{\hbar}{2}-1\right),\; 0\leq j< \widehat{\varepsilon}_1\left(\nu;\frac{\hbar}{2}-1\right)+\Xi_{\{\eta_\star\}}+k+\frac{\hbar}{2} \right\}\\
&\bigcup&\left\{\widetilde{\mathbf{S}}_{2k-1}(i,j)\,\bigg|\,0\leq i-\widehat{\varepsilon}_2\left(\nu;\frac{\hbar}{2}-1\right)<\Theta_{\{\eta_\star\}},\; 0\leq j< \widehat{\varepsilon}_1\left(\nu;\frac{\hbar}{2}-1\right)\right\}.
\een
(iii) If  $-\frac{\hbar}{2}-\Xi_{\{\eta_\star\}}-\widehat{\varepsilon}_1\left(\nu;\frac{\hbar}{2}-1\right)\leq k<-\frac{\hbar}{2}-\Xi_{\{\eta_\star\}}$, set
\ben
\widetilde{\underline{\pi}}_{2k}(\nu):&=&\left\{\widetilde{\mathbf{S}}_{2k}(i,j)\,\bigg|\,0\leq i<\widehat{\varepsilon}_2\left(\nu;\frac{\hbar}{2}-1\right)+\Theta_{\{\eta_\star\}},\; 0\leq j< \widehat{\varepsilon}_1\left(\nu;\frac{\hbar}{2}-1\right)+k+\frac{\hbar}{2}+\Xi_{\{\eta_\star\}}+1\right\},\\
\widetilde{\underline{\pi}}_{2k-1}(\nu):&=&\left\{\widetilde{\mathbf{S}}_{2k-1}(i,j)\,\bigg|\,0\leq i<\widehat{\varepsilon}_2\left(\nu;\frac{\hbar}{2}-1\right)+\Theta_{\{\eta_\star\}},\; 0\leq j< \widehat{\varepsilon}_1\left(\nu;\frac{\hbar}{2}-1\right)+k+\frac{\hbar}{2}+\Xi_{\{\eta_\star\}}\right\}.
\een
(iv) If  $k<-\frac{\hbar}{2}-\Xi_{\{\eta_\star\}}-\widehat{\varepsilon}_1\left(\nu;\frac{\hbar}{2}-1\right)$, set 
$\widetilde{\underline{\pi}}_{2k-1}(\nu)=\widetilde{\underline{\pi}}_{2k}(\nu):=\emptyset$.\\
(v) If  $0< k\leq\frac{\hbar}{2}$, set
\ben
\widetilde{\underline{\pi}}_{2k}(\nu):&=&\left\{ \widetilde{\mathbf{S}}_{2k}(i,j)\,\bigg|\,0\leq i<\widehat{\varepsilon}_4(\nu;k),\; 0\leq j< \widehat{\varepsilon}_3(\nu;k)+\Xi_{\{\eta_\star\}} \right\}\\
&\bigcup&\left\{\widetilde{\mathbf{S}}_{2k}(i,j)\,\bigg|\,0\leq i-\widehat{\varepsilon}_4(\nu;k)<\Theta_{\{\eta_\star\}},\; 0\leq j< \widehat{\varepsilon}_3(\nu;k)+(\eta_{2k})_{i-\widehat{\varepsilon}_4(\nu;k)}\right\},\\
\widetilde{\underline{\pi}}_{2k-1}(\nu):&=&\left\{ \widetilde{\mathbf{S}}_{2k-1}(i,j)\,\bigg|\,0\leq i<\widehat{\varepsilon}_4(\nu;k),\; 0\leq j< \widehat{\varepsilon}_3(\nu;k-1) +\Xi_{\{\eta_\star\}} \right\}\\
&\bigcup&\left\{\widetilde{\mathbf{S}}_{2k-1}(i,j)\,\bigg|\,0\leq i-\widehat{\varepsilon}_4(\nu;k)<\Theta_{\{\eta_\star\}},\; 0\leq j< \widehat{\varepsilon}_3(\nu;k-1)+(\eta_{2k-1})_{i-\widehat{\varepsilon}_4(\nu;k)}\right\}.
\een
(vi) If  $\frac{\hbar}{2}<k\leq\frac{\hbar}{2}+\Theta_{\{\eta_\star\}}$, set 
\ben
\widetilde{\underline{\pi}}_{2k}(\nu)&:=&\left\{ \widetilde{\mathbf{S}}_{2k}(i,j)\,\bigg|\,0\leq i<\widehat{\varepsilon}_4\left(\nu;\frac{\hbar}{2}\right),\; 0\leq j< \widehat{\varepsilon}_3\left(\nu;\frac{\hbar}{2}\right)+\Xi_{\{\eta_\star\}}  \right\}\\
&\bigcup&\left\{\widetilde{\mathbf{S}}_{2k}(i,j)\,\bigg|\,0\leq i-\widehat{\varepsilon}_4\left(\nu;\frac{\hbar}{2}\right)<\Theta_{\{\eta_\star\}}-k+\frac{\hbar}{2},\; 0\leq j< \widehat{\varepsilon}_3\left(\nu;\frac{\hbar}{2}\right)\right\}\\
\widetilde{\underline{\pi}}_{2k-1}(\nu)&:=&\left\{ \widetilde{\mathbf{S}}_{2k-1}(i,j)\,\bigg|\,0\leq i<\widehat{\varepsilon}_4\left(\nu;\frac{\hbar}{2}\right),\; 0\leq j< \widehat{\varepsilon}_3\left(\nu;\frac{\hbar}{2}\right)+\Xi_{\{\eta_\star\}} \right\}\\
&\bigcup&\left\{\widetilde{\mathbf{S}}_{2k-1}(i,j)\,\bigg|\,0\leq i-\widehat{\varepsilon}_4\left(\nu;\frac{\hbar}{2}\right)<\Theta_{\{\eta_\star\}}-k+\frac{\hbar}{2},\; 0\leq j< \widehat{\varepsilon}_3\left(\nu;\frac{\hbar}{2}\right)\right\}.
\een
(vii) If  $\frac{\hbar}{2}+\Theta_{\{\eta_\star\}}<k\leq\frac{\hbar}{2}+\Theta_{\{\eta_\star\}}+\widehat{\varepsilon}_4\left(\nu;\frac{\hbar}{2}\right)$, set 
\ben
\widetilde{\underline{\pi}}_{2k}(\nu)&:=&\left\{ \widetilde{\mathbf{S}}_{2k}(i,j)\,\bigg|\,0\leq i<\widehat{\varepsilon}_4\left(\nu;\frac{\hbar}{2}\right)-k+\frac{\hbar}{2}+\Theta_{\{\eta_\star\}},\; 0\leq j< \widehat{\varepsilon}_3\left(\nu;\frac{\hbar}{2}\right)+\Xi_{\{\eta_\star\}}  \right\},\\
\widetilde{\underline{\pi}}_{2k-1}(\nu):&=&\left\{ \widetilde{\mathbf{S}}_{2k-1}(i,j)\,\bigg|\,0\leq i<\widehat{\varepsilon}_4\left(\nu;\frac{\hbar}{2}\right)-k+\frac{\hbar}{2}+\Theta_{\{\eta_\star\}},\; 0\leq j< \widehat{\varepsilon}_3\left(\nu;\frac{\hbar}{2}\right)+\Xi_{\{\eta_\star\}} \right\}.
\een
(viii) If  $k>\frac{\hbar}{2}+\Theta_{\{\eta_\star\}}+\widehat{\varepsilon}_4\left(\nu;\frac{\hbar}{2}\right)$, set
$
\widetilde{\underline{\pi}}_{2k-1}(\nu)=\widetilde{\underline{\pi}}_{2k}(\nu):=\emptyset.\\
$
In the above construction of $\{\widetilde{\underline{\pi}}_k\,|\,k\in\mathbb{Z}\}$, we take the notations in  $\eqref{notations-interlacing1}$ and $\eqref{notations-interlacing2}$.

Let $\pi=\bigcup\limits_{k\in\mathbb{Z}}\widetilde{\underline{\pi}}_k$. It is clear that  $\widetilde{\underline{\pi}}_k$ is a partition for any $k\in\mathbb{Z}$. By Remark \ref{characteristic-PP}, to show $\pi$ is  a pyramid partition, one only needs to show it satisfies the following interlacing property along its antidiagonal slices:
\bea\label{interlacing<=0}
&&\widetilde{\underline{\pi}}_{2k}(\nu)\succ^\prime\widetilde{\underline{\pi}}_{2k-1}(\nu)\;\;\mbox{and}\;\;\;\widetilde{\underline{\pi}}_{2k-1}(\nu)\succ\widetilde{\underline{\pi}}_{2k-2}(\nu),\;\;\;\;\;\; \forall\, k\in\mathbb{Z}_{\leq0},\\
\label{interlacing>0}
&&\widetilde{\underline{\pi}}_{2k}(\nu)\prec^\prime\widetilde{\underline{\pi}}_{2k-1}(\nu)\;\;\mbox{and}\;\;\;\widetilde{\underline{\pi}}_{2k-1}(\nu)\prec\widetilde{\underline{\pi}}_{2k-2}(\nu),\;\;\;\;\;\;\, \forall\, k\in\mathbb{Z}_{>0}.
\eea
By assumption, we have 
	\ben
	\eta_{2k}\underset{\nu^\prime(-2k)}{\prec\succ^\prime}\eta_{2k-1}\;\;\mbox{and}\;\;\;\eta_{2k+1}\underset{\nu^\prime(-2k-1)}{\prec\succ}\eta_{2k},\;\;\;\;\;\; \forall\, k\in\mathbb{Z}.
	\een
Suppose  for any $i,j\geq0$ and $k\in\mathbb{Z}$,
\ben
&&\widetilde{\underline{R}}_k^j(\nu)=\left\{\widetilde{\mathbf{S}}_{k}(i,j)\;\bigg|\;0\leq i< \widetilde{\underline{\ell}}_{k}^j(\nu)\right\},\\
&&\widetilde{\underline{R}}_{k,i}(\nu)=\left\{\widetilde{\mathbf{S}}_{k}(i,j)\;\bigg|\;0\leq j<\widetilde{\underline{\ell}}_{k,i}(\nu)\right\},
\een
where $\widetilde{\underline{R}}_k^j(\nu)$  (resp. $\widetilde{\underline{R}}_{k,i}(\nu)$) is the set of bricks in the $j$-th column  (resp. $i$-th row) of $\widetilde{\underline{\pi}}_{k}(\nu)$, and $|\widetilde{\underline{R}}_k^j(\nu)|=\widetilde{\underline{\ell}}_{k}^j(\nu)$ (resp. $|\widetilde{\underline{R}}_{k,i}(\nu)|=\widetilde{\underline{\ell}}_{k,i}(\nu)$).

If $-\frac{\hbar}{2}\leq k\leq0$, then we have
\ben
\widetilde{\underline{\ell}}_{2k,i}(\nu)=\left\{
\begin{aligned}
	&\widehat{\varepsilon}_1(\nu;-k-1)+\Xi_{\{\eta_\star\}},   \;\;  \;\;  \;\;  \;\;  \;\;  \;\;  \;\;  \;\;  \;\;  \;\;  \;\;  \;\, \mbox{if  $0\leq i< \widehat{\varepsilon}_2(\nu;-k-1)$},\\
	&\widehat{\varepsilon}_1(\nu;-k-1)+(\eta_{2k})_{i-\widehat{\varepsilon}_2(\nu;-k-1)}, \;\; \;\; \;\;  \mbox{if  $\widehat{\varepsilon}_2(\nu;-k-1)\leq i< \widehat{\varepsilon}_2(\nu;-k-1)+\Theta_{\{\eta_\star\}}$}.
\end{aligned}
\right.
\een
and 
\ben
\widetilde{\underline{\ell}}_{2k-1,i}(\nu)=\left\{
\begin{aligned}
	&\widehat{\varepsilon}_1(\nu;-k)+\Xi_{\{\eta_\star\}},   \;\;  \;\;  \;\;  \;\;  \;\;  \;\;  \;\;  \;\;  \;\;  \;\;  \;\;  \;\;\; \;\; \mbox{if  $0\leq i< \widehat{\varepsilon}_2(\nu;-k-1)$},\\
	&\widehat{\varepsilon}_1(\nu;-k)+(\eta_{2k-1})_{i-\widehat{\varepsilon}_2(\nu;-k-1)}, \;\; \;\; \;\;   \mbox{if  $\widehat{\varepsilon}_2(\nu;-k-1)\leq i< \widehat{\varepsilon}_2(\nu;-k-1)+\Theta_{\{\eta_\star\}}$}.
\end{aligned}
\right.
\een
Combined with $\eta_{2k}\underset{\nu^\prime(-2k)}{\prec\succ^\prime}\eta_{2k-1}$, one can show that $\widetilde{\underline{\ell}}_{2k,i}(\nu)-\widetilde{\underline{\ell}}_{2k-1,i}(\nu)\in\{0,1\}$ for any $0\leq i< \widehat{\varepsilon}_2(\nu;-k-1)+\Theta_{\{\eta_\star\}}$, which implies $\widetilde{\underline{\pi}}_{2k}(\nu)\succ^\prime\widetilde{\underline{\pi}}_{2k-1}(\nu)$. On the other hand, we also have 
\ben
\widetilde{\underline{\pi}}_{2k}(\nu):&=&\left\{ \widetilde{\mathbf{S}}_{2k}(i,j)\,\bigg|\,0\leq j<\widehat{\varepsilon}_1(\nu;-k-1),\; 0\leq i< \widehat{\varepsilon}_2(\nu;-k-1)+\Theta_{\{\eta_\star\}} \right\}\\
&\bigcup&\left\{\widetilde{\mathbf{S}}_{2k}(i,j)\,\bigg|\,0\leq j-\widehat{\varepsilon}_1(\nu;-k-1)<\Xi_{\{\eta_\star\}},\; 0\leq i< \widehat{\varepsilon}_2(\nu;-k-1)+(\eta_{2k})^{j-\widehat{\varepsilon}_1(\nu;-k-1)}\right\},\\
\widetilde{\underline{\pi}}_{2k-1}(\nu):&=&\left\{ \widetilde{\mathbf{S}}_{2k-1}(i,j)\,\bigg|\,0\leq j<\widehat{\varepsilon}_1(\nu;-k),\; 0\leq i< \widehat{\varepsilon}_2(\nu;-k-1)+\Theta_{\{\eta_\star\}} \right\}\\
&\bigcup&\left\{\widetilde{\mathbf{S}}_{2k-1}(i,j)\,\bigg|\,0\leq j-\widehat{\varepsilon}_1(\nu;-k)<\Xi_{\{\eta_\star\}},\; 0\leq i< \widehat{\varepsilon}_2(\nu;-k-1)+(\eta_{2k-1})^{j-\widehat{\varepsilon}_1(\nu;-k)}\right\}.
\een
Then if $-\frac{\hbar}{2}< k\leq0$, we have 
\ben
\widetilde{\underline{\ell}}_{2k-2}^j(\nu)=\left\{
\begin{aligned}
	&\widehat{\varepsilon}_2(\nu;-k)+\Theta_{\{\eta_\star\}},   \;\;  \;\;  \;\;  \;\;  \;\;  \;\;  \;\;  \;\;  \;\;  \;\;  \;\;  \;\; \; \;\, \mbox{if  $0\leq j< \widehat{\varepsilon}_1(\nu;-k)$},\\
	&\widehat{\varepsilon}_2(\nu;-k)+(\eta_{2k-2})^{j-\widehat{\varepsilon}_1(\nu;-k)}, \;\; \;\; \;\; \;\; \;\, \mbox{if  $\widehat{\varepsilon}_1(\nu;-k)\leq j< \widehat{\varepsilon}_1(\nu;-k)+\Xi_{\{\eta_\star\}}$}.
\end{aligned}
\right.
\een
and 
\ben
\widetilde{\underline{\ell}}_{2k-1}^j(\nu)=\left\{
\begin{aligned}
	&\widehat{\varepsilon}_2(\nu;-k-1)+\Theta_{\{\eta_\star\}},   \;\;  \;\;  \;\;  \;\;  \;\;  \;\;  \;\;  \;\;  \;\;  \;\;  \;\;  \;\, \mbox{if  $0\leq j< \widehat{\varepsilon}_1(\nu;-k)$},\\
	&\widehat{\varepsilon}_2(\nu;-k-1)+(\eta_{2k-1})^{j-\widehat{\varepsilon}_1(\nu;-k)}, \;\; \;\; \;\;   \mbox{if  $\widehat{\varepsilon}_1(\nu;-k)\leq j< \widehat{\varepsilon}_1(\nu;-k)+\Xi_{\{\eta_\star\}}$}.
\end{aligned}
\right.
\een
Using $\eta_{2k-1}\underset{\nu^\prime(-2k+1)}{\prec\succ}\eta_{2k-2}$, one can show  $\widetilde{\underline{\ell}}_{2k-1}^j(\nu)-\widetilde{\underline{\ell}}_{2k-2}^j(\nu)\in\{0,1\}$ for any $0\leq j< \widehat{\varepsilon}_1(\nu;-k)+\Xi_{\{\eta_\star\}}$. Then 
$\widetilde{\underline{\pi}}_{2k-1}(\nu)\succ\widetilde{\underline{\pi}}_{2k-2}(\nu)$. Now the interlacing property $\eqref{interlacing<=0}$ of  $\widetilde{\underline{\pi}}_k$ holds when $-\hbar-1\leq k\leq0$.

Similary, one can show the interlacing property $\eqref{interlacing<=0}$ of  $\widetilde{\underline{\pi}}_k$ holds for the following cases:\\
$(a)$ $-\hbar-2\Xi_{\{\eta_\star\}}-1\leq k\leq-\hbar-2$,\\
$(b)$ $-\hbar-2\Xi_{\{\eta_\star\}}-2\widehat{\varepsilon}_1\left(\nu;\frac{\hbar}{2}-1\right)-1\leq k\leq-\hbar-2\Xi_{\{\eta_\star\}}-2$,\\
$(c)$ $k\leq-\hbar-2\Xi_{\{\eta_\star\}}-2\widehat{\varepsilon}_1\left(\nu;\frac{\hbar}{2}-1\right)-2$,\\
and the interlacing property $\eqref{interlacing>0}$ of  $\widetilde{\underline{\pi}}_k$ holds for the following cases:\\
$(d)$ $1\leq k\leq \hbar$,\\
$(e)$ $\hbar+1\leq k\leq \hbar+2\Theta_{\{\eta_\star\}}$,\\
$(f)$ $\hbar+2\Theta_{\{\eta_\star\}}+1\leq k\leq\hbar+2\Theta_{\{\eta_\star\}}+2\widehat{\varepsilon}_4\left(\nu;\frac{\hbar}{2}\right)$,\\
$(g)$ $k\geq \hbar+2\Theta_{\{\eta_\star\}}+2\widehat{\varepsilon}_4\left(\nu;\frac{\hbar}{2}\right)+1$.

To prove that $\pi$ is a pyramid partition, it remains to show the following interlacing relations:\\
$(A)$ $\widetilde{\underline{\pi}}_{-\hbar-1}(\nu)\succ\widetilde{\underline{\pi}}_{-\hbar-2}(\nu)$;\;\; $\widetilde{\underline{\pi}}_{-\hbar-2\Xi_{\{\eta_\star\}}-1}(\nu)\succ\widetilde{\underline{\pi}}_{-\hbar-2\Xi_{\{\eta_\star\}}-2}(\nu)$;\;\; $\widetilde{\underline{\pi}}_{-\hbar-2\Xi_{\{\eta_\star\}}-2\widehat{\varepsilon}_1\left(\nu;\frac{\hbar}{2}-1\right)-1}(\nu)\succ\emptyset$;\\
$(B)$
$\widetilde{\underline{\pi}}_0(\nu)\succ\widetilde{\underline{\pi}}_1(\nu)$;\;\;
$\widetilde{\underline{\pi}}_\hbar(\nu)\succ\widetilde{\underline{\pi}}_{\hbar+1}(\nu)$;\;\; $\widetilde{\underline{\pi}}_{\hbar+2\Theta_{\{\eta_\star\}}}(\nu)\succ\widetilde{\underline{\pi}}_{\hbar+2\Theta_{\{\eta_\star\}}+1}(\nu)$;\;\; $\widetilde{\underline{\pi}}_{\hbar+2\Theta_{\{\eta_\star\}}+2\widehat{\varepsilon}_4\left(\nu;\frac{\hbar}{2}\right)}(\nu)\succ\emptyset$.

We only deal with  $\widetilde{\underline{\pi}}_{-\hbar-1}(\nu)\succ\widetilde{\underline{\pi}}_{-\hbar-2}(\nu)$ and $\widetilde{\underline{\pi}}_0(\nu)\succ\widetilde{\underline{\pi}}_1(\nu)$, the other cases are similar. It can  be shown that
\ben
&&\widetilde{\underline{\ell}}_{-\hbar-1}^j(\nu)=\left\{
\begin{aligned}
	&\widehat{\varepsilon}_2\left(\nu;\frac{\hbar}{2}-1\right)+\Theta_{\{\eta_\star\}},   \;\;  \;\;  \;\;  \;\;  \;\;  \;\;  \;\;  \;\;  \;\;  \;\;  \;\;  \;\, \mbox{if  $0\leq j< \widehat{\varepsilon}_1(\nu;\frac{\hbar}{2})$},\\
	&\widehat{\varepsilon}_2\left(\nu;\frac{\hbar}{2}-1\right)+(\eta_{-\hbar-1})^{j-\widehat{\varepsilon}_1(\nu;\frac{\hbar}{2})}, \;\; \;\; \;\;\,   \mbox{if  $\widehat{\varepsilon}_1(\nu;\frac{\hbar}{2})\leq j< \widehat{\varepsilon}_1(\nu;\frac{\hbar}{2})+\Xi_{\{\eta_\star\}}$}.
\end{aligned}
\right.\\
&&\widetilde{\underline{\ell}}_{-\hbar-2}^j(\nu)=\left\{
\begin{aligned}
	&\widehat{\varepsilon}_2\left(\nu;\frac{\hbar}{2}-1\right)+\Theta_{\{\eta_\star\}},   \;\;  \;\;  \;\;  \;\;  \;\;  \;\;  \;\;  \;\;  \;\;  \;\;  \;\;  \;\, \mbox{if  $0\leq j< \widehat{\varepsilon}_1(\nu;\frac{\hbar}{2}-1)$},\\
	&\widehat{\varepsilon}_2\left(\nu;\frac{\hbar}{2}-1\right), \;\; \;\; \;\;  \;\; \;\; \;\;  \;\; \;\; \;\;  \;\; \;\; \;\;  \;\; \;\; \;\;  \;\; \;\; \;\; \;  \mbox{if  $\widehat{\varepsilon}_1(\nu;\frac{\hbar}{2}-1)\leq j< \widehat{\varepsilon}_1(\nu;\frac{\hbar}{2})+\Xi_{\{\eta_\star\}}$}.
\end{aligned}
\right.
\een
and 
\ben
&&\widetilde{\underline{\ell}}_{0}^j(\nu)=\left\{
\begin{aligned}
	&\rho_1(\nu)+\Theta_{\{\eta_\star\}},   \;\;  \;\;  \;\;  \;\;  \;\;  \;\;  \;\;  \;\;  \;\;  \;\;  \;\;  \;\;  \;\;  \;\;  \;\;  \;\;  \;\;  \;\;  \;\, \mbox{if  $0\leq j< \rho_2(\nu)$},\\
	&\rho_1(\nu)+(\eta_0)^{j-\rho_2(\nu)}, \;\; \;\; \;\;   \;\;  \;\;  \;\;  \;\;  \;\;  \;\;  \;\;  \;\;  \;\;  \;\;  \;\;  \;\;   \mbox{if  $\rho_2(\nu)\leq j< \rho_2(\nu)+\Xi_{\{\eta_\star\}}$}.
\end{aligned}
\right.\\
&&\widetilde{\underline{\ell}}_{1}^j(\nu)=\left\{
\begin{aligned}
	&\rho_1(\nu)-\frac{1-\nu^\prime(-1)}{2}+\Theta_{\{\eta_\star\}},   \;\;  \;\;  \;\;  \;\;  \;\;  \;\;  \;\;  \;\;   \mbox{if  $0\leq j< \rho_2(\nu)$},\\
	&\rho_1(\nu)-\frac{1-\nu^\prime(-1)}{2}+(\eta_1)^{j-\rho_2(\nu)}, \;\; \;\; \;\;  \;\;    \mbox{if  $\rho_2(\nu)\leq j< \rho_2(\nu)+\Xi_{\{\eta_\star\}}$}.
\end{aligned}
\right.
\een
Then  $\widetilde{\underline{\ell}}_{-\hbar-1}^j(\nu)-\widetilde{\underline{\ell}}_{-\hbar-2}^j(\nu)=0$ for any $0\leq j< \widehat{\varepsilon}_1(\nu;\frac{\hbar}{2})+\Xi_{\{\eta_\star\}}$ since $\eta_{-\hbar-1}=\emptyset$ and $\widehat{\varepsilon}_1(\nu;\frac{\hbar}{2})= \widehat{\varepsilon}_1(\nu;\frac{\hbar}{2}-1)=\rho_2(\nu)-\widehat{\varepsilon}_1(\nu;+\infty)$. This implies that $\widetilde{\underline{\pi}}_{-\hbar-1}(\nu)\succ\widetilde{\underline{\pi}}_{-\hbar-2}(\nu)$.  And $\widetilde{\underline{\ell}}_{0}^j(\nu)-\widetilde{\underline{\ell}}_{1}^j(\nu)\in\{0,1\}$ for any $0\leq j< \rho_2(\nu)+\Xi_{\{\eta_\star\}}$ due to the assumption $\eta_{1}\underset{\nu^\prime(-1)}{\prec\succ}\eta_{0}$. Then $\widetilde{\underline{\pi}}_0(\nu)\succ\widetilde{\underline{\pi}}_1(\nu)$.

With the pyramid partition $\pi=\bigcup\limits_{k\in\mathbb{Z}}\widetilde{\underline{\pi}}_k$, we have if $-\frac{\hbar}{2}\leq k\leq0$,
\ben
\widetilde{\bm{\wp}}_{\pi}(\nu)_{2k}&=&\widetilde{\mathcal{S}}_{2k}(\nu)\cap\widetilde{\underline{\pi}}_{2k}(\nu)\\
&=&\left\{\widetilde{\mathbf{S}}_{2k}(i,j)\,\bigg|\,0\leq i-\widehat{\varepsilon}_2(\nu;-k-1)<\Theta_{\{\eta_\star\}},\; 0\leq j-\widehat{\varepsilon}_1(\nu;-k-1)< (\eta_{2k})_{i-\widehat{\varepsilon}_2(\nu;-k-1)}\right\}
\een
which has the same 2D Young diagram as $\eta_{2k}$, that is, $\eta_{2k}=\widetilde{\bm{\wp}}_{\pi}(\nu)_{2k}$. The other cases are similar.
\end{proof}

\subsection{Restricted pyramid configurations with symmetric interlacing property}
Let $\nu\neq\emptyset$ be a partition. 
In this subsection, we investigate a class of restricted pyramid configurations $\mathfrak{RP}_{+}(\nu)$ $($or  $\mathfrak{RP}_{-}(\nu)$$)$ of antidiagonal $($or diagonal$)$   type $(\nu,0)$, see Definition \ref{Two classes of RPC} and Remark \ref{v-0-type-classes}. And we show $\mathfrak{RP}_{+}(\nu)$  $($or  $\mathfrak{RP}_{-}(\nu)$$)$ is the unique class of restricted pyramid configurations of antidiagonal $($or diagonal$)$   type $(\nu,0)$ with the symmetric interlacing property of type $\nu$ if $\nu=(m,m-1,\cdots,2,1)$ for some $m\geq1$, see Proposition \ref{unique symmetric interlacing}. Now let us start with the case when $\nu=(m,m-1,\cdots,2,1)$ with $m\in\mathbb{Z}_{\geq1}$ and then show   $\widetilde{\bm{\wp}}_{\pi}(\nu)=\bm{\wp}_{\pi}(\nu)$ for any  pyramid partition $\pi$. We need the following
\begin{lemma}\label{even-symmetry}
	Assume  $\nu=(m,m-1,\cdots,2,1)$ with $m\in\mathbb{Z}_{\geq1}$. Then we  have 
	\ben
	\bigcup_{k\in\mathbb{Z}}\left(\widetilde{\mathcal{S}}_{2k}(\nu)\right)^c=\bigcup_{k\in\mathbb{Z}}\left(\mathcal{S}_{2k}(\nu)\right)^c.
	\een
\end{lemma}

\begin{proof}
 We will show the case	when $m$ is even, the other case is similar.   One can show
	\ben
	\left(\widetilde{\mathcal{S}}_{2k}(\nu)\right)^c=\left\{
	\begin{aligned}
		&\left\{ \widetilde{\mathbf{S}}_{2k}(i,j)\,\bigg|\,i\geq0, \; 0\leq j\leq \frac{m}{2}-1 \right\},  \;\;\;\;\;\;\;\;\;\;\;\;\; \;\;\,\quad\quad\quad\quad\quad\quad\quad\quad\quad\quad \mbox{if  $k\leq \displaystyle-\frac{m}{2}$},\\
		&\left\{ \widetilde{\mathbf{S}}_{2k}(i,j)\,\bigg|\,i\geq0, \; 0\leq j\leq \frac{m}{2}-1\mbox{ or }  0\leq i\leq\frac{m}{2}+k-1,\;j\geq\frac{m}{2}	 \right\},  \;\;\;\;\mbox{if  $\displaystyle-\frac{m}{2}< k\leq0$},\\
		&\left\{ \widetilde{\mathbf{S}}_{2k}(i,j)\,\bigg|\,0\leq i \leq \frac{m}{2}-1,\; j\geq0 \mbox{ or } i\geq\frac{m}{2}, \;0\leq j\leq \frac{m}{2}-k-1 \right\},\;\;\;\; \mbox{if  $0< k\leq\displaystyle\frac{m}{2}$},\\
		&\left\{ \widetilde{\mathbf{S}}_{2k}(i,j)\,\bigg|\,0\leq i \leq \frac{m}{2}-1,\; j\geq0  \right\},\;\;\;\;\;\;\;\;\;\;\;\; \;\;\,\;\;\quad\quad\quad\quad\quad\quad\quad\quad\quad\quad\mbox{if  $k>\displaystyle\frac{m}{2}$}.
	\end{aligned}
	\right.
	\een
	By Lemma \ref{antidiagonal-diagonal}, we have 
	\ben
	\left(\widetilde{\mathcal{S}}_{2k}(\nu)\right)^c=\left\{
	\begin{aligned}
		&A_{k,-}(\nu)\cup A_{k,+}(\nu),  \;\;\;\;\;\;\;\;\;\;\;\; \;\;\,\quad\quad\;\;\, \mbox{if  $k\leq \displaystyle-\frac{m}{2}$},\\
		&A_{k,-}(\nu)\cup A_{k,+}(\nu) \cup B_{k,-}(\nu),  \;\;\;\;\;\; \;\;  \mbox{if  $\displaystyle-\frac{m}{2}< k\leq0$},\\
		&A_{k,-}(\nu)\cup A_{k,+}(\nu) \cup B_{k,+}(\nu),\;\;\;\;\;\;\;\; \mbox{if  $0< k\leq\displaystyle\frac{m}{2}$},\\
		&A_{k,-}(\nu)\cup A_{k,+}(\nu),\;\;\;\;\;\;\;\;\;\;\;\; \;\;\,\;\; \quad\quad\; \mbox{if  $k>\displaystyle\frac{m}{2}$}.
	\end{aligned}
	\right.
	\een
	where 
	\ben
	&&A_{k,-}(\nu)=\left\{
	\begin{aligned}
		&\left\{ \mathbf{S}_{-2(j-i)}(i,i-k)\,\bigg|\,0\leq i\leq j,\; 0\leq j\leq \frac{m}{2}-1 \right\},  \;\;\;\;\;\;\, \mbox{if  $k\leq 0$},\\
		&\left\{ \mathbf{S}_{-2(j-i)}(i+k,i)\,\bigg|\,0\leq i\leq\frac{m}{2}-1,\; j\geq i \right\},  \;\;\;\;\;\;\quad\quad  \mbox{if  $k>0$}.
	\end{aligned}
	\right.\\
	&&A_{k,+}(\nu)=\left\{
	\begin{aligned}
		&\left\{ \mathbf{S}_{2(i-j)}(j,j-k)\,\bigg|\, i> j,\; 0\leq j\leq \frac{m}{2}-1 \right\},  \;\;\;\;\;\;\;\;\;\;\;\; \;\;\, \mbox{if  $k\leq 0$},\\
		&\left\{ \mathbf{S}_{2(i-j)}(j+k,j)\,\bigg|\,0\leq i\leq\frac{m}{2}-1,\;0\leq j< i \right\},  \;\;\;\;\;\;\;\;  \mbox{if  $k>0$}.
	\end{aligned}
	\right.\\
	&&B_{k,-}(\nu)=\left\{ \mathbf{S}_{-2(j-i)}(i,i-k)\,\bigg|\,0\leq i\leq\frac{m}{2}+k-1,\;  j\geq \frac{m}{2} \right\},  \;\;\;\;\;\;\;\, \mbox{if  $\displaystyle-\frac{m}{2}<k\leq 0$},\\
	&&B_{k,+}(\nu)=\left\{ \mathbf{S}_{2(i-j)}(j+k,j)\,\bigg|\, i\geq\frac{m}{2},\;0\leq j\leq\frac{m}{2}-k-1 \right\},  \;\;\;\;\;\;\;\;\;   \mbox{if  $0<k\leq\displaystyle\frac{m}{2}$}.
	\een
	Then we have
	\ben
	\bigcup_{k\in\mathbb{Z}}\left(\widetilde{\mathcal{S}}_{2k}(\nu)\right)^c=\left(\bigcup_{k\in\mathbb{Z}}A_{k,-}(\nu)\cup A_{k,+}(\nu)\right)\bigcup \left(\bigcup_{-\frac{m}{2}<k\leq0} B_{k,-}(\nu)\right)\bigcup \left(\bigcup_{0<k\leq\frac{m}{2}}B_{k,+}(\nu)\right).
	\een
	Since we have
	\ben
	&&\bigcup_{k\leq0}A_{k,-}(\nu)=\bigcup_{0\leq l< \frac{m}{2}}\left\{\mathbf{S}_{-2l}(i,i-k)\,\bigg|\,0\leq i\leq\frac{m}{2}-1-l, \,k\leq0 \right\},\\
	&&\bigcup_{k>0}A_{k,-}(\nu)=\bigcup_{l\geq 0}\left\{\mathbf{S}_{-2l}(i+k,i)\,\bigg|\,0\leq i\leq\frac{m}{2}-1, \,k>0 \right\},\\
	&&\bigcup_{-\frac{m}{2}<k\leq0} B_{k,-}(\nu)=\bigcup_{1\leq l< \frac{m}{2}}\left\{\mathbf{S}_{-2l}(i,i-k)\,\bigg|\,\frac{m}{2}-l\leq i\leq\frac{m}{2}+k-1, \,1-l\leq k\leq0 \right\}\\
	&&\quad\quad\quad\quad\quad\quad\quad\quad\;\;\bigcup\bigcup_{l\geq \frac{m}{2}}\left\{\mathbf{S}_{-2l}(i,i-k)\,\bigg|\,0\leq i\leq\frac{m}{2}+k-1, \,-\frac{m}{2}< k\leq0 \right\}
	\een
	and
		\ben
		&&\bigcup_{k\leq0}A_{k,+}(\nu)=\bigcup_{l\geq 1}\left\{\mathbf{S}_{2l}(j,j-k)\,\bigg|\,0\leq j\leq\frac{m}{2}-1, \,k\leq0 \right\}\\
		&&\bigcup_{k>0}A_{k,+}(\nu)=\bigcup_{1\leq l<\frac{m}{2}}\left\{\mathbf{S}_{2l}(j+k,j)\,\bigg|\,0\leq j\leq\frac{m}{2}-1-l, \,k>0 \right\}\\
		&&\bigcup_{0<k\leq\frac{m}{2}} B_{k,+}(\nu)=\bigcup_{2\leq l\leq \frac{m}{2}}\left\{\mathbf{S}_{2l}(j+k,j)\,\bigg|\,\frac{m}{2}-l\leq j\leq\frac{m}{2}-k-1, \,0< k\leq l-1 \right\}\\
		&&\quad\quad\quad\quad\quad\quad\quad\quad\bigcup\bigcup_{l> \frac{m}{2}}\left\{\mathbf{S}_{2l}(j+k,j)\,\bigg|\,0\leq j\leq\frac{m}{2}-k-1, \,0<k\leq\frac{m}{2} \right\},
		\een
	then one can show
	\ben
	&&\left(\bigcup_{k\in\mathbb{Z}}A_{k,-}(\nu)\right)\bigcup \left(\bigcup_{-\frac{m}{2}<k\leq0} B_{k,-}(\nu)\right)\\
	&=&\left(\bigcup_{l\leq-\frac{m}{2}}\left\{ \mathbf{S}_{2l}(i,j)\,\bigg|\,i\geq0, \,j\geq \frac{m}{2} \right\}^c\right)\bigcup \left(\bigcup_{-\frac{m}{2}<l\leq0}\left\{ \mathbf{S}_{2l}(i,j)\,\bigg|\,i\geq\frac{m}{2}+l,\, j\geq \frac{m}{2} \right\}^c\right)
	\een
 and	
	\ben
	&&\left(\bigcup_{k\in\mathbb{Z}} A_{k,+}(\nu)\right)\bigcup \left(\bigcup_{0<k\leq\frac{m}{2}}B_{k,+}(\nu)\right)\\
	&=&\left(\bigcup_{l>\frac{m}{2}}\left\{ \mathbf{S}_{2l}(i,j)\,\bigg|\,i\geq\frac{m}{2},\, j\geq 0 \right\}^c\right)\bigcup \left(\bigcup_{0<l\leq\frac{m}{2}}\left\{ \mathbf{S}_{2l}(i,j)\,\bigg|\,i\geq\frac{m}{2}, \,j\geq \frac{m}{2}-l \right\}^c\right)
	\een
	Now the proof  follows from the definition of $\mathcal{S}_{2k}(\nu)$  for any $k\in\mathbb{Z}$ when $\nu=(m,m-1,\cdots,2,1)$.
\end{proof}

Similarly, we have 
\begin{lemma}\label{odd-symmetry}
Assume  $\nu=(m,m-1,\cdots,2,1)$ with $m\in\mathbb{Z}_{\geq1}$. Then we  have 
\ben
\bigcup_{k\in\mathbb{Z}}\left(\widetilde{\mathcal{S}}_{2k-1}(\nu)\right)^c=\bigcup_{k\in\mathbb{Z}}\left(\mathcal{S}_{2k-1}(\nu)\right)^c.
\een
\end{lemma}

\begin{proof}
For the convenience of readers, we present the proof of  the case when $m$ is odd. The other case is similar. Actually, one can show
\ben
\left(\widetilde{\mathcal{S}}_{2k-1}(\nu)\right)^c=\left\{
\begin{aligned}
	&\left\{ \widetilde{\mathbf{S}}_{2k-1}(i,j)\,\bigg|\,0\leq i\leq \frac{m-1}{2},\,j\geq0 \right\},  \;\;\;\;\;\;\;\;\;\;\;\; \;\;\,\quad\quad\quad\quad\quad\quad\quad\quad\quad\quad\quad\; \mbox{if  $k\leq \displaystyle-\frac{m+1}{2}$},\\
	&\left\{ \widetilde{\mathbf{S}}_{2k-1}(i,j)\,\bigg|\,0\leq i\leq \frac{m-1}{2},\,j\geq0 \mbox{ or } i\geq\frac{m+1}{2}, \; 0\leq j\leq\frac{m-3}{2}+k	 \right\},  \; \mbox{if  $\displaystyle-\frac{m+1}{2}< k\leq0$},\\
	&\left\{ \widetilde{\mathbf{S}}_{2k-1}(i,j)\,\bigg|\,i\geq0,\,0\leq j \leq \frac{m-1}{2} \mbox{ or } 0\leq i\leq \frac{m-1}{2}-k,\,j\geq\frac{m+1}{2} \right\},\;\; \mbox{if  $0< k\leq\displaystyle\frac{m+1}{2}$},\\
	&\left\{ \widetilde{\mathbf{S}}_{2k-1}(i,j)\,\bigg|\,i\geq0,\,0\leq j \leq \frac{m-1}{2}  \right\},\;\;\;\;\;\;\;\;\;\;\;\; \;\;\;\;\,\quad\quad\quad\quad\quad\quad\quad\quad\quad\quad\quad\mbox{if  $k>\displaystyle\frac{m+1}{2}$}.
\end{aligned}
\right.
\een
By Lemma \ref{antidiagonal-diagonal}, we have 
\ben
\left(\widetilde{\mathcal{S}}_{2k-1}(\nu)\right)^c=\left\{
\begin{aligned}
	&C_{k,-}(\nu)\cup C_{k,+}(\nu),  \;\;\;\;\;\;\;\;\;\;\;\; \;\;\, \;\;\;\;\;\;\;\;\mbox{if  $k\leq \displaystyle-\frac{m+1}{2}$},\\
	&C_{k,-}(\nu)\cup C_{k,+}(\nu) \cup D_{k,+}(\nu),  \;\;\;\;\;\;\mbox{if  $\displaystyle-\frac{m+1}{2}< k\leq0$},\\
	&C_{k,-}(\nu)\cup C_{k,+}(\nu) \cup D_{k,-}(\nu),\;\;\;\;\;\;\, \mbox{if  $0< k\leq\displaystyle\frac{m+1}{2}$},\\
	&C_{k,-}(\nu)\cup C_{k,+}(\nu),\;\;\;\;\;\;\;\;\;\;\;\; \;\;\,\;\;\;\;\;\;\;\;\, \mbox{if  $k>\displaystyle\frac{m+1}{2}$}.
\end{aligned}
\right.
\een
where
\ben
&&C_{k,-}(\nu)=\left\{
\begin{aligned}
	&\left\{ \mathbf{S}_{-2(j-i)-1}(i,i-k)\,\bigg|\,0\leq i\leq\frac{m+1}{2},\; j\geq i  \right\},  \;\;\;\;\;\;\;\;\;\;\;\; \;\;\,\;\;\;\;\;\;\;\;\;\, \mbox{if  $k\leq 0$},\\
	&\left\{ \mathbf{S}_{-2(j-i-1)-1}(i+k,i)\,\bigg|\, 0\leq i< j,\; 0\leq j< \frac{m+1}{2}\right\},  \;\;\;\;\;\;\;\;\;\;\;\;\;\;  \mbox{if  $k>0$}.
\end{aligned}
\right.\\
&&C_{k,+}(\nu)=\left\{
\begin{aligned}
	&\left\{ \mathbf{S}_{2(i-j)-1}(j,j-k+1)\,\bigg|\,0\leq i<\frac{m+1}{2},\;0\leq j< i  \right\},  \;\;\;\;\;\;\;\;\;\;\;\; \;\, \mbox{if  $k\leq 0$},\\
	&\left\{ \mathbf{S}_{2(i-j+1)-1}(j+k-1,j)\,\bigg|\, i\geq j,\; 0\leq j< \frac{m+1}{2} \right\},  \;\;\;\;\;\;\;\;\;\;\;\;\;\;\;\; \mbox{if  $k>0$}.
\end{aligned}
\right.\\
&&D_{k,-}(\nu)=\left\{ \mathbf{S}_{-2(j-i-1)-1}(i+k,i)\,\bigg|\,0\leq i\leq \frac{m+1}{2}-k-1,\; j\geq \frac{m+1}{2} \right\},  \;\;\; \mbox{if  $0<k\leq\displaystyle\frac{m+1}{2}$},\\
&&D_{k,+}(\nu)=\left\{ \mathbf{S}_{2(i-j)-1}(j,j-k+1)\,\bigg|\, i\geq\frac{m+1}{2},\;0\leq j\leq \frac{m-1}{2}+k-1 \right\},  \;\; \mbox{if  $\displaystyle-\frac{m+1}{2}<k\leq 0$}.
\een
Then we have
\ben
\bigcup_{k\in\mathbb{Z}}\left(\widetilde{\mathcal{S}}_{2k-1}(\nu)\right)^c=\left(\bigcup_{k\in\mathbb{Z}} C_{k,-}(\nu)\cup C_{k,+}(\nu)\right)\bigcup \left(\bigcup_{0<k\leq\frac{m+1}{2}} D_{k,-}(\nu)\right)\bigcup \left(\bigcup_{-\frac{m+1}{2}<k\leq0}D_{k,+}(\nu)\right).
\een
Since we have
\ben
&&\bigcup_{k\leq0}C_{k,-}(\nu)=\bigcup_{l\geq0 }\left\{\mathbf{S}_{-2l-1}(i,i-k)\,\bigg|\,0\leq i\leq\frac{m-1}{2}, \,k\leq0 \right\}\\
&&\bigcup_{k>0}C_{k,-}(\nu)=\bigcup_{0\leq l\leq \frac{m-3}{2}}\left\{\mathbf{S}_{-2l-1}(i+k,i)\,\bigg|\,0\leq i\leq\frac{m-3}{2}-l, \,k>0 \right\}
\\
&&\bigcup_{0<k\leq\frac{m+1}{2}} D_{k,-}(\nu)=\bigcup_{1\leq l\leq \frac{m-1}{2}}\left\{\mathbf{S}_{-2l-1}(i+k,i)\,\bigg|\,\frac{m-1}{2}-l\leq i\leq\frac{m-1}{2}-k, \,0< k\leq l \right\}\\
&&\quad\quad\quad\quad\quad\quad\quad\quad\quad\bigcup\bigcup_{l\geq \frac{m+1}{2}}\left\{\mathbf{S}_{-2l-1}(i+k,i)\,\bigg|\,0\leq i\leq\frac{m-1}{2}-k, \,0<k\leq\frac{m+1}{2} \right\}
\een
and
\ben
&&\bigcup_{k\leq0}C_{k,+}(\nu)=\bigcup_{1\leq l\leq\frac{m-1}{2} }\left\{\mathbf{S}_{2l-1}(j,j-k+1)\,\bigg|\,0\leq j\leq\frac{m-1}{2}-l, \,k\leq0 \right\}\\
&&\bigcup_{k>0}C_{k,+}(\nu)=\bigcup_{l\geq 1}\left\{\mathbf{S}_{2l-1}(j+k-1,j)\,\bigg|\,0\leq j\leq\frac{m-1}{2}, \,k>0 \right\}
\\
&&\bigcup_{-\frac{m+1}{2}<k\leq0} D_{k,+}(\nu)=\bigcup_{2\leq l\leq \frac{m+1}{2}}\left\{\mathbf{S}_{2l-1}(j,j-k+1)\,\bigg|\,\frac{m+1}{2}-l\leq j\leq\frac{m-3}{2}+k, \,2-l \leq k\leq 0 \right\}\\
&&\quad\quad\quad\quad\quad\quad\quad\quad\quad\quad\bigcup\bigcup_{l> \frac{m+1}{2}}\left\{\mathbf{S}_{2l-1}(j,j-k+1)\,\bigg|\,0\leq j\leq\frac{m-3}{2}+k, \,-\frac{m+1}{2}<k\leq0 \right\}.
\een
then 
\ben
&&\left(\bigcup_{k\in\mathbb{Z}}C_{k,-}(\nu)\right)\bigcup \left(\bigcup_{0<k\leq\frac{m+1}{2}} D_{k,-}(\nu)\right)\\
&=&\left(\bigcup_{l\leq-\frac{m+1}{2}}\left\{ \mathbf{S}_{-2l-1}(i,j)\,\bigg|\,i\geq\frac{m+1}{2}, \,j\geq 0 \right\}^c\right)\bigcup \left(\bigcup_{-\frac{m+1}{2}<l\leq0}\left\{ \mathbf{S}_{-2l-1}(i,j)\,\bigg|\,i\geq\frac{m+1}{2},\, j\geq \frac{m-1}{2}-l \right\}^c\right)
\een
and	
\ben
&&\left(\bigcup_{k\in\mathbb{Z}} C_{k,+}(\nu)\right)\bigcup \left(\bigcup_{-\frac{m+1}{2}<k\leq0}D_{k,+}(\nu)\right)\\
&=&\left(\bigcup_{l\geq\frac{m+1}{2}}\left\{ \mathbf{S}_{2l-1}(i,j)\,\bigg|\,i\geq0,\, j\geq \frac{m+1}{2} \right\}^c\right)\bigcup \left(\bigcup_{0<l<\frac{m+1}{2}}\left\{ \mathbf{S}_{2l-1}(i,j)\,\bigg|\,i\geq\frac{m+1}{2}-l, \,j\geq \frac{m+1}{2} \right\}^c\right)
\een

Then the proof is completed by combining all the above results with the definition of $\mathcal{S}_{2k-1}(\nu)$.
\end{proof}

\begin{lemma}\label{even-odd-symmetry}
	Assume  $\nu=(m,m-1,\cdots,2,1)$ with $m\in\mathbb{Z}_{\geq1}$. Then we  have 
	\ben
	\bigcup_{k\in\mathbb{Z}}\left(\widetilde{\mathcal{S}}_{k}(\nu)\right)^c=\bigcup_{k\in\mathbb{Z}}\left(\mathcal{S}_{k}(\nu)\right)^c.
	\een
\end{lemma}
\begin{proof}
It follows from Lemma \ref{even-symmetry} and Lemma \ref{odd-symmetry}.
\end{proof}
Then we have
\begin{lemma}\label{symmetric-interlacing}
Assume  $\nu=(m,m-1,\cdots,2,1)$ with $m\in\mathbb{Z}_{\geq1}$. 
For any pyramid partition $\pi\in\mathfrak{P}$, we have  $\widetilde{\bm{\wp}}_{\pi}(\nu)=\bm{\wp}_{\pi}(\nu)$.
Hence we have  $\mathfrak{RP}_{+}(\nu)=\mathfrak{RP}_{-}(\nu)$.
\end{lemma}
\begin{proof}
Since for any  pyramid partition $\pi\in\mathfrak{P}$, we have
\ben
&&\widetilde{\bm{\wp}}_{\pi}(\nu)=\bigcup_{k\in\mathbb{Z}}\widetilde{\mathcal{S}}_{k}(\nu)\cap\widetilde{\pi}_{k}\subseteq\pi,\\
&&\bm{\wp}_{\pi}(\nu)=\bigcup_{k\in\mathbb{Z}}\mathcal{S}_{k}(\nu)\cap\pi_{k}\subseteq\pi.
\een
By Lemma \ref{even-odd-symmetry}, we have  
\ben
\bigcup\limits_{k\in\mathbb{Z}}\left(\widetilde{\mathcal{S}}_{k}(\nu)\right)^c\bigcap\widetilde{\pi}_{k}&=&\left(\bigcup\limits_{k\in\mathbb{Z}}\left(\widetilde{\mathcal{S}}_{k}(\nu)\right)^c\right)\bigcap\left(\bigcup\limits_{k\in\mathbb{Z}}\widetilde{\pi}_{k}\right)\\
&=&\left(\bigcup_{k\in\mathbb{Z}}\left(\mathcal{S}_{k}(\nu)\right)^c\right)\bigcap\left(\bigcup\limits_{k\in\mathbb{Z}}\pi_{k}\right)=\bigcup_{k\in\mathbb{Z}}\left(\mathcal{S}_{k}(\nu)\right)^c\bigcap\pi_{k}.
\een
Then we have
\ben
[B]\in\widetilde{\bm{\wp}}_{\pi}(\nu)\Longrightarrow [B]\notin\bigcup\limits_{k\in\mathbb{Z}}\left(\widetilde{\mathcal{S}}_{k}(\nu)\right)^c\bigcap\widetilde{\pi}_{k}
\Longrightarrow[B]\notin\bigcup_{k\in\mathbb{Z}}\left(\mathcal{S}_{k}(\nu)\right)^c\bigcap\pi_{k}\Longrightarrow[B]\in\bm{\wp}_{\pi}(\nu).
\een
where  $[B]$ is any brick  in  $\widetilde{\bm{\wp}}_{\pi}(\nu)$ and the last implication can be proved by negation.
Then  $\widetilde{\bm{\wp}}_{\pi}(\nu)\subseteq\bm{\wp}_{\pi}(\nu)$. Similarly, we have $\bm{\wp}_{\pi}(\nu)\subseteq\widetilde{\bm{\wp}}_{\pi}(\nu)$. Then the proof is completed.
\end{proof}	

Next, we will show	if $\widetilde{\bm{\wp}}_{\pi}(\nu)=\bm{\wp}_{\pi}(\nu)$  holds for
 any pyramid partition $\pi\in\mathfrak{P}$ where $\nu\neq\emptyset$, then we must have
 $\nu=(m,m-1,\cdots,2,1)$  where  $m\in\mathbb{Z}_{\geq1}$ is the length of $\nu$. We need the following steps.
\begin{lemma}\label{symmetric-step1}
Let $\nu\neq\emptyset$ be a  partition. Assume that 
\ben
\bigcup_{k\in\mathbb{Z}}\left(\widetilde{\mathcal{S}}_{2k}(\nu)\right)^c=\bigcup_{k\in\mathbb{Z}}\left(\mathcal{S}_{2k}(\nu)\right)^c,
\een
then we have $\rho_1(\nu)=\rho_2(\nu)\geq1$.
\end{lemma}

\begin{proof}
Assume that $\rho_1(\nu)\neq\rho_2(\nu)$. Without loss of generality, we 
suppose $\rho_2(\nu)>\rho_1(\nu)$, the other case takes the similar argument. Let 
\ben
\rho_2(\nu)=\rho_1(\nu)+\ell, \;\;\;\ell\in\mathbb{Z}_{\geq1}
\een
By assumption and Remark \ref{empty-case1}, we have 
\bea\label{S_0+=S_0-}
\left(\mathcal{S}_{0}(\nu)\right)^c=\mathcal{S}_{0}(\emptyset)\bigcap\left(\displaystyle\bigcup\limits_{k\in\mathbb{Z}}\left(\widetilde{\mathcal{S}}_{2k}(\nu)\right)^c\right).
\eea
Since if $k\leq0$,
\ben
\left(\widetilde{\mathcal{S}}_{2k}(\nu)\right)^c&=&\left\{ \widetilde{\mathbf{S}}_{2k}(i,j)\,\bigg|\,i\geq0,\; 0\leq j\leq \rho_{2}(\nu)-\varepsilon_1(\nu;-k-1)-1\right\}\\
&&\bigcup\left\{ \widetilde{\mathbf{S}}_{2k}(i,j)\,\bigg|\,0\leq i\leq \rho_{1}(\nu)-\varepsilon_2(\nu;-k-1)-1,\;  j\geq \rho_{2}(\nu)-\varepsilon_1(\nu;-k-1)\right\}
\een
and if $k>0$,
\ben
\left(\widetilde{\mathcal{S}}_{2k}(\nu)\right)^c&=&\left\{ \widetilde{\mathbf{S}}_{2k}(i,j)\,\bigg|\,i\geq0,\; 0\leq j\leq \rho_{2}(\nu)-\varepsilon_3(\nu;k)-1\right\}\\
&&\bigcup\left\{ \widetilde{\mathbf{S}}_{2k}(i,j)\,\bigg|\,0\leq i\leq \rho_{1}(\nu)-\varepsilon_4(\nu;k)-1,\;  j\geq \rho_{2}(\nu)-\varepsilon_3(\nu;k)\right\},
\een
then by Lemma \ref{antidiagonal-diagonal}, we have
\bea\label{S_0+}
&&\mathcal{S}_{0}(\emptyset)\bigcap\left(\displaystyle\bigcup\limits_{k\in\mathbb{Z}}\left(\widetilde{\mathcal{S}}_{2k}(\nu)\right)^c\right)\\
&=&\bigcup_{k\leq0}\left\{ \mathbf{S}_{0}(j,j-k)\,\bigg|\,0\leq j\leq \rho_{2}(\nu)-\varepsilon_1(\nu;-k-1)-1\right\}\nonumber\\
&&\bigcup\bigcup_{k\leq 0}\left\{ \mathbf{S}_{0}(j,j-k)\,\bigg|\,\rho_{2}(\nu)-\varepsilon_1(\nu;-k-1)\leq j\leq \rho_{1}(\nu)-\varepsilon_2(\nu;-k-1)-1\right\}\nonumber\\
&&\bigcup\bigcup_{k> 0}\left\{ \mathbf{S}_{0}(j+k,j)\,\bigg|\, 0\leq j\leq \rho_{2}(\nu)-\varepsilon_3(\nu;k)-1\right\}\nonumber\\
&&\bigcup\bigcup_{k> 0}\left\{\mathbf{S}_{0}(j+k,j)\,\bigg|\, \rho_{2}(\nu)-\varepsilon_3(\nu;k)\leq j\leq \rho_{1}(\nu)-\varepsilon_4(\nu;k)-1\right\}.\nonumber
\eea
On the other hand, we have
\bea\label{S_0-}
\left(\mathcal{S}_{0}(\nu)\right)^c&=&\left\{ \mathbf{S}_{0}(i,j)\,\bigg|\,i\geq\rho_1(\nu),\; j\geq \rho_2(\nu) \right\}^c\\
&=&\bigcup_{k<-\ell}\left\{\mathbf{S}_{0}(j,j-k)\,\bigg|\,0\leq j\leq \rho_{1}(\nu)-1\right\}\nonumber\\
&&\bigcup\bigcup_{-\ell\leq k\leq0}\left\{ \mathbf{S}_{0}(j,j-k)\,\bigg|\,0\leq j\leq \rho_{2}(\nu)+k-1\right\}\nonumber\\
&&\bigcup\bigcup_{k>0}\left\{ \mathbf{S}_{0}(j+k,j)\,\bigg|\,0\leq j\leq \rho_{2}(\nu)-1\right\}.\nonumber
\eea
By \eqref{S_0+=S_0-},  \eqref{S_0+} and \eqref{S_0-}, one can show by induction  on $k\leq0$ and $k>0$ that
\ben
\mbox{either}\;\;\;\;
\left\{
\begin{aligned}
	&\varepsilon_1(\nu;\imath-1)=\imath, \;\;\forall\;1\leq\imath\leq\ell, \\
	&\varepsilon_1(\nu;+\infty)=\ell, \\
	&\varepsilon_3(\nu;+\infty)=0,
\end{aligned}
\right.\;\;\;\;\mbox{or }\;\;\;\;
\left\{
\begin{aligned}
	&\varepsilon_1(\nu;\imath-1)=\imath, \;\;\;\;\;\;\forall\;1\leq\imath\leq\ell, \\
	&\varepsilon_1(\nu;+\infty)\geq\ell+1, \\
	&\varepsilon_2(\nu;+\infty)=0,\\
	&\varepsilon_3(\nu;+\infty)=0.
\end{aligned}
\right.
\een
For the first case, we have $\rho_2(\nu)=\ell$ and $\rho_1(\nu)=0$, and hence $\varepsilon_2(\nu;+\infty)=\varepsilon_4(\nu;+\infty)=0$. Then
\ben
\left\{
\begin{aligned}
	&\nu^\prime(t)=1,\;\; \quad\quad\quad \;\;\forall\, t<0,\\
	&\nu^\prime(2t+1)=-1, \,\;\;\;\forall\; t\geq0, \\
	&\nu^\prime(2t)=1,\;\;\quad\quad\quad \forall\; 0\leq t\leq \ell-1,\\
	&\nu^\prime(2t)=-1, \quad\quad\quad \forall\, t\geq \ell,
\end{aligned}
\right.
\een
which implies $|\mathit{S}(\nu^\prime)_{+}|=\ell=\rho_2(\nu)$ and  $|\mathit{S}(\nu^\prime)_{-}|=0=\rho_1(\nu)$, which contradicts $|\mathit{S}(\nu^\prime)_{+}|=|\mathit{S}(\nu^\prime)_{-}|$ by Remark \ref{charge0}.
Similarly, for the second case, we have $|\mathit{S}(\nu^\prime)_{+}|=\rho_2(\nu)$ and $|\mathit{S}(\nu^\prime)_{-}|=\rho_1(\nu)$, which again contradicts $|\mathit{S}(\nu^\prime)_{+}|=|\mathit{S}(\nu^\prime)_{-}|$ by Remark \ref{charge0}. Now the proof is completed by Remark \ref{empty-case}.
\end{proof}

The second step is 
\begin{lemma}\label{symmetric-step2}
	Let $\nu\neq\emptyset$ be a  partition. Assume that 
	\ben
	\bigcup_{k\in\mathbb{Z}}\left(\widetilde{\mathcal{S}}_{2k}(\nu)\right)^c=\bigcup_{k\in\mathbb{Z}}\left(\mathcal{S}_{2k}(\nu)\right)^c,
	\een
	then 
	\ben
	\mbox{either}\;\;\;\;
	\left\{
	\begin{aligned}
		&\varepsilon_1(\nu;+\infty)\geq1,  \\
		&\varepsilon_2(\nu;+\infty)=0, \\
		&\varepsilon_3(\nu;+\infty)=0,\\
		&\varepsilon_4(\nu;+\infty)\geq1.
	\end{aligned}
	\right.
	\;\;\;\;\mbox{or }\;\;\;\;
	\left\{
	\begin{aligned}
		&\varepsilon_1(\nu;+\infty)=0,  \\
		&\varepsilon_2(\nu;+\infty)\geq1, \\
		&\varepsilon_3(\nu;+\infty)\geq1,\\
		&\varepsilon_4(\nu;+\infty)=0.
	\end{aligned}
	\right.
	\een
\end{lemma}

\begin{proof}
By Lemma \ref{symmetric-step1}, we have $\rho_1(\nu)=\rho_2(\nu)\geq1$. 
Now we have 
\bea\label{newS_0-}
\;\;\;\;\;\;\quad\left(\mathcal{S}_{0}(\nu)\right)^c
=\bigcup_{k\leq0}\left\{\mathbf{S}_{0}(j,j-k)\,\bigg|\,0\leq j\leq \rho_{1}(\nu)-1\right\}\bigcup\bigcup_{k>0}\left\{ \mathbf{S}_{0}(j+k,j)\,\bigg|\,0\leq j\leq \rho_{2}(\nu)-1\right\}.
\eea
Then the proof follows from $\rho_1(\nu)=\rho_2(\nu)\geq1$ and \eqref{S_0+=S_0-},  \eqref{S_0+}, \eqref{newS_0-} by induction  on $k\leq0$ and $k>0$.
\end{proof}

The third step is 
\begin{lemma}\label{symmetric-step3}
	Let $\nu\neq\emptyset$ be a  partition. Assume that 
	\ben
	\bigcup_{k\in\mathbb{Z}}\left(\widetilde{\mathcal{S}}_{2k}(\nu)\right)^c=\bigcup_{k\in\mathbb{Z}}\left(\mathcal{S}_{2k}(\nu)\right)^c,
	\een
	then we have 
	\ben
	\mbox{either}\;\;\;\;
	\left\{
	\begin{aligned}
		&\varepsilon_1(\nu;0)=1,  \\
		&\varepsilon_2(\nu;+\infty)=0, \\
		&\varepsilon_3(\nu;+\infty)=0,\\
		&\varepsilon_4(\nu;1)=1.
	\end{aligned}
	\right.
	\;\;\;\;\mbox{or }\;\;\;\;
	\left\{
	\begin{aligned}
		&\varepsilon_1(\nu;+\infty)=0,  \\
		&\varepsilon_2(\nu;0)=1, \\
		&\varepsilon_3(\nu;1)=1,\\
		&\varepsilon_4(\nu;+\infty)=0.
	\end{aligned}
	\right.
	\een
\end{lemma}

\begin{proof}
By assumption and Remark \ref{empty-case1}, we have 	
\bea\label{S_-2+=S_-2-}
\left(\mathcal{S}_{-2}(\nu)\right)^c=\mathcal{S}_{-2}(\emptyset)\bigcap\left(\displaystyle\bigcup\limits_{k\in\mathbb{Z}}\left(\widetilde{\mathcal{S}}_{2k}(\nu)\right)^c\right),
\eea
and 
\bea\label{S_2+=S_2-}
\left(\mathcal{S}_{2}(\nu)\right)^c=\mathcal{S}_{2}(\emptyset)\bigcap\left(\displaystyle\bigcup\limits_{k\in\mathbb{Z}}\left(\widetilde{\mathcal{S}}_{2k}(\nu)\right)^c\right).
\eea	
Since if $k\leq0$,
\ben
\left(\widetilde{\mathcal{S}}_{2k}(\nu)\right)^c&=&\left\{ \widetilde{\mathbf{S}}_{2k}(i,j)\,\bigg|\,0\leq i\leq \rho_{1}(\nu)-\varepsilon_2(\nu;-k-1)-1,\;  j\geq0\right\}\\
&&\bigcup\left\{ \widetilde{\mathbf{S}}_{2k}(i,j)\,\bigg|\, i\geq \rho_{1}(\nu)-\varepsilon_2(\nu;-k-1),\; 0\leq j\leq \rho_{2}(\nu)-\varepsilon_1(\nu;-k-1)-1\right\}
\een
and if $k>0$,
\ben
\left(\widetilde{\mathcal{S}}_{2k}(\nu)\right)^c&=&\left\{ \widetilde{\mathbf{S}}_{2k}(i,j)\,\bigg|\,0\leq i\leq \rho_{1}(\nu)-\varepsilon_4(\nu;k)-1,\; j\geq0\right\}\\
&&\bigcup\left\{ \widetilde{\mathbf{S}}_{2k}(i,j)\,\bigg|\, i\geq\rho_{1}(\nu)-\varepsilon_4(\nu;k) ,\;  0\leq j\leq \rho_{2}(\nu)-\varepsilon_3(\nu;k)-1\right\},
\een
then by Lemma \ref{antidiagonal-diagonal}, we have
\bea\label{S_-2+}
&&\mathcal{S}_{-2}(\emptyset)\bigcap\left(\displaystyle\bigcup\limits_{k\in\mathbb{Z}}\left(\widetilde{\mathcal{S}}_{2k}(\nu)\right)^c\right)\\
&=&\bigcup_{k\leq0}\left\{ \mathbf{S}_{-2}(i,i-k)\,\bigg|\,0\leq i\leq \rho_{1}(\nu)-\varepsilon_2(\nu;-k-1)-1\right\}\nonumber\\
&&\bigcup\bigcup_{k\leq 0}\left\{ \mathbf{S}_{-2}(i,i-k)\,\bigg|\,\rho_{1}(\nu)-\varepsilon_2(\nu;-k-1)\leq i\leq \rho_{2}(\nu)-\varepsilon_1(\nu;-k-1)-2\right\}\nonumber\\
&&\bigcup\bigcup_{k> 0}\left\{ \mathbf{S}_{-2}(i+k,i)\,\bigg|\, 0\leq i\leq \rho_{1}(\nu)-\varepsilon_4(\nu;k)-1\right\}\nonumber\\
&&\bigcup\bigcup_{k> 0}\left\{\mathbf{S}_{-2}(i+k,i)\,\bigg|\, \rho_{1}(\nu)-\varepsilon_4(\nu;k)\leq i\leq \rho_{2}(\nu)-\varepsilon_3(\nu;k)-2\right\}\nonumber.
\eea
On the other hand, we have	
\ben
\left(\mathcal{S}_{-2}(\nu)\right)^c=\left\{ \mathbf{S}_{-2}(i,j)\,\bigg|\,i\geq\rho_1(\nu)-\varepsilon_2(\nu;0),\; j\geq \rho_2(\nu)-\varepsilon_1(\nu;0) \right\}^c\\
\een	
Notice  $\rho_1(\nu)=\rho_2(\nu)\geq1$ by Lemma \ref{symmetric-step1} and  $\varepsilon_1(\nu;0), \varepsilon_2(\nu;0)\in\{0,1\}$. Then there are 3 cases:\\
(i) If $\varepsilon_1(\nu;0)=\varepsilon_2(\nu;0)=0$, then
\bea\label{S_-2-1}
\quad\quad \left(\mathcal{S}_{-2}(\nu)\right)^c
=\bigcup_{k\leq0}\left\{\mathbf{S}_{-2}(i,i-k)\,\bigg|\,0\leq i\leq \rho_{1}(\nu)-1\right\}\bigcup\bigcup_{k>0}\left\{ \mathbf{S}_{-2}(i+k,i)\,\bigg|\,0\leq i\leq \rho_{2}(\nu)-1\right\}
\eea
By \eqref{S_-2+=S_-2-}, \eqref{S_-2+} and  \eqref{S_-2-1}, one can show by induction  on $k\leq0$ and $k>0$ that
\ben
\varepsilon_2(\nu;+\infty)=0;\;\;\;\;\varepsilon_4(\nu;+\infty)=0.
\een
Then we  have $\rho_1(\nu)=0$, which is a contradiction to 	Lemma \ref{symmetric-step1}. \\ 
(ii) If $\varepsilon_1(\nu;0)=1$ and $\varepsilon_2(\nu;0)=0$, then  
\bea\label{S_-2-2}
\;\;\;\quad\left(\mathcal{S}_{-2}(\nu)\right)^c
=\bigcup_{k\leq0}\left\{\mathbf{S}_{-2}(i,i-k)\,\bigg|\,0\leq i\leq \rho_{1}(\nu)-1\right\}\bigcup\bigcup_{k>0}\left\{ \mathbf{S}_{-2}(i+k,i)\,\bigg|\,0\leq i\leq \rho_{2}(\nu)-2\right\}
\eea
By Lemma \ref{symmetric-step2}, we have $\varepsilon_2(\nu;+\infty)=0$ and 
$\varepsilon_3(\nu;+\infty)=0$ in this case since $\varepsilon_1(\nu;+\infty)\geq1$. Then it follows from  \eqref{S_-2+=S_-2-}, \eqref{S_-2+} and  \eqref{S_-2-2} by induction  on $k\leq0$ and $k>0$ that $\varepsilon_4(\nu;1)=1$.\\
(iii) If $\varepsilon_1(\nu;0)=0$ and $\varepsilon_2(\nu;0)=1$, then
\bea\label{S_-2-3}
\quad\;\;\;\; \left(\mathcal{S}_{-2}(\nu)\right)^c
=\bigcup_{k<0}\left\{\mathbf{S}_{-2}(i,i-k)\,\bigg|\,0\leq i\leq \rho_{1}(\nu)-2\right\}\bigcup\bigcup_{k\geq0}\left\{ \mathbf{S}_{-2}(i+k,i)\,\bigg|\,0\leq i\leq \rho_{2}(\nu)-1\right\}
\eea	
By Lemma \ref{symmetric-step2}, we have $\varepsilon_1(\nu;+\infty)=0$ and 
$\varepsilon_4(\nu;+\infty)=0$ in this case since $\varepsilon_2(\nu;+\infty)\geq1$. 	
Notice that one can not obtain the value of  $\varepsilon_3(\nu;1)$  from  \eqref{S_-2+=S_-2-}, \eqref{S_-2+} and  \eqref{S_-2-3} by induction  on $k\leq0$ and $k>0$.
Therefore it follows from the above argument that 
\ben
\mbox{either}\;\;\;\;
\left\{
\begin{aligned}
	&\varepsilon_1(\nu;0)=1,  \\
	&\varepsilon_2(\nu;+\infty)=0, \\
	&\varepsilon_3(\nu;+\infty)=0,\\
	&\varepsilon_4(\nu;1)=1.
\end{aligned}
\right.
\;\;\;\;\mbox{or }\;\;\;\;
\left\{
\begin{aligned}
	&\varepsilon_1(\nu;+\infty)=0,  \\
	&\varepsilon_2(\nu;0)=1, \\
	&\varepsilon_4(\nu;+\infty)=0.
\end{aligned}
\right.
\een
In addition, we need to consider the equation \eqref{S_2+=S_2-}. It follows from the similar argument above that 
\ben
\mbox{either}\;\;\;\;
\left\{
\begin{aligned}
	&\varepsilon_3(\nu;1)=1,\\
	&\varepsilon_4(\nu;+\infty)=0,\\
	&\varepsilon_1(\nu;+\infty)=0.
\end{aligned}
\right.
\;\;\;\;\mbox{or }\;\;\;\;
\left\{
\begin{aligned}
	&\varepsilon_3(\nu;+\infty)=0,\\
	&\varepsilon_4(\nu;1)=1\\
	&\varepsilon_1(\nu;0)=1,  \\
	&\varepsilon_2(\nu;+\infty)=0.
\end{aligned}
\right.
\een
Now the proof follows from the above results.
\end{proof}

The fourth step is 
\begin{lemma}\label{symmetric-step4}
	Let $\nu\neq\emptyset$ be a  partition. Assume that 
	\ben
	\bigcup_{k\in\mathbb{Z}}\left(\widetilde{\mathcal{S}}_{2k}(\nu)\right)^c=\bigcup_{k\in\mathbb{Z}}\left(\mathcal{S}_{2k}(\nu)\right)^c,
	\een
	then for any $1\leq \flat\leq\rho_1(\nu)=\rho_2(\nu)$, we have 
	\ben
	\mbox{either}\;\;\;\;
	\left\{
	\begin{aligned}
		&\varepsilon_1(\nu;\flat-1)=\flat,  \\
		&\varepsilon_2(\nu;+\infty)=0, \\
		&\varepsilon_3(\nu;+\infty)=0,\\
		&\varepsilon_4(\nu;\flat)=\flat.
	\end{aligned}
	\right.
	\;\;\;\;\mbox{or }\;\;\;\;
	\left\{
	\begin{aligned}
		&\varepsilon_1(\nu;+\infty)=0,  \\
		&\varepsilon_2(\nu;\flat-1)=\flat, \\
		&\varepsilon_3(\nu;\flat)=\flat,\\
		&\varepsilon_4(\nu;+\infty)=0.
	\end{aligned}
	\right.
	\een
\end{lemma}

\begin{proof}
Let $\rho(\nu):=\rho_1(\nu)=\rho_2(\nu)$.
Since the statement holds for $\flat=1$ or $\rho_(\nu)=1$ by Lemma \ref{symmetric-step3}, we will complete the proof by induction on $\flat$ for $\rho(\nu)>1$. Assume that the statement holds for $1\leq\flat\leq\widetilde{\ell}\leq\rho(\nu)-1$, i.e., 
	\ben
	\mbox{either}\;\;\;\;
	\left\{
	\begin{aligned}
		&\varepsilon_1\left(\nu;\flat-1\right)=\flat,  \\
		&\varepsilon_2(\nu;+\infty)=0, \\
		&\varepsilon_3(\nu;+\infty)=0,\\
		&\varepsilon_4\left(\nu;\flat\right)=\flat,
	\end{aligned}
	\right. \;\;\forall\,1\leq\flat\leq\widetilde{\ell},\;\;\;(*)
	\;\;\;\;\mbox{or }\;\;\;\;
	\left\{
	\begin{aligned}
		&\varepsilon_1(\nu;+\infty)=0,  \\
		&\varepsilon_2\left(\nu;\flat-1\right)=\flat, \\
		&\varepsilon_3\left(\nu;\flat\right)=\flat,\\
		&\varepsilon_4(\nu;+\infty)=0,
	\end{aligned}
	\right. \;\;\forall\,1\leq\flat\leq\widetilde{\ell},\;\;\;(**)
	\een
we will prove that the statement also holds for $\flat=\widetilde{\ell}+1$. By assumption and Remark \ref{empty-case1}, we have 
\bea\label{S_2l+=S_2l-}
\left(\mathcal{S}_{2(\widetilde{\ell}+1)}(\nu)\right)^c=\mathcal{S}_{2(\widetilde{\ell}+1)}(\emptyset)\bigcap\left(\displaystyle\bigcup\limits_{k\in\mathbb{Z}}\left(\widetilde{\mathcal{S}}_{2k}(\nu)\right)^c\right).
\eea	
and 
\bea\label{S_-2l+=S_-2l-}
\left(\mathcal{S}_{-2(\widetilde{\ell}+1)}(\nu)\right)^c=\mathcal{S}_{-2(\widetilde{\ell}+1)}(\emptyset)\bigcap\left(\displaystyle\bigcup\limits_{k\in\mathbb{Z}}\left(\widetilde{\mathcal{S}}_{2k}(\nu)\right)^c\right),
\eea
Now we have
\ben
&&\mathcal{S}_{2(\widetilde{\ell}+1)}(\emptyset)\bigcap\left(\displaystyle\bigcup\limits_{k\in\mathbb{Z}}\left(\widetilde{\mathcal{S}}_{2k}(\nu)\right)^c\right)\\
&=&\bigcup_{k\leq0}\left\{ \mathbf{S}_{2(\widetilde{\ell}+1)}(j,j-k)\,\bigg|\,0\leq j\leq \rho_{2}(\nu)-\varepsilon_1(\nu;-k-1)-1\right\}\nonumber\\
&&\bigcup\bigcup_{k\leq 0}\left\{ \mathbf{S}_{2(\widetilde{\ell}+1)}(j,j-k)\,\bigg|\,\rho_{2}(\nu)-\varepsilon_1(\nu;-k-1)\leq j\leq \rho_{1}(\nu)-\varepsilon_2(\nu;-k-1)-\widetilde{\ell}-2\right\}\nonumber\\
&&\bigcup\bigcup_{k> 0}\left\{ \mathbf{S}_{2(\widetilde{\ell}+1)}(j+k,j)\,\bigg|\, 0\leq j\leq \rho_{2}(\nu)-\varepsilon_3(\nu;k)-1\right\}\nonumber\\
&&\bigcup\bigcup_{k> 0}\left\{\mathbf{S}_{2(\widetilde{\ell}+1)}(j+k,j)\,\bigg|\, \rho_{2}(\nu)-\varepsilon_3(\nu;k)\leq j\leq \rho_{1}(\nu)-\varepsilon_4(\nu;k)-\widetilde{\ell}-2\right\}\nonumber.
\een
In the case $(*)$, we have
\bea\label{S_2l+}
&&\mathcal{S}_{2(\widetilde{\ell}+1)}(\emptyset)\bigcap\left(\displaystyle\bigcup\limits_{k\in\mathbb{Z}}\left(\widetilde{\mathcal{S}}_{2k}(\nu)\right)^c\right)\\
&=&\bigcup_{k\leq0}\left\{ \mathbf{S}_{2(\widetilde{\ell}+1)}(j,j-k)\,\bigg|\,0\leq j\leq \rho_{2}(\nu)-\varepsilon_1(\nu;-k-1)-1\right\}\nonumber\\
&&\bigcup\bigcup_{k\leq 0}\left\{ \mathbf{S}_{2(\widetilde{\ell}+1)}(j,j-k)\,\bigg|\,\rho_{2}(\nu)-\varepsilon_1(\nu;-k-1)\leq j\leq \rho_{1}(\nu)-\widetilde{\ell}-2\right\}\nonumber\\
&&\bigcup\bigcup_{k> 0}\left\{ \mathbf{S}_{2(\widetilde{\ell}+1)}(j+k,j)\,\bigg|\, 0\leq j\leq \rho_{2}(\nu)-1\right\}\nonumber.
\eea
and 
\bea\label{S_2l-}
\left(\mathcal{S}_{2(\widetilde{\ell}+1)}(\nu)\right)^c=\left\{ \mathbf{S}_{2(\widetilde{\ell}+1)}(i,j)\,\bigg|\,i\geq\rho_1(\nu)-\varepsilon_4\left(\nu;\widetilde{\ell}+1\right),\; j\geq \rho_2(\nu) \right\}^c
\eea
where $\varepsilon_4\left(\nu;\widetilde{\ell}+1\right)=\widetilde{\ell}$ or $\varepsilon_4\left(\nu;\widetilde{\ell}+1\right)=\widetilde{\ell}+1$ since $\varepsilon_4\left(\nu;\widetilde{\ell}\right)=\widetilde{\ell}$.

If $\varepsilon_4\left(\nu;\widetilde{\ell}+1\right)=\widetilde{\ell}$, then it follows from \eqref{S_2l+=S_2l-}, \eqref{S_2l+}, \eqref{S_2l-} by induction on $k\leq0$ and $k>0$ that $\varepsilon_1(\nu;+\infty)=\widetilde{\ell}$. Then we have $\rho_2(\nu)=\widetilde{\ell}$ which contradicts  the assumption $\widetilde{\ell}\leq\rho_2(\nu)-1$. Hence we have $\varepsilon_4\left(\nu;\widetilde{\ell}+1\right)=\widetilde{\ell}+1$.
Now we have 
\bea\label{S_2l-1}
\left(\mathcal{S}_{2(\widetilde{\ell}+1)}(\nu)\right)^c&=&\bigcup_{k\leq-\widetilde{\ell}-1}\left\{\mathbf{S}_{2(\widetilde{\ell}+1)}(j,j-k)\,\bigg|\,0\leq j\leq \rho_{1}(\nu)-\widetilde{\ell}-2\right\}\\
&&\bigcup\bigcup_{-\widetilde{\ell}\leq k\leq0}\left\{ \mathbf{S}_{2(\widetilde{\ell}+1)}(j,j-k)\,\bigg|\,0\leq j\leq \rho_{1}(\nu)+k-1\right\}\nonumber\\
&&\bigcup\bigcup_{k>0}\left\{ \mathbf{S}_{2(\widetilde{\ell}+1)}(j+k,j)\,\bigg|\,0\leq j\leq \rho_{2}(\nu)-1\right\}\nonumber
\eea
Then it follows from \eqref{S_2l+=S_2l-}, \eqref{S_2l+}, \eqref{S_2l-1} by induction on $k\leq0$ and $k>0$ that $\varepsilon_1\left(\nu;\widetilde{\ell}\right)=\widetilde{\ell}+1$. Then in  the case $(*)$, we have
\ben
\left\{
\begin{aligned}
	&\varepsilon_1\left(\nu;\widetilde{\ell}\right)=\widetilde{\ell}+1,  \\
	&\varepsilon_2(\nu;+\infty)=0, \\
	&\varepsilon_3(\nu;+\infty)=0,\\
	&\varepsilon_4\left(\nu;\widetilde{\ell}+1\right)=\widetilde{\ell}+1.
\end{aligned}
\right.
\een
Next, we consider the case $(**)$. It follows from the similar argument above that 
\ben
\left\{
\begin{aligned}
	&\varepsilon_1(\nu;+\infty)=0,  \\
	&\varepsilon_3\left(\nu;\widetilde{\ell}+1\right)=\widetilde{\ell}+1,\\
	&\varepsilon_4(\nu;+\infty)=0,
\end{aligned}
\right.
\een
where the value of $\varepsilon_2\left(\nu;\widetilde{\ell}\right)$ can not be determined by  the equation \eqref{S_2l+=S_2l-}. As in Lemma \ref{symmetric-step3}, we should take into account the equation \eqref{S_-2l+=S_-2l-} further. Applying the similar argument as above in case $(**)$, it follows from \eqref{S_-2l+=S_-2l-} that $\varepsilon_2\left(\nu;\widetilde{\ell}\right)=\widetilde{\ell}+1$. Then the statement is ture for $\flat=\widetilde{\ell}+1$, and  hence the proof is completed.
\end{proof}

\begin{remark}
The result in Lemma \ref{symmetric-step4} is compatible with the following condition  
\ben
\bigcup_{k\in\mathbb{Z}}\left(\widetilde{\mathcal{S}}_{2k-1}(\nu)\right)^c=\bigcup_{k\in\mathbb{Z}}\left(\mathcal{S}_{2k-1}(\nu)\right)^c.
\een
\end{remark}
Now we have 
\begin{proposition}\label{unique symmetric interlacing}
	Let $\nu\neq\emptyset$ be a  partition. Then  $\widetilde{\bm{\wp}}_{\pi}(\nu)=\bm{\wp}_{\pi}(\nu)$ holds for any pyramid  partition $\pi\in\mathfrak{P}$ if and only if $\nu=(m,m-1,\cdots,2,1)$ for some $m\in\mathbb{Z}_{\geq1}$.
\end{proposition}

\begin{proof}
Assume that $\widetilde{\bm{\wp}}_{\pi}(\nu)=\bm{\wp}_{\pi}(\nu)$ for any $\pi\in\mathfrak{P}$. Since
\ben
&&\widetilde{\bm{\wp}}_{\pi}(\nu)=\bigcup_{k\in\mathbb{Z}}\widetilde{\mathcal{S}}_{k}(\nu)\cap\widetilde{\pi}_{k}=\left(\bigcup_{k\in\mathbb{Z}}\widetilde{\mathcal{S}}_{k}(\nu)\right)\bigcap\left(\bigcup\limits_{k\in\mathbb{Z}}\widetilde{\pi}_{k}\right)=\left(\bigcup_{k\in\mathbb{Z}}\widetilde{\mathcal{S}}_{k}(\nu)\right)\bigcap\pi\subseteq\pi,\\
&&\bm{\wp}_{\pi}(\nu)=\bigcup_{k\in\mathbb{Z}}\mathcal{S}_{k}(\nu)\cap\pi_{k}=\left(\bigcup_{k\in\mathbb{Z}}\mathcal{S}_{k}(\nu)\right)\bigcap\left(\bigcup\limits_{k\in\mathbb{Z}}\pi_{k}\right)=\left(\bigcup_{k\in\mathbb{Z}}\mathcal{S}_{k}(\nu)\right)\bigcap\pi\subseteq\pi,
\een
then we have 
\bea\label{complement}
\left(\bigcup\limits_{k\in\mathbb{Z}}\left(\widetilde{\mathcal{S}}_{k}(\nu)\right)^c\right)\bigcap\pi=\left(\bigcup_{k\in\mathbb{Z}}\left(\mathcal{S}_{k}(\nu)\right)^c\right)\bigcap\pi,\;\;\;\mbox{ for any $\pi\in\mathfrak{P}$.}
\eea
Let $[B]$ be any brick in $\bigcup\limits_{k\in\mathbb{Z}}\left(\widetilde{\mathcal{S}}_{k}(\nu)\right)^c$. One can choose a pyramid partition $\pi$ such that $[B]\in\pi$, then 
$[B]\in\left(\bigcup\limits_{k\in\mathbb{Z}}\left(\widetilde{\mathcal{S}}_{k}(\nu)\right)^c\right)\bigcap\pi$. Combined with \eqref{complement}, we have $[B]\in\bigcup\limits_{k\in\mathbb{Z}}\left(\mathcal{S}_{k}(\nu)\right)^c$. Then we have $\bigcup\limits_{k\in\mathbb{Z}}\left(\widetilde{\mathcal{S}}_{k}(\nu)\right)^c\subseteq\bigcup\limits_{k\in\mathbb{Z}}\left(\mathcal{S}_{k}(\nu)\right)^c$. Similarly, we have $\bigcup\limits_{k\in\mathbb{Z}}\left(\mathcal{S}_{k}(\nu)\right)^c\subseteq\bigcup\limits_{k\in\mathbb{Z}}\left(\widetilde{\mathcal{S}}_{k}(\nu)\right)^c$. Hence
\ben
\bigcup_{k\in\mathbb{Z}}\left(\widetilde{\mathcal{S}}_{k}(\nu)\right)^c=\bigcup_{k\in\mathbb{Z}}\left(\mathcal{S}_{k}(\nu)\right)^c.
\een
By Lemma \ref{antidiagonal-diagonal}, we have
\ben
\bigcup_{k\in\mathbb{Z}}\left(\widetilde{\mathcal{S}}_{k}(\nu)\right)^c=\bigcup_{k\in\mathbb{Z}}\left(\mathcal{S}_{k}(\nu)\right)^c\Longleftrightarrow\left\{
\begin{aligned}
	&\bigcup_{k\in\mathbb{Z}}\left(\widetilde{\mathcal{S}}_{2k}(\nu)\right)^c=\bigcup_{k\in\mathbb{Z}}\left(\mathcal{S}_{2k}(\nu)\right)^c,  \\
	&\bigcup_{k\in\mathbb{Z}}\left(\widetilde{\mathcal{S}}_{2k-1}(\nu)\right)^c=\bigcup_{k\in\mathbb{Z}}\left(\mathcal{S}_{2k-1}(\nu)\right)^c.
\end{aligned}
\right.
\een	
By Lemma \ref{symmetric-step4}, we have either 	
	\bea\label{odd-symmetry1}
	\left\{
	\begin{aligned}
		&\varepsilon_1(\nu;\flat-1)=\flat,  \\
		&\varepsilon_2(\nu;+\infty)=0, \\
		&\varepsilon_3(\nu;+\infty)=0,\\
		&\varepsilon_4(\nu;\flat)=\flat.
	\end{aligned}
	\right.\;\forall\,1\leq\flat\leq\rho(\nu)&\Longleftrightarrow&	
	\left\{
	\begin{aligned}
		&\nu^\prime(2t)=1,\; \quad\quad\quad \,0\leq t\leq \rho(\nu)-1 \\
		&\nu^\prime(2t)=-1,\; \quad\quad\;\, t\geq \rho(\nu)\\
		&\nu^\prime(2t+1)=-1,  \quad t\geq0\\
		&\nu^\prime(-2t)=1, \;\quad\quad\;\; t\geq1\\
		&\nu^\prime(-2t+1)=-1,\; 1\leq t\leq \rho(\nu)\\
		&\nu^\prime(-2t+1)=1,\; \quad t\geq \rho(\nu)+1\\
	\end{aligned}
	\right.\\
	\nonumber&\Longleftrightarrow& \nu=(2\rho(\nu)-1,2\rho(\nu)-2,\cdots,2,1)
	\eea
or
	\bea\label{even-symmetry1}
	\left\{
	\begin{aligned}
		&\varepsilon_1(\nu;+\infty)=0,  \\
		&\varepsilon_2(\nu;\flat-1)=\flat, \\
		&\varepsilon_3(\nu;\flat)=\flat,\\
		&\varepsilon_4(\nu;+\infty)=0.
	\end{aligned}
	\right.\;\forall\,1\leq\flat\leq\rho(\nu)&\Longleftrightarrow&	
	\left\{
	\begin{aligned}
		&\nu^\prime(2t)=-1,\;\quad\quad  t\geq0 \\
		&\nu^\prime(2t+1)=1,\;\quad 0\leq t\leq \rho(\nu)-1\\
		&\nu^\prime(2t+1)=-1, \;\; t\geq\rho(\nu)\\
		&\nu^\prime(-2t)=-1, \; \quad\, 1\leq t\leq \rho(\nu)\\
		&\nu^\prime(-2t)=1,\;\quad\quad t\geq \rho(\nu)+1\\
		&\nu^\prime(-2t+1)=1,\;\; t\geq 1\\
	\end{aligned}
	\right.\\
	\nonumber&\Longleftrightarrow& \nu=(2\rho(\nu),2\rho(\nu)-1,\cdots,2,1)
	\eea	
where $\rho(\nu)=\rho_1(\nu)=\rho_2(\nu)\geq1$. This implies that $\nu=(m,m-1,\cdots,2,1)$ for some $m\in\mathbb{Z}_{\geq1}$.

Therefore, together with Lemma \ref{symmetric-interlacing}, the proof is completed.
\end{proof}

\begin{remark}
By Remark \ref{empty1}, Proposition \ref{unique symmetric interlacing} also includes the trivial case  when $\nu=\emptyset$ or equivalently $m=0$.
\end{remark}

\subsection{Other restricted pyramid configurations with symmetric interlacing property}
	Let $\nu$ be a partition and $l\in\mathbb{Z}_{\geq0}$. In this subsection, we construct some   restricted pyramid configurations of antidiagonal (resp. diagonal) type $(\nu,l)$ with symmetric interlacing property, which actuallly generalize those of antidiagonal (resp. diagonal) type $(\nu,0)$  in Section 3.1 and 3.2.

For an integer $l\geq0$, define
\ben
\widetilde{\mathcal{T}}_{k}(l)=\left\{ \widetilde{\mathbf{S}}_{k}(i,j)\,\bigg|\,i\geq l,\, j\geq l  \right\}\;\;\mbox{and}\;\;\mathcal{T}_{k}(l)=\left\{ \mathbf{S}_{k}(i,j)\,\bigg|\, i\geq l,\, j\geq l   \right\}.
\een
	
\begin{lemma}\label{l-even-odd-symmetry}
	Let $l\in\mathbb{Z}_{\geq0}$. Then we  have 
	\ben
	\bigcup_{k\in\mathbb{Z}}\left(\widetilde{\mathcal{T}}_{k}(l)\right)^c=\bigcup_{k\in\mathbb{Z}}\left(\mathcal{T}_{k}(l)\right)^c.
	\een
\end{lemma}
\begin{proof}
It follows from the similar argument in the proof of  Lemma \ref{even-symmetry} and Lemma \ref{odd-symmetry}.
\end{proof}
\begin{remark}
It is natural to consider more general notions, that is,  for any $l_1,l_2\in\mathbb{Z}_{\geq0}$, define
\ben
\widetilde{\mathcal{R}}_{k}(l_1,l_2)=\left\{ \widetilde{\mathbf{S}}_{k}(i,j)\,\bigg|\,i\geq l_1,\, j\geq l_2  \right\}\;\;\mbox{and}\;\;\mathcal{R}_{k}(l_1,l_2)=\left\{ \mathbf{S}_{k}(i,j)\,\bigg|\, i\geq l_1,\, j\geq l_2   \right\}.
\een
However,  we have $	\bigcup\limits_{k\in\mathbb{Z}}\left(\widetilde{\mathcal{R}}_{k}(l_1,l_2)\right)^c\neq\bigcup\limits_{k\in\mathbb{Z}}\left(\mathcal{R}_{k}(l_1,l_2)\right)^c$ if $l_1\neq l_2$, which can not be used for extracting the restricted  pyramid configurations with symmetric interlacing property in our construction.
\end{remark}	
 Let $\nu$ be a partition and $l\in\mathbb{Z}_{\geq0}$.  Let
 \ben
 &&\widetilde{\mathcal{S}}_{2k}(\nu,l)=\left\{
 \begin{aligned}
 	&\left\{ \widetilde{\mathbf{S}}_{2k}(i,j)\,\bigg|\,i\geq l+\rho_{1}(\nu)-\varepsilon_2(\nu;-k-1),\; j\geq l+ \rho_{2}(\nu)-\varepsilon_1(\nu;-k-1) \right\},  \;\; \mbox{if  $k\leq 0$},\\
 	&\left\{ \widetilde{\mathbf{S}}_{2k}(i,j)\,\bigg|\,i\geq l+\rho_{1}(\nu)-\varepsilon_4(\nu;k),\; j\geq l+\rho_{2}(\nu)- \varepsilon_3(\nu;k)\right\},  \;\; \;\; \;\; \;\; \;\; \;\; \;\; \;\; \;\; \;\; \mbox{if  $k> 0$}.
 \end{aligned}
 \right.
 \\
 &&\widetilde{\mathcal{S}}_{2k-1}(\nu,l)=\left\{
 \begin{aligned}
 	&\left\{ \widetilde{\mathbf{S}}_{2k-1}(i,j)\,\bigg|\,i\geq l+\rho_{1}(\nu)- \varepsilon_2(\nu;-k-1),\; j\geq l+ \rho_{2}(\nu)-\varepsilon_1(\nu;-k) \right\},  \;\; \mbox{if  $k\leq 0$},\\
 	&\left\{ \widetilde{\mathbf{S}}_{2k-1}(i,j)\,\bigg|\,i\geq l+\rho_{1}(\nu)-\varepsilon_4(\nu;k),\; j\geq l+\rho_{2}(\nu)- \varepsilon_3(\nu;k-1) \right\},  \;\;\;\; \;\; \;\, \mbox{if  $k> 0$}.
 \end{aligned}
 \right.
 \een
 and 
 \ben
 &&\mathcal{S}_{2k}(\nu,l)=\left\{
 \begin{aligned}
 	&\left\{ \mathbf{S}_{2k}(i,j)\,\bigg|\,i\geq l+\rho_{1}(\nu)-\varepsilon_2(\nu;-k-1),\; j\geq l+ \rho_{2}(\nu)-\varepsilon_1(\nu;-k-1) \right\},  \;\; \mbox{if  $k\leq 0$},\\
 	&\left\{ \mathbf{S}_{2k}(i,j)\,\bigg|\,i\geq l+\rho_{1}(\nu)- \varepsilon_4(\nu;k),\; j\geq l+\rho_{2}(\nu)- \varepsilon_3(\nu;k)\right\},  \;\; \;\; \;\; \;\; \;\; \;\; \;\; \;\; \;\; \;\; \mbox{if  $k> 0$}.
 \end{aligned}
 \right.
 \\
 &&\mathcal{S}_{2k-1}(\nu,l)=\left\{
 \begin{aligned}
 	&\left\{ \mathbf{S}_{2k-1}(i,j)\,\bigg|\,i\geq l+\rho_{1}(\nu)- \varepsilon_2(\nu;-k-1),\; j\geq l+ \rho_{2}(\nu)-\varepsilon_1(\nu;-k)  \right\},  \;\; \mbox{if  $k\leq 0$},\\
 	&\left\{ \mathbf{S}_{2k-1}(i,j)\,\bigg|\,i\geq l+\rho_{1}(\nu)-\varepsilon_4(\nu;k),\; j\geq l+\rho_{2}(\nu)- \varepsilon_3(\nu;k-1) \right\},  \;\;\;\; \;\; \;\, \mbox{if  $k> 0$}.
 \end{aligned}
 \right.
 \een
Define  the following two classes of restricted pyramid configurations
\ben
\mathfrak{RP}_{+}(\nu,l)=\left\{\widetilde{\bm{\wp}}_{\pi}(\nu,l)\;\bigg| \;\pi\in\mathfrak{P}\right\},\;\;\;\; \mathfrak{RP}_{-}(\nu,l)=\left\{\bm{\wp}_{\pi}(\nu,l)\;\bigg| \;\pi\in\mathfrak{P}\right\}.
\een
where any pyramid partition $\pi\in\mathfrak{P}$,
\ben
\widetilde{\bm{\wp}}_{\pi}(\nu,l)=\bigcup_{k\in\mathbb{Z}}\widetilde{\mathcal{S}}_{k}(\nu,l)\cap\widetilde{\pi}_{k},\;\;\;\;\bm{\wp}_{\pi}(\nu,l)=\bigcup_{k\in\mathbb{Z}}\mathcal{S}_{k}(\nu,l)\cap\pi_{k}.
\een
\begin{remark}\label{generalization of RPC notions}
It is easy to show that 
\ben
&&\widetilde{\mathcal{S}}_{k}(\nu,0)=\widetilde{\mathcal{S}}_{k}(\nu),\;\;\;\;\mathcal{S}_{k}(\nu,0)=\mathcal{S}_{k}(\nu),\;\;\;\;\widetilde{\mathcal{S}}_{k}(\emptyset,l)=\widetilde{\mathcal{T}}_{k}(l),\;\;\;\;\;\; \mathcal{S}_{k}(\emptyset,l)=\mathcal{T}_{k}(l),\\
&&\widetilde{\bm{\wp}}_{\pi}(\nu,0)=\widetilde{\bm{\wp}}_{\pi}(\nu),\;\;\;\bm{\wp}_{\pi}(\nu,0)=\bm{\wp}_{\pi}(\nu),\;\;\;\mathfrak{RP}_{+}(\nu,0)=\mathfrak{RP}_{+}(\nu),\;\;\;\mathfrak{RP}_{-}(\nu,0)=\mathfrak{RP}_{-}(\nu).
\een
\end{remark}

We have the following generalization of Lemma \ref{RPC-interlacing1} by the similar argument.
\begin{lemma}\label{generalized-RPC-interlacing1}
	Let $\nu$ be a partition and $l\in\mathbb{Z}_{\geq0}$. For any pyramid partition $\pi\in\mathfrak{P}$, the restricted pyramid configuration $\widetilde{\bm{\wp}}_{\pi}(\nu,l)$ (resp. $\bm{\wp}_{\pi}(\nu,l)$) is of antidiagonal (resp. diagonal) type $\nu$.
\end{lemma}
Due to Lemmas \ref{RPC-interlacing1}, \ref{generalized-RPC-interlacing1},  Remark \ref{generalization of RPC notions} and Definition \ref{Two classes of RPC}, we unify some definitions as follows.

\begin{definition}\label{unified-dfn-RPC}
Let $\nu$ be a partition and $l\in\mathbb{Z}_{\geq0}$. 
We call $\widetilde{\bm{\wp}}_{\pi}(\nu,l)$ $($resp. $\bm{\wp}_{\pi}(\nu,l)$$)$ a restricted pyramid configurations of anitdiagonal $($resp. diagonal$)$ type $(\nu,l)$. And we say  $\mathfrak{RP}_{+}(\nu,l)$  $($resp. $\mathfrak{RP}_{-}(\nu,l)$$)$ a class of restricted pyramid configurations of anitdiagonal $($resp. diagonal$)$ type $(\nu,l)$.
\end{definition}	
	
It follows from the similar argument in the proof of Lemma \ref{RPC-interlacing} that 
\begin{lemma}\label{generalized-RPC-interlacing}
	Let $\nu$ be a partition and $l\in\mathbb{Z}_{\geq0}$. Let $\{\eta_k\}_{k\in\mathbb{Z}}$ be any sequence  of partitions  satisfying  the interlacing property of the second type $\nu$ and $\eta_k=\emptyset$ for all $|k|\gg0$, then there exists a pyramid partition $\pi\in\mathfrak{P}$ such that  $\eta_k=\widetilde{\bm{\wp}}_{\pi}(\nu,l)_k$ $($as the same 2D Young diagrams$)$ for any $k\in\mathbb{Z}$. Similarly, there exists a pyramid pyramid partition $\vartheta\in\mathfrak{P}$ such that $\eta_k=\bm{\wp}_{\vartheta}(\nu,l)_k$ for any $k\in\mathbb{Z}$. 	
\end{lemma}

Notice that $\nu=(m,m-1,\cdots,2,1)=\emptyset$ is equivalent to $m=0$. 
By the similar argument in Section 3.2 together with Lemma \ref{l-even-odd-symmetry} and Remark \ref{generalization of RPC notions}, we have the following two results.
\begin{lemma}\label{generalized-even-odd-symmetry}
	Assume  $\nu=(m,m-1,\cdots,2,1)$ with $m\in\mathbb{Z}_{\geq0}$. Let $l\in\mathbb{Z}_{\geq0}$. Then we  have 
	\ben
	\bigcup_{k\in\mathbb{Z}}\left(\widetilde{\mathcal{S}}_{k}(\nu,l)\right)^c=\bigcup_{k\in\mathbb{Z}}\left(\mathcal{S}_{k}(\nu,l)\right)^c.
	\een
\end{lemma}
\begin{proposition}\label{generalized-unique symmetric interlacing}
	Let $\nu$ be a  partition  and $l\in\mathbb{Z}_{\geq0}$. Then for any fixed $l$, the equality $\widetilde{\bm{\wp}}_{\pi}(\nu,l)=\bm{\wp}_{\pi}(\nu,l)$ holds for any pyramid  partition $\pi\in\mathfrak{P}$ if and only if $\nu=(m,m-1,\cdots,2,1)$ for some $m\in\mathbb{Z}_{\geq0}$. In particular, if $\nu=(m,m-1,\cdots,2,1)$ with $m\in\mathbb{Z}_{\geq0}$, we have 
	$\mathfrak{RP}_{+}(\nu,l)=\mathfrak{RP}_{-}(\nu,l)$ for any $l\in\mathbb{Z}_{\geq0}$. 
\end{proposition}

\begin{remark}\label{generalize-uniqueness}
It follows from  Proposition \ref{generalized-unique symmetric interlacing} for any fixed $l\in\mathbb{Z}_{\geq0}$ that
 $\mathfrak{RP}_{+}(\nu,l)$  $($or  $\mathfrak{RP}_{-}(\nu,l)$$)$ is the unique class of restricted pyramid configurations of antidiagonal $($or diagonal$)$   type $(\nu,l)$ with the symmetric interlacing property of type $\nu$  only when $\nu=(m,m-1,\cdots,2,1)$ where  $m\in\mathbb{Z}_{\geq1}$ is the length of $\nu$.
\end{remark}

\section{The 1-leg DT $\mathbb{Z}_2\times\mathbb{Z}_2$-vertex}
In Section 4.1, we define and show generating functions of two classes of  restricted pyramid configurations $Z_{\mathfrak{RP}_{+}(\nu,l)}$ and $Z_{\mathfrak{RP}_{-}(\nu,l)}$ are both independent of $l$ and present their vertex operator expressions for any partition $\nu$ when $l=0$.  They are equal  when $\nu=(m,m-1,\cdots,2,1)$ with $m\in\mathbb{Z}_{\geq1}$. With this special 1-leg partition $\nu$, we obtain the more explicit vertex operator product expressions for $V_{\emptyset\emptyset\nu}^{\mathbb{Z}_{2}\times\mathbb{Z}_{2}}
$, $Z_{\mathfrak{RP}_{+}(\nu)}$, $Z_{\mathfrak{RP}_{-}(\nu)}$, and $V^{\mathbb{Z}_4}_{\emptyset\emptyset\nu}$. And
 we derive the relations between $Z_{\mathfrak{RP}_{+}(\nu)}$ (resp. $Z_{\mathfrak{RP}_{-}(\nu)}$) and 1-leg DT $\mathbb{Z}_2\times\mathbb{Z}_2$-vertex $V_{\emptyset\emptyset\nu}^{\mathbb{Z}_{2}\times\mathbb{Z}_{2}}
 $ (resp. 1-leg DT $\mathbb{Z}_4$-vertex $V^{\mathbb{Z}_4}_{\emptyset\emptyset\nu}$) in Section 4.2 (resp. Section 4.3)
to establish a connection between $V_{\emptyset\emptyset\nu}^{\mathbb{Z}_{2}\times\mathbb{Z}_{2}}
$ and $V^{\mathbb{Z}_4}_{\emptyset\emptyset\nu}$ in Section 4.4.  An explicit formula for 
this special class of 1-leg DT $\mathbb{Z}_2\times\mathbb{Z}_2$-vertex $V_{\emptyset\emptyset\nu}^{\mathbb{Z}_{2}\times\mathbb{Z}_{2}}
$ is presented in Section 4.5.

\subsection{Vertex operator products,  topological vertices and  restricted pyramid configurations}
We  start with describing vertex operator product expression for 1-leg DT $\mathbb{Z}_{2}\times\mathbb{Z}_{2}$-vertex  as follows. 
According to the $\mathbb{Z}_2\times\mathbb{Z}_2$-coloring in Definition \ref{coloring} and 1-leg partition $\nu$, we introduce operators
\bea\label{Z2Z2coloring}
\widetilde{Q}_{k,\nu}=\begin{cases}
	Q_{0c},\quad  \mbox{if $k\equiv0$ (mod 2) and $|\nu|_k\equiv 0$ (mod 2)}; \\
	Q_{c0},\quad  \mbox{if $k\equiv0$ (mod 2)  and $|\nu|_k\equiv 1$ (mod 2)}; \\
	Q_{ab},\quad  \mbox{if $k\equiv1$ (mod 2),  $k>0$ and $|\nu|_k\equiv 0$ (mod 2)}; \\
	Q_{ab},\quad  \mbox{if $k\equiv1$ (mod 2), $k<0$ and $|\nu|_k\equiv 1$ (mod 2)}; \\
	Q_{ba},\quad  \mbox{if $k\equiv1$ (mod 2), $k>0$ and $|\nu|_k\equiv 1$ (mod 2)}; \\
	Q_{ba},\quad  \mbox{if $k\equiv1$ (mod 2), $k<0$ and $|\nu|_k\equiv 0$ (mod 2)}.
\end{cases}
\eea
where  $|\nu|_k=|\{(i,j)\in\nu\,|\, i-j=k)\}|$.
Let $\overrightarrow{\prod\limits_{t\in\mathbb{S}}}\Psi_t$ denote the product of operators $\Psi_t$ with the order such that $t\in\mathbb{S}\subset\mathbb{Z}$ increases from left to right. Set $\Gamma_{\pm1}(x)=\Gamma_{\pm}(x)$ for any variable $x$.
\begin{lemma}\label{general-Z2Z2-vertex expression}
	For any partition $\nu$, we have
	\ben
	V_{\emptyset\emptyset\nu}^{\mathbb{Z}_{2}\times\mathbb{Z}_{2}}=\left\langle\emptyset\,\left|\; \overrightarrow{\prod_{t\in\mathbb{Z}}}\widetilde{Q}_{-t,\nu}\Gamma_{\nu^\prime(t)}(1)\,\right|\,\emptyset\right\rangle.
	\een
\end{lemma}
\begin{proof}
	It follows from the similar argument in the proof of [\cite{BCY}, Proposition 8] with the correct $\mathbb{Z}_2\times\mathbb{Z}_2$-coloring in Definition \ref{coloring} and interlacing property of the first type $\nu$ in Definition \ref{dfn-interlacing}.
\end{proof}	

As in  [\cite{BY}, Definition 6.1], we define the following two  generating functions. 
\begin{definition}\label{GF-RPC}
	For any partition $\nu$, define	
	\ben
	&&Z_{\mathfrak{RP}_{+}(\nu)}=\sum_{\bm{\wp}\in\mathfrak{RP}_{+}(\nu)}\prod_{g\in\mathbb{Z}_2\times\mathbb{Z}_2}q_{g}^{|\bm{\wp}|_g},\\
	&&Z_{\mathfrak{RP}_{-}(\nu)}=\sum_{\bm{\wp}\in\mathfrak{RP}_{-}(\nu)}\prod_{g\in\mathbb{Z}_2\times\mathbb{Z}_2}q_{g}^{|\bm{\wp}|_g},
	\een
	where $|\bm{\wp}|_g:=|K_{\mathfrak{P}}^{-1}(g)\cap\bm{\wp}|$ is the number of bricks colored $g$ in $\bm{\wp}$.
\end{definition}

\begin{remark}\label{empty2}
	It is clear that the above two generating functions  generalizes  $Z_{\mathrm{pyramid}}$ in  [\cite{BY}, Definition 6.1] since we have $Z_{\mathfrak{RP}_{+}(\emptyset)}=Z_{\mathfrak{RP}_{-}(\emptyset)}=Z_{\mathfrak{P}}=Z_{\mathrm{pyramid}}$  by Remark \ref{empty1}.
\end{remark}

According the coloring of bricks in the antidiagonal slices  of  elements $\widetilde{\bm{\wp}}_{\pi}(\nu)$ in  $\mathfrak{RP}_{+}(\nu)$, we define 
\bea\label{RPC+coloring}
\overline{Q}_{k,\nu}=\begin{cases}
	Q_{0c},\quad   \mbox{if $k\equiv0$ (mod 2) and $\mho(k,\nu)\equiv0$ (mod 2)}; \\
	Q_{c0},\quad   \mbox{if $k\equiv0$ (mod 2) and $\mho(k,\nu)\equiv1$ (mod 2)}; \\
	Q_{ba},\quad   \mbox{if $k\equiv1$ (mod 2), $k>0$, $\mho(k,\nu)\equiv0$ (mod 2)}; \\
	Q_{ba},\quad   \mbox{if $k\equiv1$ (mod 2), $k<0$, $\mho(k,\nu)\equiv1$ (mod 2)}; \\
	Q_{ab},\quad   \mbox{if $k\equiv1$ (mod 2), $k>0$, $\mho(k,\nu)\equiv1$ (mod 2)}; \\
	Q_{ab},\quad   \mbox{if $k\equiv1$ (mod 2), $k<0$, $\mho(k,\nu)\equiv0$ (mod 2)}.
\end{cases}
\eea
where
\ben
\mho(k,\nu)=\begin{cases}
	\rho_1(\nu)+\rho_2(\nu)-\varepsilon_{1}\left(\nu;\displaystyle-\frac{k}{2}-1\right)-\varepsilon_2\left(\nu;\displaystyle-\frac{k}{2}-1\right),\quad\quad\quad  \mbox{if $k\equiv0$ (mod 2) and $k\leq0$}; \\
	\rho_1(\nu)+\rho_2(\nu)-\varepsilon_{1}\left(\nu;-\displaystyle\frac{k+1}{2}\right)-\varepsilon_2\left(\nu;-\displaystyle\frac{k+1}{2}-1\right),\quad \, \mbox{if $k\equiv1$ (mod 2) and $k\leq0$}; \\
	\rho_1(\nu)+\rho_2(\nu)-\varepsilon_{3}\left(\nu;\displaystyle\frac{k}{2}\right)-\varepsilon_4\left(\nu;\displaystyle\frac{k}{2}\right),\quad\quad\quad\quad\quad\quad\quad\quad  \mbox{if $k\equiv0$ (mod 2) and $k>0$}; \\
	\rho_1(\nu)+\rho_2(\nu)-\varepsilon_{3}\left(\nu;\displaystyle\frac{k+1}{2}-1\right)-\varepsilon_4\left(\nu;\displaystyle\frac{k+1}{2}\right),\quad \quad\quad \mbox{if $k\equiv1$ (mod 2) and $k>0$}.
\end{cases}
\een
is employed to determine  the color of the orgin of the 2D Young diagram $\widetilde{\bm{\wp}}_{\pi}(\nu)_k$.
Now we have the following vertex operator product expressions for $Z_{\mathfrak{RP}_{+}(\nu)}$ and $Z_{\mathfrak{RP}_{-}(\nu)}$.

\begin{lemma}\label{vertex-expressions-RPC+-}
	For any partition $\nu$, we have
	\ben
	&&Z_{\mathfrak{RP}_{+}(\nu)}=\left\langle\emptyset\,\left|\; \overrightarrow{\prod_{t\in\mathbb{Z}}}\overline{Q}_{-2t,\nu}\Gamma_{\nu^\prime(2t)}^\prime(1)\overline{Q}_{-2t-1,\nu}\Gamma_{\nu^\prime(2t+1)}(1)\,\right|\,\emptyset\right\rangle,\\
	&&Z_{\mathfrak{RP}_{-}(\nu)}=\left\langle\emptyset\,\bigg|\, \overrightarrow{\prod_{t\in\mathbb{Z}}}Q_{0}\Gamma_{\nu^\prime(4t)}^\prime(1)Q_{a}\Gamma_{\nu^\prime(4t+1)}(1)Q_{c}\Gamma_{\nu^\prime(4t+2)}^\prime(1)Q_{b}\Gamma_{\nu^\prime(4t+3)}(1)\,\bigg|\,\emptyset\right\rangle.
	\een
\end{lemma}

\begin{proof}
	It follows from	Lemma \ref{interlacing1}, Lemma \ref{RPC-interlacing}, and Lemma  \ref{RPC-interlacing1} by using the interlacing property of slices $\{\widetilde{\bm{\wp}}_{\pi}(\nu)_k\,|\,k\in\mathbb{Z}\}$ (resp. $\{\bm{\wp}_{\pi}(\nu)_k\,|\,k\in\mathbb{Z}\}$) and the color assignments of bricks in  $\widetilde{\bm{\wp}}_{\pi}(\nu)_k$  (resp. $\bm{\wp}_{\pi}(\nu)_k$) for each $k\in\mathbb{Z}$.
\end{proof}

\begin{remark}\label{Independent-of-l}
	By Lemma \ref{generalized-RPC-interlacing1} and Lemma \ref{generalized-RPC-interlacing},  the interlacing properties for slices  of  elements $\widetilde{\bm{\wp}}_{\pi}(\nu,l)$ in $\mathfrak{RP}_{+}(\nu,l)$ $($resp. $\bm{\wp}_{\pi}(\nu,l)$ in $\mathfrak{RP}_{-}(\nu,l)$$)$ are independent of $l$.  Define 
	\ben
	Z_{\mathfrak{RP}_{+}(\nu,l)}=\sum_{\bm{\wp}\in\mathfrak{RP}_{+}(\nu,l)}\prod_{g\in\mathbb{Z}_2\times\mathbb{Z}_2}q_{g}^{|\bm{\wp}|_g},\;\;\; Z_{\mathfrak{RP}_{-}(\nu,l)}=\sum_{\bm{\wp}\in\mathfrak{RP}_{-}(\nu,l)}\prod_{g\in\mathbb{Z}_2\times\mathbb{Z}_2}q_{g}^{|\bm{\wp}|_g},
	\een
	as in Definition \ref{GF-RPC}. By the  color assignment of bricks in $\widetilde{\bm{\wp}}_{\pi}(\nu,l)_k$   $($resp. $\bm{\wp}_{\pi}(\nu,l)_k$$)$ for each $k\in\mathbb{Z}$, one can show by comparing their vertex operator product expressions that
	\ben
	Z_{\mathfrak{RP}_{+}(\nu,l)}=Z_{\mathfrak{RP}_{+}(\nu)}\;\;\;\mbox{and}\;\;\; Z_{\mathfrak{RP}_{-}(\nu,l)}=Z_{\mathfrak{RP}_{-}(\nu)}\;\;\; \mbox{for any $l\geq0$.}
	\een
	where $Z_{\mathfrak{RP}_{+}(\nu)}=Z_{\mathfrak{RP}_{+}(\nu,0)}$ and $Z_{\mathfrak{RP}_{-}(\nu)}=Z_{\mathfrak{RP}_{-}(\nu,0)}$  by Remark \ref{generalization of RPC notions}.
\end{remark}

\begin{remark}\label{general-Z2Z2-RPC+}
	By Lemma \ref{general-Z2Z2-vertex expression} and Lemma \ref{vertex-expressions-RPC+-}, we have for any partition $\nu\neq\emptyset$,
	\ben
	&&V_{\emptyset\emptyset\nu}^{\mathbb{Z}_{2}\times\mathbb{Z}_{2}}=\left\langle\emptyset\,\left|\; \overrightarrow{\prod_{t\in\mathbb{Z}}}\widetilde{Q}_{-2t,\nu}\Gamma_{\nu^\prime(2t)}(1)\widetilde{Q}_{-2t-1,\nu}\Gamma_{\nu^\prime(2t+1)}(1)\,\right|\,\emptyset\right\rangle,\\
	&&Z_{\mathfrak{RP}_{+}(\nu)}=\left\langle\emptyset\,\left|\; \overrightarrow{\prod_{t\in\mathbb{Z}}}\overline{Q}_{-2t,\nu}\Gamma_{\nu^\prime(2t)}^\prime(1)\overline{Q}_{-2t-1,\nu}\Gamma_{\nu^\prime(2t+1)}(1)\,\right|\,\emptyset\right\rangle.
	\een
	In general, the connection between 
	$Z_{\mathfrak{RP}_{+}(\nu)}$ and $V_{\emptyset\emptyset\nu}^{\mathbb{Z}_{2}\times\mathbb{Z}_{2}}$ follows from comparing $\overline{Q}_{k,\nu}$ with $\widetilde{Q}_{k,\nu}$ for each $k\in\mathbb{Z}$ and applying the relations $\Gamma_{\pm}(1)=\Gamma_{\pm}^\prime(1)E_{\pm}(1)$ in Lemma \ref{vertex-exchange1} using the vertex operator methods of Okounkov-Reshetikhin-Vafa and Bryan-Young .
\end{remark}
Next, we will focus on $Z_{\mathfrak{RP}_{+}(\nu)}$ and $Z_{\mathfrak{RP}_{-}(\nu)}$ below. 
Notice that $Z_{\mathfrak{RP}_{-}(\nu)}$ is a generating function with variables $q_0,q_b,q_c,q_a$ while $V^{\mathbb{Z}_4}_{\emptyset\emptyset\nu}$ is a generating function with variables $\tilde{q}_0, \tilde{q}_1, \tilde{q}_2, \tilde{q}_3$.  To reveal the relation between $Z_{\mathfrak{RP}_{-}(\nu)}$ and  $V^{\mathbb{Z}_4}_{\emptyset\emptyset\nu}$, taking into account the color of $k$-th diagonal slice of a pyramid partition (or a restricted pyramid configuration) in Lemma \ref{interlacing1},  it is natural and convenient to identify their variables by 
\ben
\tilde{q}_0=q_0,\;\; \tilde{q}_1=q_b, \;\;  \tilde{q}_2=q_c,\;\; \tilde{q}_3=q_a
\een
together with  identifying operators $Q_0, Q_1, Q_2, Q_3$ ($\{0,1,2,3\}=\mathbb{Z}_4$)  with operators $Q_0, Q_b, Q_c, Q_a$  ($\{0,b,c,a\}=\mathbb{Z}_2\times\mathbb{Z}_2$) respectively.
With this convention, we define the following weight operators 
\ben
\widehat{Q}_{k}=\begin{cases}
	Q_{0},\quad   \mbox{if $k\equiv0$ (mod 4)}; \\
	Q_{b},\quad  \mbox{if $k\equiv1$ (mod 4)}; \\
	Q_{c},\quad  \mbox{if $k\equiv2$ (mod 4)}; \\
	Q_{a},\quad  \mbox{if $k\equiv3$ (mod 4)}.
\end{cases}
\een
for counting the $\mathbb{Z}_4$-colored 3D Young diagrams. It is shown in [\cite{BCY}, Proposition 8] that the 1-leg DT $\mathbb{Z}_{4}$-vertex has the following vertex operator product expression:
\bea\label{Z4-coloring}
V^{\mathbb{Z}_4}_{\emptyset\emptyset\nu}(q_0,q_b,q_c,q_a)&=&\left\langle\emptyset\,\left|\; \overrightarrow{\prod_{t\in\mathbb{Z}}}\widehat{Q}_{-t}\Gamma_{\nu^\prime(t)}(1)\,\right|\,\emptyset\right\rangle\\
&=&\left\langle\emptyset\,\bigg|\, \overrightarrow{\prod_{t\in\mathbb{Z}}}Q_{0}\Gamma_{\nu^\prime(4t)}(1)Q_{a}\Gamma_{\nu^\prime(4t+1)}(1)Q_{c}\Gamma_{\nu^\prime(4t+2)}(1)Q_{b}\Gamma_{\nu^\prime(4t+3)}(1)\,\bigg|\,\emptyset\right\rangle.\nonumber
\eea
\begin{remark}
As in Remark \ref{general-Z2Z2-RPC+}, by Lemma \ref{vertex-expressions-RPC+-} and Equation \eqref{Z4-coloring}, one can connect $Z_{\mathfrak{RP}_{-}(\nu)}$ with $V_{\emptyset\emptyset\nu}^{\mathbb{Z}_{4}}$ by applying relations $\Gamma_{\pm}(1)=\Gamma_{\pm}^\prime(1)E_{\pm}(1)$ in Lemma \ref{vertex-exchange1}.
\end{remark}

The following result  establishes the connection between $V_{\emptyset\emptyset\nu}^{\mathbb{Z}_{2}\times\mathbb{Z}_{2}}$ and $V_{\emptyset\emptyset\nu}^{\mathbb{Z}_{4}}$ in some special cases.
\begin{lemma}\label{RPC+RPC-}
	Assume  $\nu=(m,m-1,\cdots,2,1)$ with $m\in\mathbb{Z}_{\geq1}$. Then
	\ben
	Z_{\mathfrak{RP}_{+}(\nu)}=Z_{\mathfrak{RP}_{-}(\nu)}.
	\een
\end{lemma}	
\begin{proof}
	It follows from Definition \ref{GF-RPC} and Lemma \ref{symmetric-interlacing}.
\end{proof}
\begin{remark}\label{RPC+---RPC-}
If $\nu\neq\emptyset$ and $\nu\neq(m,m-1,\cdots,2,1)$ for any $m\geq1$, then by Proposition \ref{unique symmetric interlacing} we have 
\ben
\bigcup_{k\in\mathbb{Z}}\left(\widetilde{\mathcal{S}}_{k}(\nu)\right)^c\neq\bigcup_{k\in\mathbb{Z}}\left(\mathcal{S}_{k}(\nu)\right)^c.
\een
Then there exists a brick $[B]$ such that $[B]\in\bigcup\limits_{k\in\mathbb{Z}}\left(\widetilde{\mathcal{S}}_{k}(\nu)\right)^c$ and $[B]\notin\bigcup\limits_{k\in\mathbb{Z}}\left(\mathcal{S}_{k}(\nu)\right)^c$, or $[B]\in\bigcup\limits_{k\in\mathbb{Z}}\left(\mathcal{S}_{k}(\nu)\right)^c$ and $[B]\notin\bigcup\limits_{k\in\mathbb{Z}}\left(\widetilde{\mathcal{S}}_{k}(\nu)\right)^c$. In the first case, for any pyramid partition $\pi$ containing $[B]$, one can show $\bm{\wp}_{\pi}(\nu)\in\mathfrak{RP}_{-}(\nu)$ and $\bm{\wp}_{\pi}(\nu)\notin\mathfrak{RP}_{+}(\nu)$ since $[B]\in\bm{\wp}_{\pi}(\nu)$ but any element in $\mathfrak{RP}_{+}(\nu)$ does not contain $[B]$. The second case is similar. Then $\mathfrak{RP}_{+}(\nu)\neq\mathfrak{RP}_{-}(\nu)$.
In such a case, I have no idea how  to establish the connection between  $Z_{\mathfrak{RP}_{+}(\nu)}$ and  $Z_{\mathfrak{RP}_{-}(\nu)}$.
\end{remark}	
In the following, we will concentrate on the 1-leg partitions $\nu=(m,m-1,\cdots,2,1)$ with $m\in\mathbb{Z}_{\geq1}$. 
To compare $\overline{Q}_{k,\nu}$ with $\widetilde{Q}_{k,\nu}$ for all $k\in\mathbb{Z}$, we start with the following
\begin{lemma}\label{operator-relation}
	Assume  $\nu=(m,m-1,\cdots,2,1)$ with $m\in\mathbb{Z}_{\geq1}$.
	Then
	\ben
	\mho(k,\nu)-|\nu|_k \equiv\begin{cases}
		0 \mbox{ (mod 2)},\quad\quad\;\;  \mbox{if $m\equiv0$ (mod 4) or  $m\equiv3$ (mod 4)}; \\
		1 \mbox{ (mod 2)},\quad \quad\;\;  \mbox{if $m\equiv1$ (mod 4) or  $m\equiv2$ (mod 4)}.
	\end{cases}
	\een
\end{lemma}
\begin{proof}
	Since $\nu=(m,m-1,\cdots,2,1)$, it follows from the direct computation that
	\bea\label{v-color1}
	|\nu|_k=\begin{cases}
		\displaystyle\frac{m}{2}-\left\lfloor\frac{|k|}{2}\right\rfloor,\quad\quad\quad\quad \;\quad   \mbox{if $m\equiv0$ (mod 2) and $-m\leq k\leq m$}; \\
		\;0, \quad\quad\quad\quad\quad\quad\quad\quad\quad\;\; \mbox{if $|k|>m$}; \\
		\displaystyle\frac{m+1}{2}-\left\lfloor\frac{1+|k|}{2}\right\rfloor,\quad\quad   \mbox{if $m\equiv1$ (mod 2) and $-m\leq k\leq m$}.
	\end{cases}
	\eea
	and 
	\bea\label{v-color2}
	\mho(k,\nu)=\begin{cases}
		\displaystyle m-\left\lfloor\frac{|k|}{2}\right\rfloor,\quad\quad\quad\quad \;\quad \; \; \mbox{if $m\equiv0$ (mod 2) and $-m\leq k\leq m$}; \\
		\displaystyle\frac{m}{2}, \quad\quad\quad\quad\quad\quad\quad\quad\quad\;\; \mbox{if $m\equiv0$ (mod 2) and $|k|>m$}; \\
		\displaystyle m+1-\left\lfloor\frac{1+|k|}{2}\right\rfloor,\quad\quad \; \, \mbox{if $m\equiv1$ (mod 2) and $-m\leq k\leq m$};\\
		\displaystyle\frac{m+1}{2}, \quad\quad\quad\quad\quad\quad\quad\quad \mbox{if $m\equiv1$ (mod 2) and $|k|>m$}. \\
	\end{cases}
	\eea
	Now the proof follows from \eqref{v-color1} and \eqref{v-color2}.
\end{proof}

Actually we have
\begin{lemma}\label{weight-operator1}
Assume  $\nu=(m,m-1,\cdots,2,1)$ with $m\in\mathbb{Z}_{\geq1}$.
Then we have
\ben
&&\mbox{(i) if  $m\equiv0$  (mod 4) and $-m\leq k\leq m$, }\;\;\;\;
\widetilde{Q}_{k,\nu}=\begin{cases}
	Q_{0c},\quad  \mbox{if $k\equiv0$ (mod 4)}; \\
	Q_{ab},\quad  \mbox{if $k\equiv1$ (mod 4)}; \\
	Q_{c0},\quad  \mbox{if $k\equiv2$ (mod 4)}; \\
	Q_{ba},\quad  \mbox{if $k\equiv3$ (mod 4)}.
\end{cases}\\
&&\mbox{(ii) if  $m\equiv1$  (mod 4) and $-m\leq k\leq m$, }\;\;\;
\widetilde{Q}_{k,\nu}=\begin{cases}
	Q_{c0},\quad  \mbox{if $k\equiv0$ (mod 4)}; \\
	Q_{ab},\quad  \mbox{if $k\equiv1$ (mod 4)}; \\
	Q_{0c},\quad  \mbox{if $k\equiv2$ (mod 4)}; \\
	Q_{ba},\quad  \mbox{if $k\equiv3$ (mod 4)}.
\end{cases}\\
&&\mbox{(iii) if  $m\equiv2$  (mod 4) and $-m\leq k\leq m$, }\;\;
\widetilde{Q}_{k,\nu}=\begin{cases}
	Q_{c0},\quad  \mbox{if $k\equiv0$ (mod 4)}; \\
	Q_{ba},\quad  \mbox{if $k\equiv1$ (mod 4)}; \\
	Q_{0c},\quad  \mbox{if $k\equiv2$ (mod 4)}; \\
	Q_{ab},\quad  \mbox{if $k\equiv3$ (mod 4)}.
\end{cases}\\
&&\mbox{(iv) if  $m\equiv3$  (mod 4) and $-m\leq k\leq m$, }\;\;
\widetilde{Q}_{k,\nu}=\begin{cases}
	Q_{0c},\quad  \mbox{if $k\equiv0$ (mod 4)}; \\
	Q_{ba},\quad  \mbox{if $k\equiv1$ (mod 4)}; \\
	Q_{c0},\quad  \mbox{if $k\equiv2$ (mod 4)}; \\
	Q_{ab},\quad  \mbox{if $k\equiv3$ (mod 4)}.
\end{cases}
\een
and 
\ben
&&\mbox{(v) if  $k> m$, }\;\;\;\;\;\,
\widetilde{Q}_{k,\nu}=\begin{cases}
	Q_{0c},\quad  \mbox{if $k\equiv0$ (mod 2)}; \\
	Q_{ab},\quad  \mbox{if $k\equiv1$ (mod 2)}.
\end{cases}\\
&&\mbox{(vi) if  $k<-m$, }\;\;
\widetilde{Q}_{k,\nu}=\begin{cases}
	Q_{0c},\quad  \mbox{if $k\equiv0$ (mod 2)}; \\
	Q_{ba},\quad  \mbox{if $k\equiv1$ (mod 2)}.
\end{cases}
\een
\end{lemma}
\begin{proof}
It follows from \eqref{Z2Z2coloring} together with \eqref{v-color1} by the direct computation.
\end{proof}

\begin{corollary}\label{weight-operator2}
Assume  $\nu=(m,m-1,\cdots,2,1)$ with $m\in\mathbb{Z}_{\geq1}$.
Then for any $k\in\mathbb{Z}$, we have
\ben
\overline{Q}_{k,\nu}=
\begin{cases}\medskip 
	\widetilde{Q}_{k,\nu}\bigg|_{a\leftrightarrow b},\quad  \mbox{if $m\equiv0$  (mod 4)}; \\ \medskip 
	\widetilde{Q}_{k,\nu}\bigg|_{0\leftrightarrow c},\quad  \mbox{if $m\equiv1$ (mod 4)}; \\ \medskip
	\widetilde{Q}_{k,\nu}\bigg|_{0\leftrightarrow c},\quad  \mbox{if $m\equiv2$ (mod 4)}; \\ \medskip
	\widetilde{Q}_{k,\nu}\bigg|_{a\leftrightarrow b},\quad  \mbox{if $m\equiv3$ (mod 4)}.
\end{cases}
\een
where $\widetilde{Q}_{k,\nu}\bigg|_{a\leftrightarrow b}$ denotes  the exchange of  $a$ and $b$ in the expression of $\widetilde{Q}_{k,\nu}$.
	
\end{corollary}	
\begin{proof}
It follows from Lemma \ref{operator-relation} and Lemma \ref{weight-operator1} by comparing \eqref{Z2Z2coloring} with \eqref{RPC+coloring}.
\end{proof}

Next, we will present the vertex operator product expressions for $V_{\emptyset\emptyset\nu}^{\mathbb{Z}_{2}\times\mathbb{Z}_{2}}
$, $Z_{\mathfrak{RP}_{+}(\nu)}$, $Z_{\mathfrak{RP}_{-}(\nu)}$, and $V^{\mathbb{Z}_4}_{\emptyset\emptyset\nu}$ when  $\nu=(m,m-1,\cdots,2,1)$ with $m\in\mathbb{Z}_{\geq1}$ as follows.
\begin{lemma}
Assume  $\nu=(m,m-1,\cdots,2,1)$ with $m\in\mathbb{Z}_{\geq1}$.
Then we have\\
(i) if $m\equiv0\; \mbox{(mod 2)}$, 
\bea\label{Z2Z2-0}
V_{\emptyset\emptyset\nu}^{\mathbb{Z}_{2}\times\mathbb{Z}_{2}}
&=&
\bigg\langle\emptyset\,\bigg|\, \prod_{i=1}^{\infty}Q_{0c}\Gamma_{+}(1)Q_{ab}\Gamma_{+}(1)\cdot\prod_{i=1}^{\frac{m}{2}}Q_{0c}\Gamma_{-}(1)Q_{ba}\Gamma_{+}(1)Q_{c0}\Gamma_{-}(1)Q_{ab}\Gamma_{+}(1)\\
&&\quad\cdot\prod_{i=1}^{\infty}Q_{0c}\Gamma_{-}(1)Q_{ba}\Gamma_{-}(1)\, \bigg|\,\emptyset\bigg\rangle\nonumber,
\eea
(ii) if $m\equiv1\; \mbox{(mod 2)}$, 
\bea\label{Z2Z2-1}
V_{\emptyset\emptyset\nu}^{\mathbb{Z}_{2}\times\mathbb{Z}_{2}}
&=&
\bigg\langle\emptyset\,\bigg|\, \prod_{i=1}^{\infty}Q_{ab}\Gamma_{+}(1)Q_{0c}\Gamma_{+}(1)\cdot\left(\prod_{i=1}^{\frac{m-1}{2}}Q_{ab}\Gamma_{-}(1)Q_{c0}\Gamma_{+}(1)Q_{ba}\Gamma_{-}(1)Q_{0c}\Gamma_{+}(1)\right)\\
&&\quad\cdot Q_{ab}\Gamma_{-}(1)Q_{c0}\Gamma_{+}(1)
\cdot\prod_{i=1}^{\infty}Q_{ba}\Gamma_{-}(1)Q_{0c}\Gamma_{-}(1)\, \bigg|\,\emptyset\bigg\rangle\nonumber.
\eea
\end{lemma}
\begin{proof}
It follows from Lemma \ref{general-Z2Z2-vertex expression}, Lemma \ref{weight-operator1}, together with the two latter equivalences in Equation \eqref{odd-symmetry1} and Equation \eqref{even-symmetry1}.
\end{proof}
\begin{lemma}
	Assume  $\nu=(m,m-1,\cdots,2,1)$ with $m\in\mathbb{Z}_{\geq1}$.
	Then we have\\
	(i) if $m\equiv0\; \mbox{(mod 4)}$, 
	\bea
	\label{RPC+0}
	Z_{\mathfrak{RP}_{+}(\nu)}
	&=&
	\bigg\langle\emptyset\,\bigg|\, \prod_{i=1}^{\infty}Q_{0c}\Gamma_{+}^\prime(1)Q_{ba}\Gamma_{+}(1)\cdot\prod_{i=1}^{\frac{m}{2}}Q_{0c}\Gamma_{-}^\prime(1)Q_{ab}\Gamma_{+}(1)Q_{c0}\Gamma_{-}^\prime(1)Q_{ba}\Gamma_{+}(1)\\
	&&\quad\cdot\prod_{i=1}^{\infty}Q_{0c}\Gamma_{-}^\prime(1)Q_{ab}\Gamma_{-}(1)\, \bigg|\,\emptyset\bigg\rangle\nonumber,
	\eea
	(ii) if $m\equiv1\; \mbox{(mod 4)}$, 
	\bea
	\label{RPC+1}
	Z_{\mathfrak{RP}_{+}(\nu)}
	&=&
	\bigg\langle\emptyset\,\bigg|\, \prod_{i=1}^{\infty}Q_{ab}\Gamma_{+}(1)Q_{c0}\Gamma_{+}^\prime(1)\cdot\left(\prod_{i=1}^{\frac{m-1}{2}}Q_{ab}\Gamma_{-}(1)Q_{0c}\Gamma_{+}^\prime(1)Q_{ba}\Gamma_{-}(1)Q_{c0}\Gamma_{+}^\prime(1)\right)\\
	&&\quad\cdot Q_{ab}\Gamma_{-}(1)Q_{0c}\Gamma_{+}^\prime(1)\cdot\prod_{i=1}^{\infty}Q_{ba}\Gamma_{-}(1)Q_{c0}\Gamma_{-}^\prime(1)\, \bigg|\,\emptyset\bigg\rangle\nonumber,
	\eea
	(iii) if $m\equiv2\; \mbox{(mod 4)}$, 
	\ben
	Z_{\mathfrak{RP}_{+}(\nu)}
	&=&
	\bigg\langle\emptyset\,\bigg|\, \prod_{i=1}^{\infty}Q_{c0}\Gamma_{+}^\prime(1)Q_{ab}\Gamma_{+}(1)\cdot\prod_{i=1}^{\frac{m}{2}}Q_{c0}\Gamma_{-}^\prime(1)Q_{ba}\Gamma_{+}(1)Q_{0c}\Gamma_{-}^\prime(1)Q_{ab}\Gamma_{+}(1)\\
	&&\quad\cdot\prod_{i=1}^{\infty}Q_{c0}\Gamma_{-}^\prime(1)Q_{ba}\Gamma_{-}(1)\, \bigg|\,\emptyset\bigg\rangle,
	\een
	(iv) if $m\equiv3\; \mbox{(mod 4)}$, 
	\ben
	Z_{\mathfrak{RP}_{+}(\nu)}
	&=&
	\bigg\langle\emptyset\,\bigg|\, \prod_{i=1}^{\infty}Q_{ba}\Gamma_{+}(1)Q_{0c}\Gamma_{+}^\prime(1)\cdot\left(\prod_{i=1}^{\frac{m-1}{2}}Q_{ba}\Gamma_{-}(1)Q_{c0}\Gamma_{+}^\prime(1)Q_{ab}\Gamma_{-}(1)Q_{0c}\Gamma_{+}^\prime(1)\right)\\
	&&\quad\cdot Q_{ba}\Gamma_{-}(1)Q_{c0}\Gamma_{+}^\prime(1)\cdot\prod_{i=1}^{\infty}Q_{ab}\Gamma_{-}(1)Q_{0c}\Gamma_{-}^\prime(1)\, \bigg|\,\emptyset\bigg\rangle.
	\een
\end{lemma}
\begin{proof}
Combined with the two latter equivalences in Equation \eqref{odd-symmetry1} and Equation \eqref{even-symmetry1}, the proof follows from Lemma \ref{vertex-expressions-RPC+-}, Lemma \ref{weight-operator1} and  Corollary \ref{weight-operator2}.
\end{proof}
\begin{lemma}
Assume  $\nu=(m,m-1,\cdots,2,1)$ with $m\in\mathbb{Z}_{\geq1}$.
Then we have\\
(i) if $m\equiv0\; \mbox{(mod 4)}$, 
\bea
Z_{\mathfrak{RP}_{-}(\nu)}
&=&\bigg\langle\emptyset\,\bigg|\, \prod_{i=1}^{\infty}Q_{0}\Gamma_{+}^\prime(1)Q_{a}\Gamma_{+}(1)Q_{c}\Gamma_{+}^\prime(1)Q_{b}\Gamma_{+}(1)\,\label{RPC-0}\\
&&\quad\cdot\prod_{i=1}^{\frac{m}{2}}Q_{0}\Gamma_{-}^\prime(1)Q_{a}\Gamma_{+}(1)Q_{c}\Gamma_{-}^\prime(1)Q_{b}\Gamma_{+}(1)\nonumber\\
&&\quad\cdot\prod_{i=1}^{\infty}Q_{0}\Gamma_{-}^\prime(1)Q_{a}\Gamma_{-}(1)Q_{c}\Gamma_{-}^\prime(1)Q_{b}\Gamma_{-}(1)\; \bigg|\,\emptyset\bigg\rangle\nonumber,\\
\label{Z4-0}
V^{\mathbb{Z}_4}_{\emptyset\emptyset\nu}(q_0,q_b,q_c,q_a)&=&\bigg\langle\emptyset\,\bigg|\, \prod_{i=1}^{\infty}Q_{0}\Gamma_{+}(1)Q_{a}\Gamma_{+}(1)Q_{c}\Gamma_{+}(1)Q_{b}\Gamma_{+}(1)\,\\
&&\quad\cdot\prod_{i=1}^{\frac{m}{2}}Q_{0}\Gamma_{-}(1)Q_{a}\Gamma_{+}(1)Q_{c}\Gamma_{-}(1)Q_{b}\Gamma_{+}(1)\nonumber\\
&&\quad\cdot\prod_{i=1}^{\infty}Q_{0}\Gamma_{-}(1)Q_{a}\Gamma_{-}(1)Q_{c}\Gamma_{-}(1)Q_{b}\Gamma_{-}(1) \;\bigg|\,\emptyset\bigg\rangle\nonumber,
\eea
(ii) if $m\equiv1\; \mbox{(mod 4)}$, 
\bea
Z_{\mathfrak{RP}_{-}(\nu)}
&=&\bigg\langle\emptyset\,\bigg|\, \prod_{i=1}^{\infty}Q_{b}\Gamma_{+}(1)Q_{0}\Gamma_{+}^\prime(1)Q_{a}\Gamma_{+}(1)Q_{c}\Gamma_{+}^\prime(1)\label{RPC--1}\\
&&\quad\cdot\left(\prod_{i=1}^{\frac{m-1}{2}}Q_{b}\Gamma_{-}(1)Q_{0}\Gamma_{+}^\prime(1)Q_{a}\Gamma_{-}(1)Q_{c}\Gamma_{+}^\prime(1)\right)\nonumber,\\
&&\quad\cdot Q_{b}\Gamma_{-}(1)Q_{0}\Gamma_{+}^\prime(1)\cdot\prod_{i=1}^{\infty}Q_{a}\Gamma_{-}(1)Q_{c}\Gamma_{-}^\prime(1)Q_{b}\Gamma_{-}(1)Q_{0}\Gamma_{-}^\prime(1) \bigg|\,\emptyset\bigg\rangle\nonumber,\\
\label{Z4-1}
V^{\mathbb{Z}_4}_{\emptyset\emptyset\nu}(q_0,q_b,q_c,q_a)&=&\bigg\langle\emptyset\,\bigg|\, \prod_{i=1}^{\infty}Q_{b}\Gamma_{+}(1)Q_{0}\Gamma_{+}(1)Q_{a}\Gamma_{+}(1)Q_{c}\Gamma_{+}(1)\,\\
&&\quad\cdot\left(\prod_{i=1}^{\frac{m-1}{2}}Q_{b}\Gamma_{-}(1)Q_{0}\Gamma_{+}(1)Q_{a}\Gamma_{-}(1)Q_{c}\Gamma_{+}(1)\right)\nonumber\\
&&\quad\cdot Q_{b}\Gamma_{-}(1)Q_{0}\Gamma_{+}(1)\cdot\prod_{i=1}^{\infty}Q_{a}\Gamma_{-}(1)Q_{c}\Gamma_{-}(1)Q_{b}\Gamma_{-}(1)Q_{0}\Gamma_{-}(1) \bigg|\,\emptyset\bigg\rangle\nonumber,
\eea
(iii) if $m\equiv2\; \mbox{(mod 4)}$, 
\ben
Z_{\mathfrak{RP}_{-}(\nu)}
&=&\bigg\langle\emptyset\,\bigg|\, \prod_{i=1}^{\infty}Q_{c}\Gamma_{+}^\prime(1)Q_{b}\Gamma_{+}(1)Q_{0}\Gamma_{+}^\prime(1)Q_{a}\Gamma_{+}(1)\,\\
&&\quad\cdot\prod_{i=1}^{\frac{m}{2}}Q_{c}\Gamma_{-}^\prime(1)Q_{b}\Gamma_{+}(1)Q_{0}\Gamma_{-}^\prime(1)Q_{a}\Gamma_{+}(1)\\
&&\quad\cdot\prod_{i=1}^{\infty}Q_{c}\Gamma_{-}^\prime(1)Q_{b}\Gamma_{-}(1)Q_{0}\Gamma_{-}^\prime(1)Q_{a}\Gamma_{-}(1)\; \bigg|\,\emptyset\bigg\rangle,\\
V^{\mathbb{Z}_4}_{\emptyset\emptyset\nu}(q_0,q_b,q_c,q_a)&=&\bigg\langle\emptyset\,\bigg|\, \prod_{i=1}^{\infty}Q_{c}\Gamma_{+}(1)Q_{b}\Gamma_{+}(1)Q_{0}\Gamma_{+}(1)Q_{a}\Gamma_{+}(1)\,\\
&&\quad\cdot\prod_{i=1}^{\frac{m}{2}}Q_{c}\Gamma_{-}(1)Q_{b}\Gamma_{+}(1)Q_{0}\Gamma_{-}(1)Q_{a}\Gamma_{+}(1)\\
&&\quad\cdot\prod_{i=1}^{\infty}Q_{c}\Gamma_{-}(1)Q_{b}\Gamma_{-}(1)Q_{0}\Gamma_{-}(1)Q_{a}\Gamma_{-}(1)\; \bigg|\,\emptyset\bigg\rangle,
\een
(iv) if $m\equiv3\; \mbox{(mod 4)}$, 
\ben
Z_{\mathfrak{RP}_{-}(\nu)}
&=&\bigg\langle\emptyset\,\bigg|\, \prod_{i=1}^{\infty}Q_{a}\Gamma_{+}(1)Q_{c}\Gamma_{+}^\prime(1)Q_{b}\Gamma_{+}(1)Q_{0}\Gamma_{+}^\prime(1)\,\\
&&\quad\cdot\left(\prod_{i=1}^{\frac{m-1}{2}}Q_{a}\Gamma_{-}(1)Q_{c}\Gamma_{+}^\prime(1)Q_{b}\Gamma_{-}(1)Q_{0}\Gamma_{+}^\prime(1)\right)\\
&&\quad\cdot Q_{a}\Gamma_{-}(1)Q_{c}\Gamma_{+}^\prime(1)\cdot\prod_{i=1}^{\infty}Q_{b}\Gamma_{-}(1)Q_{0}\Gamma_{-}^\prime(1)Q_{a}\Gamma_{-}(1)Q_{c}\Gamma_{-}^\prime(1)\; \bigg|\,\emptyset\bigg\rangle,\\
V^{\mathbb{Z}_4}_{\emptyset\emptyset\nu}(q_0,q_b,q_c,q_a)
&=&\bigg\langle\emptyset\,\bigg|\, \prod_{i=1}^{\infty}Q_{a}\Gamma_{+}(1)Q_{c}\Gamma_{+}(1)Q_{b}\Gamma_{+}(1)Q_{0}\Gamma_{+}(1)\,\\
&&\quad\cdot\left(\prod_{i=1}^{\frac{m-1}{2}}Q_{a}\Gamma_{-}(1)Q_{c}\Gamma_{+}(1)Q_{b}\Gamma_{-}(1)Q_{0}\Gamma_{+}(1)\right)\\
&&\quad\cdot Q_{a}\Gamma_{-}(1)Q_{c}\Gamma_{+}(1)\cdot\prod_{i=1}^{\infty}Q_{b}\Gamma_{-}(1)Q_{0}\Gamma_{-}(1)Q_{a}\Gamma_{-}(1)Q_{c}\Gamma_{-}(1)\; \bigg|\,\emptyset\bigg\rangle.
\een
\end{lemma}
\begin{proof}
Combined with the two latter equivalences in Equation \eqref{odd-symmetry1} and Equation \eqref{even-symmetry1}, the proof follows from Lemma \ref{vertex-expressions-RPC+-} and Equation \eqref{Z4-coloring}.
\end{proof}
For simplicity of presentation in next subsections, we introduce the following notations. Let
\ben
&&\mathbf{M}_0(x,q)=\mathbf{M}(x,q)\cdot \mathbf{M}(-x,q),\\
&&\mathbf{M}_1(x,y,q)=\frac{\mathbf{M}(x,q)}{\mathbf{M}(xy,q)},\\
&&\mathbf{M}_2(x,q;l)=\frac{\mathbf{M}(xq^l,q)}{\mathbf{M}(x,q)\cdot(x;q)_{\infty}^l},\\
&&\widehat{\mathbf{M}}_1(x,y,q)=\mathbf{M}_1(x,y,q)\cdot \mathbf{M}_1(-x,y,q),\\
&&\widehat{\mathbf{M}}_2(x,q;l)=\mathbf{M}_2(x,q;l)\cdot\mathbf{M}_2(-x,q;l),
\een
where $l\in\mathbb{Z}$, and $x,y,q$ are commutative variables.
\begin{remark}
For commutative variables $q$ and $x$, we have 
\bea\label{identity1}
\mathbf{M}(x,q)=\mathbf{M}(xq^{-1},q)\cdot(x;q)_{\infty},
\eea
which is an identity used  for simplification in the next subsections.
\end{remark}

\subsection{1-leg DT $\mathbb{Z}_{2}\times\mathbb{Z}_{2}$-vertex and restricted pyramid configurations}
We  will present the relation between the 1-leg DT $\mathbb{Z}_{2}\times\mathbb{Z}_{2}$-vertex $V_{\emptyset\emptyset\nu}^{\mathbb{Z}_{2}\times\mathbb{Z}_{2}}$ and the generating function of restricted pyramid configurations $Z_{\mathfrak{RP}_{+}(\nu)}$ in this subsection. Following the vertex operator method of Okounkov-Reshetikhin-Vafa and Bryan-Young as in  [\cite{BY}, Lemma 7.2], we first need the following
\begin{lemma}\label{antidiag-E-move-1}
Assume  $\{z_1,z_2,z_3,z_4\}=\{0,a,b,c\}$. Let $q=q_aq_bq_cq_0$. Then 
\ben
&&\prod_{i=1}^{\infty}Q_{z_1z_2}\Gamma_{+}(1)Q_{z_3z_4}\Gamma_{+}(1)=\prod_{i=1}^{\infty}Q_{z_1z_2}\Gamma^\prime_{+}(1)Q_{z_3z_4}\Gamma_{+}(1)\cdot\prod_{i=0}^\infty E_{+}\left(\sqrt{q^iq_{z_3}q_{z_4}}\right),\\
&&\prod_{i=1}^{\infty}Q_{z_1z_2}\Gamma_{+}(1)Q_{z_3z_4}\Gamma_{+}(1)=\prod_{i=1}^{\infty}Q_{z_1z_2}\Gamma_{+}(1)Q_{z_3z_4}\Gamma_{+}^\prime(1)\cdot\prod_{i=0}^\infty E_{+}\left(\sqrt{q^i}\right),
\een	
and
\ben
&&\prod_{i=1}^{\infty}Q_{z_1z_2}\Gamma_{-}(1)Q_{z_3z_4}\Gamma_{-}(1)=\prod_{i=0}^\infty E_{-}\left(\sqrt{q^iq_{z_1}q_{z_2}}\right)\cdot\prod_{i=1}^{\infty}Q_{z_1z_2}\Gamma^\prime_{-}(1)Q_{z_3z_4}\Gamma_{-}(1),\\
&&\prod_{i=1}^{\infty}Q_{z_1z_2}\Gamma_{-}(1)Q_{z_3z_4}\Gamma_{-}(1)=\prod_{i=1}^\infty E_{-}\left(\sqrt{q^i}\right)\cdot\prod_{i=1}^{\infty}Q_{z_1z_2}\Gamma_{-}(1)Q_{z_3z_4}\Gamma^\prime_{-}(1).
\een
	
\end{lemma}
\begin{proof}
It follows from Lemmas \ref{vertex-exchange1} and \ref{vertex-exchange2} by applying $\Gamma_{\pm}(1)=\Gamma_{\pm}^\prime(1)E_{\pm}(1)$ in the suitable positions and then moving all $E_{+}$-operators (resp. $E_{-}$-operators) to the right (resp. left) handside.
\end{proof}

\begin{lemma}\label{antidiag-E-move-2}
	Assume  $\{z_1,z_2,z_3,z_4\}=\{0,a,b,c\}$, $k\in\{0,1\}$ and $l\in\mathbb{Z}_{\geq0}$. Let $q=q_aq_bq_cq_0$ and $x$ be a variable.  Then 
	\ben
	&&\prod_{i=0}^\infty E_{+}\left(x\sqrt{q^i}\right)\cdot\prod_{i=1}^{l}Q_{z_1z_2}\Gamma_{-}(1)Q_{z_3z_4}\Gamma_{+}(1)Q_{z_2z_1}\Gamma_{-}(1)Q_{z_4z_3}\Gamma_{+}(1)\cdot\prod_{i=1}^{k}Q_{z_1z_2}\Gamma_{-}(1)Q_{z_3z_4}\Gamma_{+}(1)\\
	&=&\mathbf{M}_1\left(x^2(q_{z_3}q_{z_4})^{-1},q^{2l+k},q\right)\cdot\prod_{i=1}^{l}Q_{z_1z_2}\Gamma_{-}(1)Q_{z_3z_4}\Gamma_{+}(1)Q_{z_2z_1}\Gamma_{-}(1)Q_{z_4z_3}\Gamma_{+}(1)\\
	&&\cdot\prod_{i=1}^{k}Q_{z_1z_2}\Gamma_{-}(1)Q_{z_3z_4}\Gamma_{+}(1)\cdot\prod_{i=0}^\infty E_{+}\left(xq^l\sqrt{q^{i+k}}\right)
	\een	
	and
		\ben
		&&\left(\prod_{i=1}^{l}Q_{z_1z_2}\Gamma_{-}(1)Q_{z_3z_4}\Gamma_{+}(1)Q_{z_2z_1}\Gamma_{-}(1)Q_{z_4z_3}\Gamma_{+}(1)\right)\cdot Q_{z_1z_2}\Gamma_{-}(1)Q_{z_3z_4}\Gamma_{+}(1)\\
		&=&\mathbf{M}_2(q_{z_1}q_{z_2},q;2l)\cdot\left(\prod_{i=1}^{l}Q_{z_1z_2}\Gamma_{-}(1)Q_{z_3z_4}\Gamma_{+}^\prime(1)Q_{z_2z_1}\Gamma_{-}(1)Q_{z_4z_3}\Gamma_{+}^\prime(1)\right)\\
		&&\cdot Q_{z_1z_2}\Gamma_{-}(1)Q_{z_3z_4}\Gamma_{+}^\prime(1)\cdot\prod_{i=0}^{2l} E_{+}\left(\sqrt{q^{i}}\right).
		\een	
\end{lemma}
\begin{proof}
The first equation follows from  moving each $ E_{+}\left(x\sqrt{q^i}\right)$ to the right by Lemmas  \ref{vertex-exchange1} and \ref{vertex-exchange2}. The second equation follows from the similar argument in the proof of Lemma \ref{antidiag-E-move-1}.
\end{proof}

\begin{lemma}\label{antidiag-E-move-3}
	Assume  $\{z_1,z_2,z_3,z_4\}=\{0,a,b,c\}$, $k\in\{0,1\}$ and $l\in\mathbb{Z}_{\geq1}$.  Let $q=q_aq_bq_cq_0$ and $x$ be a variable.   Then 
	\ben
	&&\prod_{i=0}^\infty E_{+}\left(x\sqrt{q^i}\right)\cdot\bigg|\prod_{i=1}^{\infty}Q_{z_1z_2}\Gamma_{-}(1)Q_{z_3z_4}\Gamma_{-}(1)\, \bigg|\,\emptyset\bigg\rangle\\
	&=&\mathbf{M}(x^2(q_{z_3}q_{z_4})^{-1},q)\cdot \mathbf{M}(x^2,q)\cdot\bigg|\prod_{i=1}^{\infty}Q_{z_1z_2}\Gamma_{-}(1)Q_{z_3z_4}\Gamma_{-}(1)\, \bigg|\,\emptyset\bigg\rangle.
	\een	
\end{lemma}
\begin{proof}
It follows from Equation \eqref{E+right}, Lemmas \ref{vertex-exchange1}, \ref{vertex-exchange2} by commuting all $E_{+}\left(x\sqrt{q^i}\right)$ out to the right.
\end{proof}

\begin{lemma}\label{antidiag-E-move-4}
	Assume  $\{z_1,z_2,z_3,z_4\}=\{0,a,b,c\}$  and $l\in\mathbb{Z}_{\geq1}$. Let $q=q_aq_bq_cq_0$.  Then 
	\ben
	&&\prod_{i=1}^{l}Q_{z_1z_2}\Gamma_{-}(1)Q_{z_3z_4}\Gamma_{+}(1)Q_{z_2z_1}\Gamma_{-}(1)Q_{z_4z_3}\Gamma_{+}(1)\\
	&=&\mathbf{M}_2(q_{z_1}q_{z_2},q;2l-1)\cdot\prod_{i=0}^{2l-1} E_{-}\left(\sqrt{q_{z_1}q_{z_2}q^{i}}\right)\cdot\prod_{i=1}^{l}Q_{z_1z_2}\Gamma_{-}^\prime(1)Q_{z_3z_4}\Gamma_{+}(1)Q_{z_2z_1}\Gamma_{-}^\prime(1)Q_{z_4z_3}\Gamma_{+}(1).
	\een	
\end{lemma}
\begin{proof}
It follows from the similar argument in the proof of Lemma \ref{antidiag-E-move-1}.	
\end{proof}

\begin{lemma}\label{antidiag-E-move-5}
	Assume  $\{z_1,z_2,z_3,z_4\}=\{0,a,b,c\}$  and $l\in\mathbb{Z}_{\geq0}$. Let $q=q_aq_bq_cq_0$ and $x$ be a variable.  Then 
	\ben
	&&\prod_{i=1}^{l}Q_{z_1z_2}\Gamma_{-}^\prime(1)Q_{z_3z_4}\Gamma_{+}(1)Q_{z_2z_1}\Gamma_{-}^\prime(1)Q_{z_4z_3}\Gamma_{+}(1)\cdot\prod_{i=0}^\infty E_{-}\left(x\sqrt{q^i}\right)\\
	&=&\mathbf{M}_1(x^2q^{-1},q^{2l},q)\cdot\prod_{i=0}^\infty E_{-}\left(xq^l\sqrt{q^i}\right)\cdot\prod_{i=1}^{l}Q_{z_1z_2}\Gamma_{-}^\prime(1)Q_{z_3z_4}\Gamma_{+}(1)Q_{z_2z_1}\Gamma_{-}^\prime(1)Q_{z_4z_3}\Gamma_{+}(1)
	\een	
	and
		\ben
		&&\left(\prod_{i=1}^{l}Q_{z_1z_2}\Gamma_{-}(1)Q_{z_3z_4}\Gamma_{+}^\prime(1)Q_{z_2z_1}\Gamma_{-}(1)Q_{z_4z_3}\Gamma_{+}^\prime(1)\right)\cdot Q_{z_1z_2}\Gamma_{-}(1)Q_{z_3z_4}\Gamma_{+}^\prime(1)\cdot\prod_{i=0}^\infty E_{-}\left(x\sqrt{q^i}\right)\\
		&=&\frac{1}{\mathbf{M}_1(x^2q^{-1},q^{2l+1},q)}\cdot\prod_{i=0}^\infty E_{-}\left(xq^l\sqrt{q^{i+1}}\right)\\
		&&\cdot\left(\prod_{i=1}^{l}Q_{z_1z_2}\Gamma_{-}(1)Q_{z_3z_4}\Gamma_{+}^\prime(1)Q_{z_2z_1}\Gamma_{-}(1)Q_{z_4z_3}\Gamma_{+}^\prime(1)\right)\cdot Q_{z_1z_2}\Gamma_{-}(1)Q_{z_3z_4}\Gamma_{+}^\prime(1)
		\een	
	
\end{lemma}
\begin{proof}
The proof follows from 	moving each $E_{-}\left(x\sqrt{q^i}\right)$ to the left by Lemmas  \ref{vertex-exchange1} and \ref{vertex-exchange2}. 
\end{proof}

\begin{lemma}\label{antidiag-E-move-6}
	Assume  $\{z_1,z_2,z_3,z_4\}=\{0,a,b,c\}$. Let $q=q_aq_bq_cq_0$ and $x$ be a variable.   Then 
	\ben
	&&\bigg\langle\emptyset\,\bigg|\, \prod_{i=1}^{\infty}Q_{z_1z_2}\Gamma_{+}^\prime(1)Q_{z_3z_4}\Gamma_{+}(1)\bigg|\cdot\prod_{i=0}^\infty E_{-}\left(x\sqrt{q^i}\right)\\
	&=&\mathbf{M}_1(x^2q^{-1},q_{z_3}q_{z_4},q)\cdot\bigg\langle\emptyset\,\bigg|\, \prod_{i=1}^{\infty}Q_{z_1z_2}\Gamma_{+}^\prime(1)Q_{z_3z_4}\Gamma_{+}(1)\bigg|,\\
	&&\bigg\langle\emptyset\,\bigg|\, \prod_{i=1}^{\infty}Q_{z_1z_2}\Gamma_{+}(1)Q_{z_3z_4}\Gamma_{+}^\prime(1)\bigg|\cdot\prod_{i=0}^\infty E_{-}\left(x\sqrt{q^i}\right)\\
	&=&\frac{1}{\mathbf{M}_1(x^2q^{-1},q_{z_3}q_{z_4},q)}\cdot\bigg\langle\emptyset\,\bigg|\, \prod_{i=1}^{\infty}Q_{z_1z_2}\Gamma_{+}(1)Q_{z_3z_4}\Gamma_{+}^\prime(1)\bigg|.
	\een	
\end{lemma}
\begin{proof}
It follows from Equation \eqref{E-left}, Lemmas \ref{vertex-exchange1}, \ref{vertex-exchange2} by commuting all $E_{-}\left(x\sqrt{q^i}\right)$ out to the left.		
\end{proof}
Now we have
\begin{lemma}\label{Z2Z2-RPC+}
Assume  $\nu=(m,m-1,\cdots,2,1)$ with $m\in\mathbb{Z}_{\geq1}$. Let $q=q_aq_bq_cq_0$. Then we have  \\
(i) if $m\equiv0\; \mbox{(mod 4)}$ or $m\equiv3\; \mbox{(mod 4)}$,
\ben
V_{\emptyset\emptyset\nu}^{\mathbb{Z}_{2}\times\mathbb{Z}_{2}}
=\widetilde{\mathbf{M}}(q_aq_b,q)\cdot\widetilde{\mathbf{M}}_{2\{\frac{m}{2}\}}\left(q_aq_b,q;2\left\lfloor\frac{m+1}{2}\right\rfloor\right)\cdot\left( Z_{\mathfrak{RP}_{+}(\nu)}\bigg|_{q_a\leftrightarrow q_b}\right),
\een	
(ii) if $m\equiv1\; \mbox{(mod 4)}$ or $m\equiv2\; \mbox{(mod 4)}$,
\ben
V_{\emptyset\emptyset\nu}^{\mathbb{Z}_{2}\times\mathbb{Z}_{2}}
=\widetilde{\mathbf{M}}(q_aq_b,q)\cdot\widetilde{\mathbf{M}}_{2\{\frac{m}{2}\}}\left(q_aq_b,q;2\left\lfloor\frac{m+1}{2}\right\rfloor\right)\cdot\left( Z_{\mathfrak{RP}_{+}(\nu)}\bigg|_{q_0\leftrightarrow q_c}\right).
\een
\end{lemma}	
\begin{proof}
We first deal with the case when $m\equiv0\; \mbox{(mod 4)}$. In Equation \eqref{Z2Z2-0}, we  apply $\Gamma_{+}(1)=\Gamma_{+}^\prime(1)E_{+}(1)$  in the position  corresponding to that of $\Gamma_{+}^\prime(1)$ in Equation \eqref{RPC+0},  and then moving all the appearence of $E_{+}$-operators to the right by Lemmas \ref{antidiag-E-move-1}, \ref{antidiag-E-move-2} and \ref{antidiag-E-move-3}. Then we have 
\bea
V_{\emptyset\emptyset\nu}^{\mathbb{Z}_{2}\times\mathbb{Z}_{2}}&=&C_1(\{q_{\star}\};m)\cdot\bigg\langle\emptyset\,\bigg|\, \prod_{i=1}^{\infty}Q_{0c}\Gamma_{+}^\prime(1)Q_{ab}\Gamma_{+}(1)\cdot\prod_{i=1}^{\frac{m}{2}}Q_{0c}\Gamma_{-}(1)Q_{ba}\Gamma_{+}(1)Q_{c0}\Gamma_{-}(1)Q_{ab}\Gamma_{+}(1)\nonumber\\
&&\quad\quad\quad\quad\quad\quad\quad \cdot\prod_{i=1}^{\infty}Q_{0c}\Gamma_{-}(1)Q_{ba}\Gamma_{-}(1)\, \bigg|\,\emptyset\bigg\rangle\label{E+toR1}
\eea
where 
\ben
C_1(\{q_{\star}\};m)=\mathbf{M}_1\left(1,q^{m},q\right)\cdot\mathbf{M}(q^{m},q)\cdot \mathbf{M}(q^{m}q_aq_b,q).
\een
In Equation \eqref{E+toR1}, one applies $\Gamma_{-}(1)=\Gamma_{-}^\prime(1)E_{-}(1)$  in the positions corresponding to those of $\Gamma_{-}^\prime(1)$ in Equation \eqref{RPC+0},  and then moving all the appearence of $E_{-}$-operators to the left by Lemmas \ref{antidiag-E-move-4}, \ref{antidiag-E-move-1}, \ref{antidiag-E-move-5} and \ref{antidiag-E-move-6}. Then we have 
\ben
V_{\emptyset\emptyset\nu}^{\mathbb{Z}_{2}\times\mathbb{Z}_{2}}&=&C_1(\{q_{\star}\};m)\cdot C_2(\{q_{\star}\};m)\\
&&\cdot\bigg\langle\emptyset\,\bigg|\, \prod_{i=1}^{\infty}Q_{0c}\Gamma_{+}^\prime(1)Q_{ab}\Gamma_{+}(1)\cdot\prod_{i=1}^{\frac{m}{2}}Q_{0c}\Gamma^\prime_{-}(1)Q_{ba}\Gamma_{+}(1)Q_{c0}\Gamma^\prime_{-}(1)Q_{ab}\Gamma_{+}(1)\nonumber\\
&&\quad \cdot\prod_{i=1}^{\infty}Q_{0c}\Gamma^\prime_{-}(1)Q_{ba}\Gamma_{-}(1)\, \bigg|\,\emptyset\bigg\rangle\\
&=&C_1(\{q_{\star}\};m)\cdot C_2(\{q_{\star}\};m) \cdot\left( Z_{\mathfrak{RP}_{+}(\nu)}\bigg|_{q_a\leftrightarrow q_b}\right)
\een
where 
\ben
C_2(\{q_{\star}\};m)=\mathbf{M}_2(q_0q_c,q;m-1)\cdot\mathbf{M}_1((q_aq_b)^{-1},q^m,q)\cdot \mathbf{M}_1((q_aq_b)^{-1},q_aq_b,q).
\een
By  Equation \eqref{identity1},  the proof is completed by the direct computation of $C_1(\{q_{\star}\};m)\cdot C_2(\{q_{\star}\};m)$.

Next, we treat the case when $m\equiv1\; \mbox{(mod 4)}$. In Equation \eqref{Z2Z2-1}, we substitute $\Gamma_{+}(1)=\Gamma_{+}^\prime(1)E_{+}(1)$ into the positions
corresponding to those of $\Gamma_{+}^\prime(1)$ in Equation \eqref{RPC+1} and then move all $E_{+}$-operators to the right by Lemmas \ref{antidiag-E-move-1}, \ref{antidiag-E-move-2} and \ref{antidiag-E-move-3}. Then we have 
\bea
V_{\emptyset\emptyset\nu}^{\mathbb{Z}_{2}\times\mathbb{Z}_{2}}
&=&C_{3}(\{q_{\star}\};m)\cdot
\bigg\langle\emptyset\,\bigg|\, \prod_{i=1}^{\infty}Q_{ab}\Gamma_{+}(1)Q_{0c}\Gamma^\prime_{+}(1)\cdot\left(\prod_{i=1}^{\frac{m-1}{2}}Q_{ab}\Gamma_{-}(1)Q_{c0}\Gamma^\prime_{+}(1)Q_{ba}\Gamma_{-}(1)Q_{0c}\Gamma^\prime_{+}(1)\right)\nonumber\\
&&\quad\quad\quad\quad\quad\quad\quad\quad\cdot Q_{ab}\Gamma_{-}(1)Q_{c0}\Gamma^\prime_{+}(1)
\cdot\prod_{i=1}^{\infty}Q_{ba}\Gamma_{-}(1)Q_{0c}\Gamma_{-}(1)\, \bigg|\,\emptyset\bigg\rangle\label{E+toR2}
\eea
where 
\ben
C_{3}(\{q_{\star}\};m)=\mathbf{M}_1(q^{-1}q_aq_b,q^m,q)\cdot\mathbf{M}_2(q_aq_b,q;m-1)\cdot\mathbf{M}(q^{-1}(q_aq_b),q)\cdot\mathbf{M}(1,q).
\een
In Equation \eqref{E+toR2}, one applies $\Gamma_{-}(1)=\Gamma_{-}^\prime(1)E_{-}(1)$  in the positions corresponding to that of $\Gamma_{-}^\prime(1)$ in Equation \eqref{RPC+1},  and then move all  $E_{-}$-operators to the left by Lemmas \ref{antidiag-E-move-1}, \ref{antidiag-E-move-5} and \ref{antidiag-E-move-6}. By comparing with Equation \eqref{RPC+1},  we have 
\ben
V_{\emptyset\emptyset\nu}^{\mathbb{Z}_{2}\times\mathbb{Z}_{2}}
=C_3(\{q_{\star}\};m)\cdot C_4(\{q_{\star}\};m) \cdot\left( Z_{\mathfrak{RP}_{+}(\nu)}\bigg|_{q_0\leftrightarrow q_c}\right)
\een
where 
\ben
C_4(\{q_{\star}\};m)=\frac{1}{\mathbf{M}_1(1,q^m,q)\cdot\mathbf{M}_1(q^m,q_0q_c,q)}.
\een
Then the proof follows from the direct computation of $C_3(\{q_{\star}\};m)\cdot C_4(\{q_{\star}\};m)$ when $m\equiv1\; \mbox{(mod 4)}$.

The other cases are similar, or can be simply proved just by comparing with the above two cases.
\end{proof}

\subsection{1-leg DT $\mathbb{Z}_{4}$-vertex and restricted pyramid configurations}
In this subsection, we  will present the relation between the 1-leg DT $\mathbb{Z}_{4}$-vertex $V_{\emptyset\emptyset\nu}^{\mathbb{Z}_{4}}$ and the generating function of restricted pyramid configurations $Z_{\mathfrak{RP}_{-}(\nu)}$. Here we take the convention of identifying variables and operators mentioned in Section 4.1.  As in Section 4.2, we first derive the following

\begin{lemma}\label{diag-E-move-1}
	Assume  $\{z_1,z_2,z_3,z_4\}=\{0,a,b,c\}$. Let $q=q_aq_bq_cq_0$.  Then 
	\ben
	&&\prod_{i=1}^\infty Q_{z_1}\Gamma_{+}(1)Q_{z_2}\Gamma_{+}(1)Q_{z_3}\Gamma_{+}(1)Q_{z_4}\Gamma_{+}(1)\\
	&=&\left(\prod_{i=1}^\infty Q_{z_1}\Gamma^\prime_{+}(1)Q_{z_2}\Gamma_{+}(1)Q_{z_3}\Gamma_{+}^\prime(1)Q_{z_4}\Gamma_{+}(1)\right)\cdot\prod_{i=0}^\infty E_{+}(q^iq_{z_4})E_{+}(q^iq_{z_2}q_{z_3}q_{z_4}),\\
	&&\prod_{i=1}^\infty Q_{z_1}\Gamma_{+}(1)Q_{z_2}\Gamma_{+}(1)Q_{z_3}\Gamma_{+}(1)Q_{z_4}\Gamma_{+}(1)\\
	&=&\left(\prod_{i=1}^\infty Q_{z_1}\Gamma_{+}(1)Q_{z_2}\Gamma^\prime_{+}(1)Q_{z_3}\Gamma_{+}(1)Q_{z_4}\Gamma^\prime_{+}(1)\right)\cdot\prod_{i=0}^\infty E_{+}(q^i)E_{+}(q^iq_{z_3}q_{z_4})
	\een
	and 
	\ben
	&&\prod_{i=1}^\infty Q_{z_1}\Gamma_{-}(1)Q_{z_2}\Gamma_{-}(1)Q_{z_3}\Gamma_{-}(1)Q_{z_4}\Gamma_{-}(1)\\
	&=&\prod_{i=0}^\infty E_{-}(q^iq_{z_1})E_{-}(q^iq_{z_1}q_{z_2}q_{z_3})\cdot\left(\prod_{i=1}^\infty Q_{z_1}\Gamma^\prime_{-}(1)Q_{z_2}\Gamma_{-}(1)Q_{z_3}\Gamma_{-}^\prime(1)Q_{z_4}\Gamma_{-}(1)\right),\\
	&&\prod_{i=1}^\infty Q_{z_1}\Gamma_{-}(1)Q_{z_2}\Gamma_{-}(1)Q_{z_3}\Gamma_{-}(1)Q_{z_4}\Gamma_{-}(1)\\
	&=&\prod_{i=0}^\infty E_{-}(q^iq_{z_1}q_{z_2})E_{-}(q^{i+1})\cdot\left(\prod_{i=1}^\infty Q_{z_1}\Gamma_{-}(1)Q_{z_2}\Gamma^\prime_{-}(1)Q_{z_3}\Gamma_{-}(1)Q_{z_4}\Gamma^\prime_{-}(1)\right)
	\een
	
\end{lemma}
\begin{proof}
It follows from the similar argument in the proof of Lemma \ref{antidiag-E-move-1}.	
\end{proof}

\begin{lemma}\label{diag-E-move-2}
	Assume  $\{z_1,z_2,z_3,z_4\}=\{0,a,b,c\}$, $k\in\{0,1\}$ and $l\in\mathbb{Z}_{\geq0}$. Let $q=q_aq_bq_cq_0$ and $x$ be a variable.  Then 
	\ben
	&&\prod_{i=0}^\infty E_{+}(xq^i)\cdot\prod_{i=1}^l Q_{z_1}\Gamma_{-}(1)Q_{z_2}\Gamma_{+}(1)Q_{z_3}\Gamma_{-}(1)Q_{z_4}\Gamma_{+}(1)\cdot\prod_{i=1}^k Q_{z_1}\Gamma_{-}(1)Q_{z_2}\Gamma_{+}(1)\\
	&=&\widehat{\mathbf{M}}_1(xq^{-1}q_{z_1},q^{l+k},q)\cdot\widehat{\mathbf{M}}_1(xq_{z_4}^{-1},q^l,q)\cdot\prod_{i=1}^l Q_{z_1}\Gamma_{-}(1)Q_{z_2}\Gamma_{+}(1)Q_{z_3}\Gamma_{-}(1)Q_{z_4}\Gamma_{+}(1)\\
	&&\cdot\prod_{i=1}^k Q_{z_1}\Gamma_{-}(1)Q_{z_2}\Gamma_{+}(1)\cdot\prod_{i=0}^\infty E_{+}(xq^{l+i}q_{z_1}^kq_{z_2}^k)
	\een
	and 
	\ben
	&&\left(\prod_{i=1}^l Q_{z_1}\Gamma_{-}(1)Q_{z_2}\Gamma_{+}(1)Q_{z_3}\Gamma_{-}(1)Q_{z_4}\Gamma_{+}(1)\right)\cdot Q_{z_1}\Gamma_{-}(1)Q_{z_2}\Gamma_{+}(1)\\
	&=&\widehat{\mathbf{M}}_2(q_{z_1},q;l)\cdot\widehat{\mathbf{M}}_2(qq_{z_2}^{-1},q;l)\cdot\widehat{\mathbf{M}}_2(qq_{z_4}^{-1},q;l-1)\cdot\widehat{\mathbf{M}}_2(q_{z_3},q;l)\\
	&&\cdot\left(\prod_{i=1}^l Q_{z_1}\Gamma_{-}(1)Q_{z_2}\Gamma^\prime_{+}(1)Q_{z_3}\Gamma_{-}(1)Q_{z_4}\Gamma_{+}^\prime(1)\right)
	\cdot Q_{z_1}\Gamma_{-}(1)Q_{z_2}\Gamma^\prime_{+}(1)\cdot\prod_{i=0}^{l-1} E_{+}(q^{i}q_{z_1}q_{z_2})\cdot\prod_{i=0}^{l}E_{+}(q^{i})
	\een
\end{lemma}
\begin{proof}
It follows from the similar argument in the proof of Lemma \ref{antidiag-E-move-2}.	
\end{proof}

\begin{lemma}\label{diag-E-move-3}
	Assume  $\{z_1,z_2,z_3,z_4\}=\{0,a,b,c\}$. Let $q=q_aq_bq_cq_0$ and $x$ be a variable.  Then 
	\ben
	&&\prod_{i=0}^\infty E_{+}(xq^i)\cdot\bigg|\prod_{i=1}^\infty Q_{z_1}\Gamma_{-}(1)Q_{z_2}\Gamma_{-}(1)Q_{z_3}\Gamma_{-}(1)Q_{z_4}\Gamma_{-}(1)\, \bigg|\,\emptyset\bigg\rangle\\
	&=&\mathbf{M}_0(x,q)\cdot\mathbf{M}_0(xq_{z_4}^{-1},q)\cdot\mathbf{M}_0(x(q_{z_3}q_{z_4})^{-1},q)\cdot\mathbf{M}_0(xq^{-1}q_{z_1},q)\\
	&&\cdot\bigg|\prod_{i=1}^\infty Q_{z_1}\Gamma_{-}(1)Q_{z_2}\Gamma_{-}(1)Q_{z_3}\Gamma_{-}(1)Q_{z_4}\Gamma_{-}(1)\,\bigg|\,\emptyset\bigg\rangle
	\een
	
\end{lemma}
\begin{proof}
It follows from the similar argument in the proof of Lemma \ref{antidiag-E-move-3}.
\end{proof}

\begin{lemma}\label{diag-E-move-4}
	Assume  $\{z_1,z_2,z_3,z_4\}=\{0,a,b,c\}$ and $l\in\mathbb{Z}_{\geq1}$. Let $q=q_aq_bq_cq_0$. Then 
	\ben
	&&\prod_{i=1}^l Q_{z_1}\Gamma_{-}(1)Q_{z_2}\Gamma_{+}(1)Q_{z_3}\Gamma_{-}(1)Q_{z_4}\Gamma_{+}(1)\\
	&=&\widehat{\mathbf{M}}_2(q_{z_1},q;l-1)\cdot\widehat{\mathbf{M}}_2(qq_{z_2}^{-1},q;l-1)\cdot\widehat{\mathbf{M}}_2(qq_{z_4}^{-1},q;l-1)\cdot\widehat{\mathbf{M}}_2(q_{z_3},q;l)\\
	&&\cdot\prod_{i=1}^l E_{-}(q^{i-1}q_{z_1})E_{-}(q^{i}q_{z_4}^{-1})\cdot\left(\prod_{i=1}^l Q_{z_1}\Gamma^\prime_{-}(1)Q_{z_2}\Gamma_{+}(1)Q_{z_3}\Gamma_{-}^\prime(1)Q_{z_4}\Gamma_{+}(1)\right)
	\een
\end{lemma}
\begin{proof}
It follows from the similar argument in the proof of Lemma \ref{antidiag-E-move-4}.
\end{proof}
\begin{lemma}\label{diag-E-move-5}
	Assume  $\{z_1,z_2,z_3,z_4\}=\{0,a,b,c\}$ and $l\in\mathbb{Z}_{\geq0}$. Let $q=q_aq_bq_cq_0$ and $x$ be a variable.   Then 
	\ben
	&&\prod_{i=1}^l Q_{z_1}\Gamma_{-}^\prime(1)Q_{z_2}\Gamma_{+}(1)Q_{z_3}\Gamma_{-}^\prime(1)Q_{z_4}\Gamma_{+}(1)\cdot\prod_{i=0}^\infty E_{-}(xq^i)\\
	&=&\widehat{\mathbf{M}}_1(xq^{-1},q^l,q)\cdot\widehat{\mathbf{M}}_1(x(q_{z_1}q_{z_2})^{-1},q^l,q)\cdot\prod_{i=0}^\infty E_{-}(xq^{l+i})\cdot\prod_{i=1}^l Q_{z_1}\Gamma_{-}^\prime(1)Q_{z_2}\Gamma_{+}(1)Q_{z_3}\Gamma_{-}^\prime(1)Q_{z_4}\Gamma_{+}(1)
	\een
	and 
   	\ben
   	&&\left(\prod_{i=1}^l Q_{z_1}\Gamma_{-}(1)Q_{z_2}\Gamma^\prime_{+}(1)Q_{z_3}\Gamma_{-}(1)Q_{z_4}\Gamma^\prime_{+}(1)\right)\cdot Q_{z_1}\Gamma_{-}(1)Q_{z_2}\Gamma^\prime_{+}(1)\cdot\prod_{i=0}^\infty E_{-}(xq^i)\\
   	&=&\frac{1}{\widehat{\mathbf{M}}_1(xq^{-1},q^{l+1},q)\cdot\widehat{\mathbf{M}}_1(x(q_{z_3}q_{z_4})^{-1},q^{l},q)}\\
   	&&\cdot\prod_{i=0}^\infty E_{-}(xq^{l+i}q_{z_1}q_{z_2})\cdot\left(\prod_{i=1}^l Q_{z_1}\Gamma_{-}(1)Q_{z_2}\Gamma^\prime_{+}(1)Q_{z_3}\Gamma_{-}(1)Q_{z_4}\Gamma^\prime_{+}(1)\right)\cdot Q_{z_1}\Gamma_{-}(1)Q_{z_2}\Gamma^\prime_{+}(1)
   	\een
	
\end{lemma}
\begin{proof}
It follows from the similar argument in the proof of Lemma \ref{antidiag-E-move-5}.
\end{proof}

\begin{lemma}\label{diag-E-move-6}
	Assume  $\{z_1,z_2,z_3,z_4\}=\{0,a,b,c\}$.  Let $q=q_aq_bq_cq_0$ and $x$ be a variable.  Then 
	
	\ben
	&&\bigg\langle\emptyset\,\bigg|\,\prod_{i=1}^\infty Q_{z_1}\Gamma_{+}^\prime(1)Q_{z_2}\Gamma_{+}(1)Q_{z_3}\Gamma_{+}^\prime(1)Q_{z_4}\Gamma_{+}(1)\bigg|\cdot\prod_{i=0}^\infty E_{-}(xq^i)\\
	&=&\widehat{\mathbf{M}}_1(xq^{-1},qq_{z_1}^{-1},q)\cdot\widehat{\mathbf{M}}_1(x(q_{z_1}q_{z_2})^{-1},q_{z_3}^{-1},q)\cdot\bigg\langle\emptyset\,\bigg|\,\prod_{i=1}^\infty Q_{z_1}\Gamma_{+}^\prime(1)Q_{z_2}\Gamma_{+}(1)Q_{z_3}\Gamma_{+}^\prime(1)Q_{z_4}\Gamma_{+}(1)\bigg|,\\
	&&\bigg\langle\emptyset\,\bigg|\,\prod_{i=1}^\infty Q_{z_1}\Gamma_{+}(1)Q_{z_2}\Gamma_{+}^\prime(1)Q_{z_3}\Gamma_{+}(1)Q_{z_4}\Gamma^\prime_{+}(1)\bigg|\cdot\prod_{i=0}^\infty E_{-}(xq^i)\\
	&=&\frac{1}{\widehat{\mathbf{M}}_1(xq^{-1},qq_{z_1}^{-1},q)\cdot\widehat{\mathbf{M}}_1(x(q_{z_1}q_{z_2})^{-1},q_{z_3}^{-1},q)}\cdot\bigg\langle\emptyset\,\bigg|\,\prod_{i=1}^\infty Q_{z_1}\Gamma_{+}(1)Q_{z_2}\Gamma_{+}^\prime(1)Q_{z_3}\Gamma_{+}(1)Q_{z_4}\Gamma^\prime_{+}(1)\bigg|
	\een
\end{lemma}
\begin{proof}
It follows from the similar argument in the proof of Lemma \ref{antidiag-E-move-6}.	
\end{proof}

Then we have
\begin{lemma}\label{Z4-RPC-}
Assume  $\nu=(m,m-1,\cdots,2,1)$ with $m\in\mathbb{Z}_{\geq1}$. Let $q=q_aq_bq_cq_0$. Then we have   \\
(i) if $m\equiv0\; \mbox{(mod 4)}$ or $m\equiv3\; \mbox{(mod 4)}$,
\ben
Z_{\mathfrak{RP}_{-}(\nu)}
=\frac{V^{\mathbb{Z}_4}_{\emptyset\emptyset\nu}(q_0,q_b,q_c,q_a)\cdot\widehat{\mathbf{M}}(q_a,q)\cdot\widehat{\mathbf{M}}(q_b,q)\cdot\widehat{\mathbf{M}}(q_c,q)\cdot\widehat{\mathbf{M}}(q_aq_bq_c,q)}{\widehat{\mathbf{M}}_{2\{\frac{m}{2}\}}(q_a,q;\lfloor\frac{m+1}{2}\rfloor)\cdot\widehat{\mathbf{M}}_{2\{\frac{m}{2}\}}(q_b,q;\lfloor\frac{m+1}{2}\rfloor)\cdot\widehat{\mathbf{M}}_{1-2\{\frac{m}{2}\}}(q_c,q;\lfloor\frac{m+1}{2}\rfloor)\cdot\widehat{\mathbf{M}}_{2\{\frac{m}{2}\}}(q_aq_bq_c,q;\lfloor\frac{m+1}{2}\rfloor)},
\een
(ii) if $m\equiv1\; \mbox{(mod 4)}$ or $m\equiv2\; \mbox{(mod 4)}$,
\ben
Z_{\mathfrak{RP}_{-}(\nu)}
=\frac{V^{\mathbb{Z}_4}_{\emptyset\emptyset\nu}(q_0,q_b,q_c,q_a)\cdot\widehat{\mathbf{M}}(q_a,q)\cdot\widehat{\mathbf{M}}(q_b,q)\cdot\widehat{\mathbf{M}}(q_0,q)\cdot\widehat{\mathbf{M}}(q_aq_bq_0,q)}{\widehat{\mathbf{M}}_{2\{\frac{m}{2}\}}(q_a,q;\lfloor\frac{m+1}{2}\rfloor)\cdot\widehat{\mathbf{M}}_{2\{\frac{m}{2}\}}(q_b,q;\lfloor\frac{m+1}{2}\rfloor)\cdot\widehat{\mathbf{M}}_{1-2\{\frac{m}{2}\}}(q_0,q;\lfloor\frac{m+1}{2}\rfloor)\cdot\widehat{\mathbf{M}}_{2\{\frac{m}{2}\}}(q_aq_bq_0,q;\lfloor\frac{m+1}{2}\rfloor)}.
\een		
\end{lemma}

\begin{proof}
As in Lemma \ref{Z2Z2-RPC+}, we first deal with the case when $m\equiv0\; \mbox{(mod 4)}$. In Equation \eqref{Z4-0}, we substitute $\Gamma_{+}(1)=\Gamma_{+}^\prime(1)E_{+}(1)$  into the positions  corresponding to those of $\Gamma_{+}^\prime(1)$ in Equation \eqref{RPC-0},  and then move all  $E_{+}$-operators to the right by Lemmas \ref{diag-E-move-1}, \ref{diag-E-move-2} and \ref{diag-E-move-3}. Then we have 
\bea\label{E+toR3}
V^{\mathbb{Z}_4}_{\emptyset\emptyset\nu}(q_0,q_b,q_c,q_a)&=&C_5(\{q_{\star}\};m)\cdot\bigg\langle\emptyset\,\bigg|\, \prod_{i=1}^{\infty}Q_{0}\Gamma^\prime_{+}(1)Q_{a}\Gamma_{+}(1)Q_{c}\Gamma^\prime_{+}(1)Q_{b}\Gamma_{+}(1)\,\\
&&\quad\quad\quad\quad\quad\quad\quad\cdot\prod_{i=1}^{\frac{m}{2}}Q_{0}\Gamma_{-}(1)Q_{a}\Gamma_{+}(1)Q_{c}\Gamma_{-}(1)Q_{b}\Gamma_{+}(1)\nonumber\\
&&\quad\quad\quad\quad\quad\quad\quad\cdot\prod_{i=1}^{\infty}Q_{0}\Gamma_{-}(1)Q_{a}\Gamma_{-}(1)Q_{c}\Gamma_{-}(1)Q_{b}\Gamma_{-}(1) \;\bigg|\,\emptyset\bigg\rangle\nonumber
\eea
where 
\ben
C_5(\{q_{\star}\};m)&=&\widehat{\mathbf{M}}_1\left(1,q^{\frac{m}{2}},q\right)^2\cdot\widehat{\mathbf{M}}_1\left(q_aq_c,q^{\frac{m}{2}},q\right)\cdot\widehat{\mathbf{M}}_1\left((q_aq_c)^{-1},q^{\frac{m}{2}},q\right)\\
&&\cdot\mathbf{M}_0\left(q^{\frac{m}{2}}q_aq_bq_c,q\right)\cdot\mathbf{M}_0\left(q^{\frac{m}{2}}q_aq_c,q\right)\cdot\mathbf{M}_0\left(q^{\frac{m}{2}}q_a,q\right)\cdot\mathbf{M}_0\left(q^{\frac{m}{2}},q\right)^2\\
&&\cdot\mathbf{M}_0\left(q^{\frac{m}{2}}q_b,q\right)\cdot\mathbf{M}_0\left(q^{\frac{m}{2}}(q_aq_c)^{-1},q\right)\cdot\mathbf{M}_0\left(q^{\frac{m}{2}}q_c^{-1},q\right)
\een
In Equation \eqref{E+toR3}, one substitutes $\Gamma_{-}(1)=\Gamma_{-}^\prime(1)E_{-}(1)$  into the positions corresponding to those of $\Gamma_{-}^\prime(1)$ in Equation \eqref{RPC-0},  and then move all  $E_{-}$-operators to the left by Lemmas \ref{diag-E-move-4}, \ref{diag-E-move-1}, \ref{diag-E-move-5} and \ref{diag-E-move-6}. Then we have 
\ben
V^{\mathbb{Z}_4}_{\emptyset\emptyset\nu}(q_0,q_b,q_c,q_a)=C_5(\{q_{\star}\};m)\cdot C_6(\{q_{\star}\};m)\cdot Z_{\mathfrak{RP}_{-}(\nu)}
\een
where 
\ben
C_6(\{q_{\star}\};m)&=&\widehat{\mathbf{M}}_2\left(q_0,q;\frac{m}{2}-1\right)\cdot\widehat{\mathbf{M}}_2\left(qq_a^{-1},q;\frac{m}{2}-1\right)\cdot\widehat{\mathbf{M}}_2\left(qq_b^{-1},q;\frac{m}{2}-1\right)\cdot\widehat{\mathbf{M}}_2\left(q_c,q;\frac{m}{2}\right)\\
&&\cdot\widehat{\mathbf{M}}_1\left(q_0q^{-1},q^{\frac{m}{2}},q\right)\cdot\widehat{\mathbf{M}}_1\left(q_a^{-1},q^{\frac{m}{2}},q\right)\cdot\widehat{\mathbf{M}}_1\left(q_b^{-1},q^{\frac{m}{2}},q\right)\cdot\widehat{\mathbf{M}}_1\left(q_c,q^{\frac{m}{2}},q\right)\\
&&\cdot\widehat{\mathbf{M}}_1\left(q_0q^{-1},qq_{0}^{-1},q\right)\cdot\widehat{\mathbf{M}}_1\left(q_a^{-1},q_{c}^{-1},q\right)\cdot\widehat{\mathbf{M}}_1\left(q_b^{-1},qq_{0}^{-1},q\right)\cdot\widehat{\mathbf{M}}_1\left(q_c,q_{c}^{-1},q\right).
\een
Then the proof follows by the direct computation of $C_5(\{q_{\star}\};m)\cdot C_6(\{q_{\star}\};m)$ 
when $m\equiv0\; \mbox{(mod 4)}$.

We also include the proof in  the case when $m\equiv1\; \mbox{(mod 4)}$ as follows. As in the above case, comparing Equation \eqref{Z4-1} with Equation \eqref{RPC--1}, and by  Lemmas \ref{diag-E-move-1}, \ref{diag-E-move-2} and \ref{diag-E-move-3}, we first have
\ben
V^{\mathbb{Z}_4}_{\emptyset\emptyset\nu}(q_0,q_b,q_c,q_a)&=&C_7(\{q_{\star}\};m)\cdot\bigg\langle\emptyset\,\bigg|\, \prod_{i=1}^{\infty}Q_{b}\Gamma_{+}(1)Q_{0}\Gamma_{+}^\prime(1)Q_{a}\Gamma_{+}(1)Q_{c}\Gamma_{+}^\prime(1)\label{RPC-1}\\
&&\quad\quad\quad\quad\quad\quad\quad\cdot\left(\prod_{i=1}^{\frac{m-1}{2}}Q_{b}\Gamma_{-}(1)Q_{0}\Gamma_{+}^\prime(1)Q_{a}\Gamma_{-}(1)Q_{c}\Gamma_{+}^\prime(1)\right)\\
&&\quad\quad\quad\quad\quad\quad\quad\cdot Q_{b}\Gamma_{-}(1)Q_{0}\Gamma_{+}^\prime(1)\cdot\prod_{i=1}^{\infty}Q_{a}\Gamma_{-}(1)Q_{c}\Gamma_{-}(1)Q_{b}\Gamma_{-}(1)Q_{0}\Gamma_{-}(1) \bigg|\,\emptyset\bigg\rangle\nonumber
\een
where 
\ben
C_7(\{q_{\star}\};m)&=&\widehat{\mathbf{M}}_1\left(q_bq^{-1},q^{\frac{m+1}{2}},q\right)\cdot\widehat{\mathbf{M}}_1\left(q_c^{-1},q^{\frac{m-1}{2}},q\right)\cdot\widehat{\mathbf{M}}_1\left(q_0^{-1},q^{\frac{m+1}{2}},q\right)\cdot\widehat{\mathbf{M}}_1\left(q_a,q^{\frac{m-1}{2}},q\right)\\
&&\cdot\widehat{\mathbf{M}}_2\left(q_b,q;\frac{m-1}{2}\right)\cdot\widehat{\mathbf{M}}_2\left(qq_0^{-1},q;\frac{m-1}{2}\right)\cdot\widehat{\mathbf{M}}_2\left(qq_c^{-1},q;\frac{m-1}{2}\right)\cdot\widehat{\mathbf{M}}_2\left(q_a,q;\frac{m-1}{2}\right)\\
&&\cdot\mathbf{M}_0(1,q)^2\cdot\mathbf{M}_0(q_0q_b,q)\cdot\mathbf{M}_0(q_b,q)\cdot\mathbf{M}_0(q_c^{-1},q)\\
&&\cdot\mathbf{M}_0(q_0^{-1},q)\cdot\mathbf{M}_0((q_0q_b)^{-1},q)\cdot\mathbf{M}_0(q^{-1}q_a,q)
\een
Then by Lemmas \ref{diag-E-move-1}, \ref{diag-E-move-5} and \ref{diag-E-move-6}, we have
\ben
V^{\mathbb{Z}_4}_{\emptyset\emptyset\nu}(q_0,q_b,q_c,q_a)=C_7(\{q_{\star}\};m)\cdot C_8(\{q_{\star}\};m)\cdot Z_{\mathfrak{RP}_{-}(\nu)}
\een
where 
\ben
&&C_8(\{q_{\star}\};m)\\
&=&\left(\widehat{\mathbf{M}}_1\left((q_0q_b)^{-1},q^{\frac{m+1}{2}},q\right)\cdot\widehat{\mathbf{M}}_1\left(1,q^{\frac{m-1}{2}},q\right)\cdot\widehat{\mathbf{M}}_1\left(1,q^{\frac{m+1}{2}},q\right)\cdot\widehat{\mathbf{M}}_1\left(q_0q_b,q^{\frac{m-1}{2}},q\right)\right)^{-1}\\
&&\cdot\left(\widehat{\mathbf{M}}_1\left(q^{\frac{m-1}{2}},qq_b^{-1},q\right)\cdot\widehat{\mathbf{M}}_1\left(q^{\frac{m-1}{2}}q_aq_c,q_a^{-1},q\right)\cdot\widehat{\mathbf{M}}_1\left(q^{\frac{m-1}{2}}q_0q_b,qq_b^{-1},q\right)\cdot\widehat{\mathbf{M}}_1\left(q^{\frac{m+1}{2}},q_a^{-1},q\right)\right)^{-1}
\een
By the direct computation of $C_7(\{q_{\star}\};m)\cdot C_8(\{q_{\star}\};m)$, we complete the proof when $m\equiv1\; \mbox{(mod 4)}$.

The other cases are similar, or can be simply proved just by comparing with the above two cases.
\end{proof}

\begin{remark}\label{another-proof-GFPP}
When $m\equiv0\; \mbox{(mod 4)}$, the argument  in the proof of  Lemma \ref{Z4-RPC-}  also holds for $m=0$ or equivalently $\nu=\emptyset$, and hence  we have 
	\ben
	Z_\mathfrak{P}=Z_{\mathfrak{RP}_{-}(\emptyset)}=Z_{\mathfrak{RP}_{+}(\emptyset)}=\mathbf{M}(1,q)^4\cdot\frac{
		\widetilde{\mathbf{M}}(q_aq_c,q)\cdot\widetilde{\mathbf{M}}(q_bq_c,q)}{\widetilde{\mathbf{M}}(-q_a,q)\cdot\widetilde{\mathbf{M}}(-q_b,q)\cdot\widetilde{\mathbf{M}}(-q_c,q)\cdot\widetilde{\mathbf{M}}(-q_aq_bq_c,q)}
	\een
	by Equation \eqref{Z4-nolegs} and Remark \ref{empty2}, which can be viewed as an alternative proof of  [\cite{BY}, Theorem 6.2].
\end{remark}

\subsection{1-leg DT $\mathbb{Z}_2\times\mathbb{Z}_2$-vertex and 1-leg DT $\mathbb{Z}_{4}$-vertex}
When the 1-leg partitions are chosen to be $\nu=(m,m-1,\cdots,2,1)$ for $m\in\mathbb{Z}_{\geq1}$,  we have the following  relations between 1-leg DT $\mathbb{Z}_2\times\mathbb{Z}_2$-vertex and 1-leg DT $\mathbb{Z}_{4}$-vertex.
\begin{theorem}\label{Z2Z2-Z4}
Assume  $\nu=(m,m-1,\cdots,2,1)$ with $m\in\mathbb{Z}_{\geq1}$. Let $q=q_aq_bq_cq_0$. Then we have  \\
(i) if $m\equiv0\; \mbox{(mod 4)}$ or $m\equiv3\; \mbox{(mod 4)}$,
\ben
&&V_{\nu\emptyset\emptyset}^{\mathbb{Z}_{2}\times\mathbb{Z}_{2}}(q_0,q_c,q_a,q_b)=
V^{\mathbb{Z}_4}_{\emptyset\emptyset\nu}(q_0,q_a,q_b,q_c)\cdot\Phi(q_c,q_a,q_b,q;m),\\
&&V_{\emptyset\nu\emptyset}^{\mathbb{Z}_{2}\times\mathbb{Z}_{2}}(q_0,q_b,q_c,q_a)= V^{\mathbb{Z}_4}_{\emptyset\emptyset\nu}(q_0,q_c,q_a,q_b)\cdot\Phi(q_b,q_c,q_a,q;m),\\
&&V_{\emptyset\emptyset\nu}^{\mathbb{Z}_{2}\times\mathbb{Z}_{2}}(q_0,q_a,q_b,q_c)
= V^{\mathbb{Z}_4}_{\emptyset\emptyset\nu}(q_0,q_b,q_c,q_a)
\cdot\Phi(q_a,q_b,q_c,q;m),
\een	
(ii) if $m\equiv1\; \mbox{(mod 4)}$ or $m\equiv2\; \mbox{(mod 4)}$,
\ben
&&V_{\nu\emptyset\emptyset}^{\mathbb{Z}_{2}\times\mathbb{Z}_{2}}(q_0,q_c,q_a,q_b)=V^{\mathbb{Z}_4}_{\emptyset\emptyset\nu}(q_b,q_a,q_0,q_c)\cdot\Phi(q_c,q_a,q_b,q;m),\\
&&V_{\emptyset\nu\emptyset}^{\mathbb{Z}_{2}\times\mathbb{Z}_{2}}(q_0,q_b,q_c,q_a)=V^{\mathbb{Z}_4}_{\emptyset\emptyset\nu}(q_a,q_c,q_0,q_b)\cdot\Phi(q_b,q_c,q_a,q;m),\\
&&V_{\emptyset\emptyset\nu}^{\mathbb{Z}_{2}\times\mathbb{Z}_{2}}(q_0,q_a,q_b,q_c)=V^{\mathbb{Z}_4}_{\emptyset\emptyset\nu}(q_c,q_b,q_0,q_a)\cdot\Phi(q_a,q_b,q_c,q;m),
\een	
where 
\ben
&&\Phi(q_a,q_b,q_c,q;m)\\
&=&\frac{\widehat{\mathbf{M}}(q_a,q)\cdot\widehat{\mathbf{M}}(q_b,q)\cdot\widehat{\mathbf{M}}(q_c,q)\cdot\widehat{\mathbf{M}}(q_aq_bq_c,q)\cdot \widetilde{\mathbf{M}}(q_aq_b,q)\cdot\widetilde{\mathbf{M}}_{2\{\frac{m}{2}\}}(q_aq_b,q;2\lfloor \frac{m+1}{2}\rfloor)}{\widehat{\mathbf{M}}_{2\{\frac{m}{2}\}}(q_a,q;\lfloor \frac{m+1}{2}\rfloor)\cdot\widehat{\mathbf{M}}_{2\{\frac{m}{2}\}}(q_b,q;\lfloor \frac{m+1}{2}\rfloor)\cdot\widehat{\mathbf{M}}_{1-2\{\frac{m}{2}\}}(q_c,q;\lfloor \frac{m+1}{2}\rfloor)\cdot\widehat{\mathbf{M}}_{2\{\frac{m}{2}\}}(q_aq_bq_c,q;\lfloor \frac{m+1}{2}\rfloor)}.
\een
\end{theorem}
\begin{proof}
Notice that $V^{\mathbb{Z}_4}_{\emptyset\emptyset\nu}(q_0,q_b,q_c,q_a)$ is a symmetric function in variables $q_a$ and $q_b$ by Lemma 4.34 and Equation \eqref{Z4-nolegs}. The proof follows from Lemma \ref{RPC+RPC-}, Lemma \ref{Z2Z2-RPC+},  Lemma \ref{Z4-RPC-}, and Remark \ref{Z2Z2-symmetry}.
\end{proof}	
	
\begin{remark}
Since the  DT $\mathbb{Z}_2\times\mathbb{Z}_2$-vertex has more symmetries than DT $\mathbb{Z}_4$-vertex (see Remark \ref{Z2Z2-symmetry} and [\cite{BCY}, Section 3.1]), we relate each 1-leg DT $\mathbb{Z}_2\times\mathbb{Z}_2$-vertex only with the 1-leg DT $\mathbb{Z}_4$-vertex $V^{\mathbb{Z}_4}_{\emptyset\emptyset\nu}(\cdots)$  in Theorem \ref{Z2Z2-Z4}. Actually  Theorem \ref{Z2Z2-Z4}  also holds for  $\nu=\emptyset$ or equivalently $m=0$, i.e.,
\ben
V_{\emptyset\emptyset\emptyset}^{\mathbb{Z}_{2}\times\mathbb{Z}_{2}}(q_0,q_a,q_b,q_c)
= V^{\mathbb{Z}_4}_{\emptyset\emptyset\emptyset}(q_0,q_b,q_c,q_a)\cdot\widehat{\mathbf{M}}(q_a,q)\cdot\widehat{\mathbf{M}}(q_b,q)\cdot\widehat{\mathbf{M}}(q_c,q)\cdot\widehat{\mathbf{M}}(q_aq_bq_c,q)\cdot \widetilde{\mathbf{M}}(q_aq_b,q).
\een	

\end{remark}

\subsection{An explicit formula for a class of 1-leg DT $\mathbb{Z}_2\times\mathbb{Z}_2$-vertex}
To derive the explicit formula for 1-leg DT $\mathbb{Z}_2\times\mathbb{Z}_2$-vertex $V_{\emptyset\emptyset\nu}^{\mathbb{Z}_{2}\times\mathbb{Z}_{2}}$ (resp. $V_{\nu\emptyset\emptyset}^{\mathbb{Z}_{2}\times\mathbb{Z}_{2}}$ and $V_{\emptyset\nu\emptyset}^{\mathbb{Z}_{2}\times\mathbb{Z}_{2}}$) when  $\nu=(m,m-1,\cdots,2,1)$ with $m\in\mathbb{Z}_{\geq1}$, we first need the following
\begin{lemma}\label{Z4-1-leg}
	Assume  $\nu=(m,m-1,\cdots,2,1)$ with $m\in\mathbb{Z}_{\geq1}$. Let $q=q_aq_bq_cq_0$. Then\\
(i) if $m\equiv0\; \mbox{(mod 4)}$ or  $m\equiv3\; \mbox{(mod 4)}$,
\ben
&&V^{\mathbb{Z}_4}_{\emptyset\emptyset\nu}(q_0,q_b,q_c,q_a)\\
&=&V^{\mathbb{Z}_4}_{\emptyset\emptyset\emptyset}(q_0,q_b,q_c,q_a)\cdot\widetilde{\mathbf{M}}_{1-2\{\frac{m}{2}\}}\left(q_c,q;\left\lfloor \frac{m+1}{2}\right\rfloor\right)\\
&&\cdot\widetilde{\mathbf{M}}_{2\{\frac{m}{2}\}}\left(q_a,q;\left\lfloor \frac{m+1}{2}\right\rfloor\right)\cdot\widetilde{\mathbf{M}}_{2\{\frac{m}{2}\}}\left(q_b,q;\left\lfloor \frac{m+1}{2}\right\rfloor\right)\cdot\widetilde{\mathbf{M}}_{2\{\frac{m}{2}\}}\left(q_aq_bq_c,q;\left\lfloor \frac{m+1}{2}\right\rfloor\right),
\een
(ii) if $m\equiv1\; \mbox{(mod 4)}$ or $m\equiv2\; \mbox{(mod 4)}$,	
		\ben
		&&V^{\mathbb{Z}_4}_{\emptyset\emptyset\nu}(q_0,q_b,q_c,q_a)\\
		&=&V_{\emptyset\emptyset\emptyset}^{\mathbb{Z}_4}(q_c,q_b,q_0,q_a)\cdot\widetilde{\mathbf{M}}_{1-2\{\frac{m}{2}\}}\left(q_0,q;\left\lfloor \frac{m+1}{2}\right\rfloor\right)\\
		&&\cdot\widetilde{\mathbf{M}}_{2\{\frac{m}{2}\}}\left(q_a,q;\left\lfloor \frac{m+1}{2}\right\rfloor\right)\cdot\widetilde{\mathbf{M}}_{2\{\frac{m}{2}\}}\left(q_b,q;\left\lfloor \frac{m+1}{2}\right\rfloor\right)\cdot\widetilde{\mathbf{M}}_{2\{\frac{m}{2}\}}\left(q_aq_bq_0,q;\left\lfloor \frac{m+1}{2}\right\rfloor\right).
		\een	
\end{lemma}
\begin{proof}
By Theorem \ref{DT-Z_n}, we have 
\ben
V^{\mathbb{Z}_4}_{\emptyset\emptyset\nu}(q_0,q_b,q_c,q_a)=V_{\emptyset\emptyset\emptyset}^{\mathbb{Z}_4}(q_0,q_b,q_c,q_a)\cdot O_{\nu}\cdot H_{\nu} 
\een
where 
\ben
&&O_\nu=\prod_{k=0}^{3}V_{\emptyset\emptyset\emptyset}^{\mathbb{Z}_4}(\tilde{q}_k,\tilde{q}_{k+1},\tilde{q}_{k+2}, \tilde{q}_{k+3})^{-2|\nu|_{\overline{k}}+|\nu|_{\overline{k+1}}+|\nu|_{\overline{k-1}}},\\
&&H_{\nu}=\prod_{(i,j)\in\nu}\frac{1}{1-\prod\limits_{k=0}^{3}\tilde{q}_k^{h^k_\nu(i,j)}},
\een
with $|\nu|_{\overline{k}}=\left|\{(i,j)\in\nu\,|\, i-j\equiv k\; (\mathrm{mod}\; 4)\}\right|$ and 
\ben
\tilde{q}_k=\left\{
\begin{aligned}
	&q_0 ,\;\;\;\;\;\; \mbox{if $k=0$ $($\mbox{mod} 4$)$}, \\
	&q_b,  \;\;\;\;\;\;\mbox{if $k=1$ $($\mbox{mod} 4$)$},\\
	&q_c , \;\;\;\;\;\; \mbox{if $k=2$ $($\mbox{mod} 4$)$}, \\
	&q_a,   \;\;\;\;\;\;\mbox{if $k=3$ $($\mbox{mod} 4$)$}.
\end{aligned}
\right.
\een

Since $\nu=(m,m-1,\cdots,2,1)$,
combined with Equation \eqref{v-color1},  one can compute $O_\nu$ as follows.\\
If $m\equiv0\; \mbox{(mod 4)}$, then 
\ben
O_{\nu}=\left(\frac{V_{\emptyset\emptyset\emptyset}^{\mathbb{Z}_4}(q_0,q_b,q_c,q_a)\cdot V_{\emptyset\emptyset\emptyset}^{\mathbb{Z}_4}(q_c,q_a,q_0,q_b)}{V_{\emptyset\emptyset\emptyset}^{\mathbb{Z}_4}(q_b,q_c,q_a,q_0)\cdot V_{\emptyset\emptyset\emptyset}^{\mathbb{Z}_4}(q_a,q_0,q_b,q_c)}\right)^\frac{m}{2}.
\een
If $m\equiv1\; \mbox{(mod 4)}$, then 
\ben
O_{\nu}=\left(\frac{V_{\emptyset\emptyset\emptyset}^{\mathbb{Z}_4}(q_b,q_c,q_a,q_0)\cdot V_{\emptyset\emptyset\emptyset}^{\mathbb{Z}_4}(q_a,q_0,q_b,q_c)}{V_{\emptyset\emptyset\emptyset}^{\mathbb{Z}_4}(q_0,q_b,q_c,q_a)\cdot V_{\emptyset\emptyset\emptyset}^{\mathbb{Z}_4}(q_c,q_a,q_0,q_b)}\right)^\frac{m+1}{2}\cdot\frac{V_{\emptyset\emptyset\emptyset}^{\mathbb{Z}_4}(q_c,q_a,q_0,q_b)}{V_{\emptyset\emptyset\emptyset}^{\mathbb{Z}_4}(q_0,q_b,q_c,q_a)}.
\een
If $m\equiv2\; \mbox{(mod 4)}$, then 
\ben
O_{\nu}=\left(\frac{V_{\emptyset\emptyset\emptyset}^{\mathbb{Z}_4}(q_0,q_b,q_c,q_a)\cdot V_{\emptyset\emptyset\emptyset}^{\mathbb{Z}_4}(q_c,q_a,q_0,q_b)}{V_{\emptyset\emptyset\emptyset}^{\mathbb{Z}_4}(q_b,q_c,q_a,q_0)\cdot V_{\emptyset\emptyset\emptyset}^{\mathbb{Z}_4}(q_a,q_0,q_b,q_c)}\right)^\frac{m}{2}\cdot\frac{V_{\emptyset\emptyset\emptyset}^{\mathbb{Z}_4}(q_c,q_a,q_0,q_b)}{V_{\emptyset\emptyset\emptyset}^{\mathbb{Z}_4}(q_0,q_b,q_c,q_a)}.
\een
If $m\equiv3\; \mbox{(mod 4)}$, then 
\ben
O_{\nu}=\left(\frac{V_{\emptyset\emptyset\emptyset}^{\mathbb{Z}_4}(q_b,q_c,q_a,q_0)\cdot V_{\emptyset\emptyset\emptyset}^{\mathbb{Z}_4}(q_a,q_0,q_b,q_c)}{V_{\emptyset\emptyset\emptyset}^{\mathbb{Z}_4}(q_0,q_b,q_c,q_a)\cdot V_{\emptyset\emptyset\emptyset}^{\mathbb{Z}_4}(q_c,q_a,q_0,q_b)}\right)^\frac{m+1}{2}.
\een

As is shown in the proof of [\cite{BCY}, Proposition 9], there exists a bijection between the finite set 
\ben
\mathcal{H}ook_{\nu^\prime}=\left\{ (t_1,t_2)\in\mathbb{Z}^2\,|\, t_1<t_2,\;  \nu^\prime(t_1)=-1, \; \nu^\prime(t_2)=1\right\}
\een
with the set of hooks of $\nu^\prime$, where each pair $(t_1,t_2)$ corresponds to two ends of the legs of a hook. Notice $\nu=\nu^\prime$  in our case. With this bijection and the two latter equivalences in Equations \eqref{odd-symmetry1}, \eqref{even-symmetry1}, one may compute $H_\nu$ by the length of each hook together with the color of its first box as follows.\\
If $m\equiv0\; \mbox{(mod 4)}$, then 
\ben
H_{\nu}=\mathbf{M}_2\left(q_a,q;\frac{m}{2}\right)\cdot\mathbf{M}_2\left(q_b,q;\frac{m}{2}\right)\cdot\mathbf{M}_2\left(q_aq_bq_c,q;\frac{m}{2}\right)\cdot\mathbf{M}_2\left(qq_c^{-1},q;\frac{m}{2}-1\right).
\een
If $m\equiv1\; \mbox{(mod 4)}$, then 
\ben
H_{\nu}=\mathbf{M}_2\left(q_0,q;\frac{m+1}{2}\right)\cdot\mathbf{M}_2\left(qq_a^{-1},q;\frac{m-1}{2}\right)\cdot\mathbf{M}_2\left(qq_b^{-1},q;\frac{m-1}{2}\right)\cdot\mathbf{M}_2\left(q_c,q;\frac{m-1}{2}\right).
\een
If $m\equiv2\; \mbox{(mod 4)}$, then 
\ben
H_{\nu}=\mathbf{M}_2\left(q_a,q;\frac{m}{2}\right)\cdot\mathbf{M}_2\left(q_b,q;\frac{m}{2}\right)\cdot\mathbf{M}_2\left(q_aq_bq_c,q;\frac{m}{2}-1\right)\cdot\mathbf{M}_2\left(qq_c^{-1},q;\frac{m}{2}\right).
\een
If $m\equiv3\; \mbox{(mod 4)}$, then 
\ben
H_{\nu}=\mathbf{M}_2\left(q_0,q;\frac{m-1}{2}\right)\cdot\mathbf{M}_2\left(qq_a^{-1},q;\frac{m-1}{2}\right)\cdot\mathbf{M}_2\left(qq_b^{-1},q;\frac{m-1}{2}\right)\cdot\mathbf{M}_2\left(q_c,q;\frac{m+1}{2}\right).
\een
Since 
\bea\label{Z4-nolegs}
&&V^{\mathbb{Z}_4}_{\emptyset\emptyset\emptyset}(q_0,q_b,q_c,q_a)\\
&=&\mathbf{M}(1,q)^4\cdot \widetilde{\mathbf{M}}(q_a,q)\cdot \widetilde{\mathbf{M}}(q_b,q)\cdot \widetilde{\mathbf{M}}(q_c,q)\cdot \widetilde{\mathbf{M}}(q_aq_c,q)\cdot \widetilde{\mathbf{M}}(q_bq_c,q)\cdot \widetilde{\mathbf{M}}(q_aq_bq_c,q)\nonumber
\eea
and 
\ben
\frac{\widetilde{\mathbf{M}}(x,q)}{\widetilde{\mathbf{M}}(qx^{-1},q)}=\frac{(x;q)_\infty}{(qx^{-1};q)_\infty},
\een
then we have 
\ben
&&\frac{V_{\emptyset\emptyset\emptyset}^{\mathbb{Z}_4}(q_0,q_b,q_c,q_a)\cdot V_{\emptyset\emptyset\emptyset}^{\mathbb{Z}_4}(q_c,q_a,q_0,q_b)}{V_{\emptyset\emptyset\emptyset}^{\mathbb{Z}_4}(q_b,q_c,q_a,q_0)\cdot V_{\emptyset\emptyset\emptyset}^{\mathbb{Z}_4}(q_a,q_0,q_b,q_c)}\\
&=&\frac{(q_a;q)_{\infty}}{(qq_a^{-1};q)_{\infty}}\cdot\frac{(q_b;q)_{\infty}}{(qq_b^{-1};q)_{\infty}}\cdot\frac{(qq_c^{-1};q)_{\infty}}{(q_c;q)_{\infty}}\cdot\frac{(q_aq_bq_c;q)_{\infty}}{(q(q_aq_bq_c)^{-1};q)_{\infty}}.
\een
Then the proof is completed by combining all the above results together with Equation \eqref{identity1}.
\end{proof}

Now we have the following 
\begin{theorem}\label{Z2Z2-1-leg}
Assume  $\nu=(m,m-1,\cdots,2,1)$ with $m\in\mathbb{Z}_{\geq1}$. Let $q=q_aq_bq_cq_0$.  Then we have
	\ben
	&&V_{\nu\emptyset\emptyset}^{\mathbb{Z}_{2}\times\mathbb{Z}_{2}}(q_0,q_c,q_a,q_b)=V_{\emptyset\emptyset\emptyset}^{\mathbb{Z}_{2}\times\mathbb{Z}_{2}}(q_0,q_c,q_a,q_b)\cdot\Upsilon(q_c,q_a,q_b,q;m),\\
	&&V_{\emptyset\nu\emptyset}^{\mathbb{Z}_{2}\times\mathbb{Z}_{2}}(q_0,q_b,q_c,q_a)=V_{\emptyset\emptyset\emptyset}^{\mathbb{Z}_{2}\times\mathbb{Z}_{2}}(q_0,q_b,q_c,q_a)\cdot\Upsilon(q_b,q_c,q_a,q;m),\\ 
	&&V_{\emptyset\emptyset\nu}^{\mathbb{Z}_{2}\times\mathbb{Z}_{2}}(q_0,q_a,q_b,q_c)=V_{\emptyset\emptyset\emptyset}^{\mathbb{Z}_{2}\times\mathbb{Z}_{2}}(q_0,q_a,q_b,q_c)\cdot\Upsilon(q_a,q_b,q_c,q;m),
	\een	
	where 
	\ben
	&&\Upsilon(q_a,q_b,q_c,q;m)\\
	&=&\frac{\widetilde{\mathbf{M}}_{2\{\frac{m}{2}\}}\left(q_aq_b,q;2\lfloor\frac{m+1}{2}\rfloor\right)}{\widetilde{\mathbf{M}}_{2\{\frac{m}{2}\}}(-q_a,q;\lfloor\frac{m+1}{2}\rfloor)\cdot\widetilde{\mathbf{M}}_{2\lbrace\frac{m}{2}\rbrace}(-q_b,q;\lfloor\frac{m+1}{2}\rfloor)\cdot\widetilde{\mathbf{M}}_{1-2\{\frac{m}{2}\}}(-q_c,q;\lfloor\frac{m+1}{2}\rfloor)\cdot\widetilde{\mathbf{M}}_{2\{\frac{m}{2}\}}(-q_aq_bq_c,q;\lfloor\frac{m+1}{2}\rfloor)}.
	\een	
\end{theorem}
\begin{proof}
It follows from Theorem \ref{Z2Z2-nolegs}, Theorem \ref{Z2Z2-Z4}, Lemma \ref{Z4-1-leg}, Equation \eqref{Z4-nolegs} and Remark \ref{Z2Z2-symmetry}.
\end{proof}
Finally, we present a relation  between  $Z_{\mathfrak{RP}_{+}(\nu)}=Z_{\mathfrak{RP}_{-}(\nu)}$  and $Z_\mathfrak{P}$  using Lemma \ref{Z4-RPC-}, Lemma \ref{Z4-1-leg}, Remark \ref{another-proof-GFPP} and Equation \eqref{Z4-nolegs} as follows.
\begin{corollary}\label{GF-RPC-withleg}
Assume  $\nu=(m,m-1,\cdots,2,1)$ with $m\in\mathbb{Z}_{\geq1}$. Let $q=q_aq_bq_cq_0$. Then    \\
(i) if $m\equiv0\; \mbox{(mod 4)}$ or $m\equiv3\; \mbox{(mod 4)}$,
\ben
Z_{\mathfrak{RP}_{+}(\nu)}=Z_{\mathfrak{RP}_{-}(\nu)}
=\frac{Z_\mathfrak{P}\cdot\left(\widetilde{\mathbf{M}}_{1-2\{\frac{m}{2}\}}(-q_c,q;\lfloor\frac{m+1}{2}\rfloor)\right)^{-1}}{\widetilde{\mathbf{M}}_{2\{\frac{m}{2}\}}(-q_a,q;\lfloor\frac{m+1}{2}\rfloor)\cdot\widetilde{\mathbf{M}}_{2\{\frac{m}{2}\}}(-q_b,q;\lfloor\frac{m+1}{2}\rfloor)\cdot\widetilde{\mathbf{M}}_{2\{\frac{m}{2}\}}(-q_aq_bq_c,q;\lfloor\frac{m+1}{2}\rfloor)},
\een
(ii) if $m\equiv1\; \mbox{(mod 4)}$ or $m\equiv2\; \mbox{(mod 4)}$,
\ben
Z_{\mathfrak{RP}_{+}(\nu)}=Z_{\mathfrak{RP}_{-}(\nu)}
=\frac{\left(Z_\mathfrak{P}\big|_{q_0\leftrightarrow q_c}\right)\cdot\left(\widetilde{\mathbf{M}}_{1-2\{\frac{m}{2}\}}(-q_0,q;\lfloor\frac{m+1}{2}\rfloor)\right)^{-1}}{\widetilde{\mathbf{M}}_{2\{\frac{m}{2}\}}(-q_a,q;\lfloor\frac{m+1}{2}\rfloor)\cdot\widetilde{\mathbf{M}}_{2\{\frac{m}{2}\}}(-q_b,q;\lfloor\frac{m+1}{2}\rfloor)\cdot\widetilde{\mathbf{M}}_{2\{\frac{m}{2}\}}(-q_aq_bq_0,q;\lfloor\frac{m+1}{2}\rfloor)},
\een
where 
\ben
Z_\mathfrak{P}=Z_{\mathfrak{RP}_{+}(\emptyset)}=Z_{\mathfrak{RP}_{-}(\emptyset)}=\mathbf{M}(1,q)^4\cdot\frac{
	\widetilde{\mathbf{M}}(q_aq_c,q)\cdot\widetilde{\mathbf{M}}(q_bq_c,q)}{\widetilde{\mathbf{M}}(-q_a,q)\cdot\widetilde{\mathbf{M}}(-q_b,q)\cdot\widetilde{\mathbf{M}}(-q_c,q)\cdot\widetilde{\mathbf{M}}(-q_aq_bq_c,q)}.
\een
is the generating function for pyramid partitions $Z_{\mathrm{pyramid}}$ in [\cite{BY}, Theorem 6.2].
\end{corollary}

\section{Future work}

By investigating a class of restricted pyramid configurations with symmetric interlacing property, we derive an explicit formula for a class of 1-leg DT $\mathbb{Z}_2\times\mathbb{Z}_2$-vertex. Although the derivation of this formula   is quite involved,  it is still far from completing the 1-leg case. On account of Remark \ref{general-Z2Z2-RPC+}, it is  natural to consider restricted pyramid configurations with non-symmetric interlacing property in general. In such a case, we may not connect $V_{\emptyset\emptyset\nu}^{\mathbb{Z}_{2}\times\mathbb{Z}_{2}}$ with $V_{\emptyset\emptyset\nu}^{\mathbb{Z}_{4}}$ through  the relation between  $Z_{\mathfrak{RP}_{+}(\nu)}$ and  $Z_{\mathfrak{RP}_{-}(\nu)}$, see Remark \ref{RPC+---RPC-}. 

However, we can also further study $Z_{\mathfrak{RP}_{+}(\nu)}$  to figure out its corresponding generating function of some new pyramid configurations along the diagonal slices  with  legs as in 3D partitions, and then connect it to the DT $\mathbb{Z}_4$-vertex. Some  observations suggest that we may establish the connection between  1-leg DT $\mathbb{Z}_2\times\mathbb{Z}_2$-vertex with the leg partition $\nu$ and  2-leg DT $\mathbb{Z}_4$-vertex for some partitions $\nu\neq(m,m-1,\cdots,2,1)$ with $m\in\mathbb{Z}_{\geq1}$. We will study this interesting phenomenon in the future to extract more information for the 1-leg DT $\mathbb{Z}_2\times\mathbb{Z}_2$-vertex.
Once the 1-leg DT $\mathbb{Z}_2\times\mathbb{Z}_2$-vertex is determined, one may generalize the graphical condensation recurrences  in [\cite{JWY}] to the DT $\mathbb{Z}_2\times\mathbb{Z}_2$-vertex for deriving  the explicit formula of the  full 3-leg  DT $\mathbb{Z}_2\times\mathbb{Z}_2$-vertex.  It is also interesting to investigate whether the explicit formula for  1-leg (or full 3-leg) DT $\mathbb{Z}_2\times\mathbb{Z}_2$-vertex can be completely determined by the DT $\mathbb{Z}_4$-vertex of Bryan-Cadman-Young.


\begin{thebibliography}{999}


\bibitem{AKMV} M. Aganagic, A. Klemm, M. Mari$\tilde{n}$o, C. Vafa, {\em The topological vertex},  Comm. Math.
Phys.  254 (2005) 425--478.



\bibitem{BCY} J. Bryan, C. Cadman and B. Young, {\em The orbifold topological vertex},  Adv. Math. 229 (2012), no. 1, 531--595.

\bibitem{BY}  J. Bryan, B. Young, Generating functions for colored 3D Young diagrams and the Donaldson-Thomas invariants of orbifolds, Duke Math. J. 152 (1) (2010) 115--153.




\bibitem{JWY} H. Jenne, G. Webb and B. Young, {\em Double-Dimer condensation and the PT-DT correspondence},  arXiv:2109.11773

\bibitem{JS} D. Joyce and Y. Song, {\em A theory of generalized Donaldson–Thomas invariants},  Mem. Amer. Math. Soc. 217 (2012), no. 1020.

\bibitem{KS} M. Kontsevich, Y. Soibelman, {\em Stability structures, motivic Donaldson–Thomas invariants and cluster transformations},  arXiv:0811.2435.





\bibitem{LLLZ} Jun Li, Chiu-Chu Melissa Liu, Kefeng Liu, and Jian Zhou,
{\em A mathematical theory of the topological vertex}, Geom. Topol. 13 (2009), no. 1, 527--621.

\bibitem{LLZ2} Chiu-Chu Melissa Liu, Kefeng Liu, Jian Zhou, {\em A formula of two-partition Hodge integrals},  J. Amer. Math. Soc.
20 (1) (2007) 149--184.

\bibitem{MNOP} D. Maulik, N. Nekrasov, A. Okounkov and R. Pandharipande, {\em Gromov-Witten theory and Donaldson-Thomas theory, I},  Compos. Math. 142 (2006), no. 5, 1263--1285. 

\bibitem{MOOP} D. Maulik, A. Oblomkov, A. Okounkov, R. Pandharipande, {\em Gromov-Witten/Donaldson-Thomas correspondence for toric 3-folds},  Invent. Math. 186, (2011), no. 2, 435--479.


\bibitem{MR} S. Mozgovoy and M. Reineke, {\em On the noncommutative Donaldson-Thomas invariants arising from brane tilings}, Adv. Math. 223 (2010), no. 5, 1521--1544.







\bibitem{NN} K. Nagao and H. Nakajima, {\em Counting invariant of perverse coherent sheaves and its wall-crossing},  Internat. Math. Res. Notices 17 (2011), 3885--3938.

\bibitem{N1} K. Nagao, {\em Refined open noncommutative Donaldson-Thomas invariants for small crepant resolutions},  Pacific J. Math. 254 (2011), no. 1, 173--209.

\bibitem{N2} K. Nagao, {\em Non-commutative Donaldson-Thomas theory and vertex operators},  Geom. Topol. 15 (2011), no. 3, 1509--1543.


\bibitem{N3} K. Nagao, {\em Derived categories of small toric Calabi-Yau 3-folds and curve counting invariants}, Q. J. Math. 63 (2012), no. 4, 965--1007.



\bibitem{Okounkov} A. Okounkov, {\em Infinite wedge and random partitions}, Selecta Math. (N.S.) 7 (2001), 57--81.

\bibitem{ORV} A. Okounkov, N. Reshetikhin, C. Vafa, {\em Quantum Calabi–Yau and classical crystals}, 
from: “The unity of mathematics”, (P Etingof, V Retakh, IM Singer, editors), Progr. Math. 244, Birkhäuser, Boston, MA (2006) 597--618.


\bibitem{Ross} D. Ross, Localization and gluing of orbifold amplitudes: the Gromov–Witten orbifold vertex. Trans. Am. Math. Soc. 366(3), 1587--1620 (2014).




\bibitem{Szendroi} B. Szendroi, {\em Non-commutative Donaldson-Thomas theory and the conifold}, Geom. Topol. 12 (2008), 1171--1202.


\bibitem{You} B. J. Young, {\em Computing a pyramid partition generating function with dimer shuffling}, J. Combin. Theory Ser. A 116 (2009), 334--350.

\bibitem{Zhou} Jian Zhou, {\em Crepant resolutions, quivers and GW/NCDT duality},  arXiv:0907.0135.




\end{thebibliography}
\end{document}